\documentclass{amsart}

\usepackage[dvips]{epsfig}
\usepackage{graphicx}
\usepackage{latexsym}
\usepackage{amsmath}
\usepackage{amsthm}
\usepackage{amssymb}
\usepackage{ulem}

\usepackage{amsmath,amsthm,amsfonts,amssymb,amscd,amsbsy,dsfont,hyperref}

\usepackage{palatino}

\usepackage{eucal}
\usepackage{float}
\usepackage{xcolor}
\usepackage{multirow} % AO, 2 septembre = j'ajoute un package, attention !

\usepackage[applemac]{inputenc}

\newtheorem{thm}{Theorem}
\newtheorem{assumption}[thm]{Assumption}
\newtheorem{lem}[thm]{Lemma}

\newtheorem{cor}[thm]{Corollary}
\newtheorem{defi}[thm]{Definition}

\newtheorem{prop}[thm]{Proposition}
\newtheorem{rk}[thm]{Remark}

\newcommand{\field}[1]{\mathbb{#1}}

\newcommand{\E}{\field{E}}

\newcommand{\PP}{\field{P}}

\newcommand{\R}{\field{R}}

\newcommand{\vip}{\vskip.2cm}

\newcommand{\Ff}{{\mathcal F}}

\allowdisplaybreaks

\begin{document}

\date{\today}

\title[Nonparametric estimation for interacting particle systems]{Nonparametric estimation for interacting particle systems: McKean-Vlasov models}

\author{Laetitia Della Maestra and Marc Hoffmann}

\address{Laetitia Della Maestra, Universit\'e Paris-Dauphine \& PSL, CNRS, CEREMADE, 75016 Paris, France}
\email{dellamaestra@ceremade.dauphine.fr}
\address{Marc Hoffmann, Universit\'e Paris-Dauphine \& PSL, CNRS, CEREMADE, 75016 Paris, France}
\email{hoffmann@ceremade.dauphine.fr}

\begin{abstract} 
We consider a system of $N$ interacting particles, governed by transport and diffusion, that converges in a mean-field limit to the solution of a McKean-Vlasov equation. From the observation of a trajectory of the system over a fixed time horizon, we investigate nonparametric estimation of the solution of the associated nonlinear Fokker-Planck equation, together with the drift term that  controls the interactions, in a large population limit $N \rightarrow \infty$. 
We build data-driven kernel estimators and establish oracle inequalities, following Lepski's principle. Our results are based on a new Bernstein concentration inequality in McKean-Vlasov models for the empirical measure around its mean, possibly of independent interest. We obtain adaptive estimators over anisotropic H\"older smoothness classes built upon the solution map of the Fokker-Planck equation, and prove their optimality in a minimax sense. In the specific case of the Vlasov model, we derive an estimator of the interaction potential and establish its consistency.  
\end{abstract}

\maketitle


\tableofcontents

\section{Introduction}
\subsection{Setting}
We continuously observe a stochastic system of $N$ interacting particles 
\begin{equation} \label{eq: data}
X_t = (X_t^1,\ldots, X_t^N),\;\;t \in [0,T]
\end{equation}
evolving in an Euclidean ambient space $\R^d$, that solves
\begin{equation} \label{eq:diff basique}
\left\{
\begin{array}{l}
dX_t^i = b(t,X_t^i,\mu^N_t)dt+\sigma(t,X_t^i) dB_t^i,\;\;1 \leq i \leq N,\; t \in [0,T],\\  \\
\mathcal L(X_0^1,\ldots, X_0^N) = \mu_0^{\otimes N},
\end{array}
\right.
\end{equation}
where 
$\mu^N_t= N^{-1} \sum_{i = 1}^N \delta_{X_t^i}$ is the empirical measure of the system.
The $B^i$
are independent $\R^d$-valued Brownian motions, and the transport and diffusion coefficients 
$b$ and $\sigma$ are sufficiently regular 
so that 
$\mu^N_t \rightarrow \mu_t$ weakly as $N \rightarrow \infty$,
where $\mu_t$ is a weak solution of the parabolic nonlinear equation
\begin{equation}  \label{eq: mckv approx}
\left\{ 
\begin{array}{ll}
\partial_{t}\mu_{t} + \mathrm{div}( b(t,\cdot,\mu_{t}) \mu_{t}) = \tfrac{1}{2}\sum_{k,k'=1}^d\partial_{kk'}^2((\sigma \sigma^\top)(t,\cdot)_{kk'}\mu_{t}),\\ 
 \mu_{t=0} = \mu_0,
\end{array}
\right. 
\end{equation} 
see Section \ref{sec: assumptions} below.  In this context, we are interested in estimating nonparametrically from data \eqref{eq: data} the solution $(t,x) \mapsto \mu_t(x)$ of \eqref{eq: mckv approx} and the drift function $(t,x,\mu) \mapsto b(t,x,\mu) \in \R^d$ at the value $(t,x,\mu) = (t,x,\mu_t)$. The time horizon $T$ is fixed and asymptotics are taken as $N \rightarrow \infty$.\\

A particular case of interest is a homogeneous drift with a linear dependence in the measure argument. The drift term in \eqref{eq:diff basique} then takes the form
\begin{equation} \label{eq: explain drift}
b(t,X_t^i,\mu^N_t) = \int_{\R^d}\widetilde b(X_t^i,y)\mu_t^N(dy) = N^{-1}\sum_{j = 1}^N \widetilde b(X_t^i,X_t^j),
\end{equation}
%$$b: \R^d \times \R^d \rightarrow \R^d$$
for some function $\widetilde b: \R^{d} \times \R^d \rightarrow \R^d$. In the paper, when we specialise on this case, we consider $\widetilde b$ of the form
$$\widetilde b(x,y) = F(x-y)+G(x)$$
for some regular $F, G: \R^d \rightarrow \R^d$. In this case, we have
$$b(t,X_t^i,\mu^N_t)  = F\star \mu_t^N(X_t^i)+G(X_t^i),$$
where $\star$ denotes convolution. The function $F$ plays the role of an interaction force applied to the particle system while $G$ accounts for an external force in the motion of each particle. If the forces $F$ and $G$ derive from smooth potentials $V,W:\R^d \rightarrow \R$, the interaction then takes the form $F(x)=-\nabla W(x)$ and we have a confinement $G(x) = -\nabla V(x)$ , see {\it e.g.} \cite{benachour1998nonlinear, herrmann2008large, CATTIAUXetal}. However, we work on a fixed time horizon $[0,T]$ in the paper, and will not need this point of view\footnote{usually required to control the model for convergence to equilibrium when $T$ is large.}.\\

In the semi-linear representation \eqref{eq: explain drift}, we are interested in estimating nonparametrically from data \eqref{eq: data} the interaction force $x \mapsto F(x)$; the parameter $x \mapsto G(x)$ is considered as a nuisance.

\subsection{Motivation}
Stochastic systems of interacting particles and associated nonlinear Markov processes in the sense of McKean \cite{MCKEAN} date back to the 1960's and originated from statistical physics in plasma physics. Their importance in probability theory progressively grew in the following decades, and a versatility of fundamental tools were developed in this context like {\it e.g.} coupling methods, geometric inequalities, propagation of chaos, concentration and fluctuations in abstract functional spaces, see Sznitman \cite{SZNITMAN, Sznitman2}, Tanaka and Hitsuda \cite{TANAKAHITSUDA}, Fernandez {\it et al.} \cite{BEGOMELEARD}, M\'el\'eard \cite{meleard1996asymptotic}, Malrieu \cite{MALRIEU}, Cattiaux {\it et al.} \cite{CATTIAUXetal}, Bolley {\it et al.} \cite{bolley2007quantitative}, among a myriad of references. However, until the early 2000's, a modern formulation of a statistical inference program in this context was out of reach (with some notable exceptions like {\it e.g.} Kasonga \cite{kasonga1990maximum}), at least for two reasons: first, the fine probabilistic tools required for nonparametric adaptive estimation were still in full development; second and perhaps more importantly, microscopic particles systems issued from statistical physics are not naturally observable and the motivation for statistical inference is not obvious in this context. The situation progressively evolved around the 2010's, with the start of a kind of renaissance of McKean-Vlasov type models in several application fields that model collective and observable dynamics, ranging from mathematical biology (neurosciences, Baladron {\it et al.} \cite{Baladron}, structured models in population dynamics, Mogilner {\it et al.} \cite{mogilner99}, Burger {\it et al.} \cite{BURGER}) to social sciences (opinion dynamics, Chazelle {\it et al.} \cite{CHAZELLE}, cooperative behaviours, Canuto {\it et al.} \cite{canuto12}) and finance (systemic risk, Fouque and Sun \cite{fouque2013systemic}). More recently, mean-field games (Cardaliguet {\it et al.} \cite{CARDA}, Cardaliaguet and Lehalle \cite{LC19}) appear as a new frontier for statistical developments, see in particular the recent contribution of Giesecke \cite{giesecke2020inference}. The field has reached enough maturity for the necessity and interest of a systematic statistical inference program, starting with nonparametric estimation. This is the topic of the paper.\\

Parallel to understanding collective dynamics from a statistical point of view, some interest in the study of statistical models related to PDE's has progressively emerged over the last decade. Typical examples include nonparametric Bayes and uncertainty quantification for inverse problems, see Abraham and Nickl \cite{abraham2019statistical}, Monard {\it et al.} \cite{monard2019consistent}, Nickl \cite{nickl2017bayesian, nickl2017bernstein},  and the references therein, or inference in structured models from microscopic data (Doumic {\it et al.} \cite{doumic2015statistical, doumic2012nonparametric}, Hoffmann and Olivier \cite{HO}, Boumezoued {\it et al.} \cite{BHJ}, Ngoc {\it et al.} \cite{ngoc2019nonparametric}, Ma\"\i da {\it et al.} \cite{maida2020statistical}). The analysis of elliptic, parabolic or transport-fragmentation equations sheds new light on the underlying nonparametric structure of companion statistical experiments and enrich the classical theory. To that extent, we provide in this paper a first step in that direction for a certain kind of nonlinear parabolic equations, in the sense of McKean \cite{MCKEAN}. Finally, our work can also be embedded in the framework of functional data analysis, where we observe $N$ diffusion processes with common dynamics, see {\it e.g.} the recent work of Comte and Genon-Catalot \cite{comte2019nonparametric} formally contained in our framework, for $d=1$ and in absence of interactions in the drift $b$.

\subsection{Results and organisation of the paper}

In Section \ref{sec: assumptions}, we detail the notations and assumptions of the model and build kernel estimators for $\mu_t(x)$ and $b(t,x,\mu_t)$. Whereas the estimation for $\mu_t$ is standard, the estimation of the drift requires a smoothing in both time and space of the empirical measure $\pi^N(dt,dx) = N^{-1}\sum_{i = 1}^N\delta_{X_t^i}(dx)X^i(dt)$, that estimates the intermediate function $\pi(t,x) = b(t,x,\mu_t)\mu_t(x)$. We then use a quotient estimator to recover $b$.\\ 

In Section \ref{sec: esti mu et b}, we adopt the Goldenshluger-Lepski method \cite{GL08,GL11,GL14} to tune the bandwidths of both estimators in a data driven way and obtain oracle inequalities in Theorems \ref{thm: GL mu} and \ref{thm: oracle b}
for both $\mu$ and $b$. We further develop a minimax theory in Section \ref{sec: adaptive estimation} when $\mu$ and $b$ belong to anisotropic H\"older spaces in time and space variable, that are built upon  the solution of the parabolic nonlinear limiting equation \eqref{eq: mckv approx} and prove the optimality and smoothness adaptivity or our estimators in Theorems \ref{thm minimax mu} and \ref{thm minimax b}. We finally explore in Section \ref{sec: identif interaction potential} the identification of the interaction force $F$ in the Vlasov model where the drift takes the form $b(x,\mu) = F\star(x) \mu+G(x)$, for some sufficiently well localised functions $F$ and $G$. We prove in Theorem \ref{thm: identif Vlasov} that one can consistently estimate $F$ (hence $G$) by means of a Fourier type estimator, inspired by blind deconvolution methods, see {\it e.g.} Johannes \cite{JJ09}.\\

We develop the probabilistic tools we need to undertake our statistical estimates in Section \ref{sec: proof concentration}. We study the fluctuations of $\mu_t^N$ around its mean $\mu_t$ in Theorem \ref{thm: concentration}, with time dependent extension to the fluctuations of $\mu_t^N(dx)\rho(dt)$ around $\nu(dt,dx)=\mu_t(dx)\rho(dt)$ for arbitrary weight measures $\rho(dt)$.  We prove a Bernstein concentration inequality that reads 
\begin{align*}
\mathrm{Prob}\Big(N^{-1}\sum_{i = 1}^N\int_0^T\phi(t,X_t^i)\rho(dt)-\int_{[0,T] \times \R^d}&\phi(t,y)\mu_t(dy)\rho(dt) \geq x\Big) \\
& \leq \kappa_1 \exp\Big(-\frac{ \kappa_2 Nx^2}{|\phi|_{L^2(\nu)}^2+|\phi|_\infty x}\Big),\;\;\forall x \geq 0,
\end{align*}
over test functions $\phi$ and for some $\kappa_i= \kappa_i(T,\sigma,b,\mu_0) >0$.  It improves on variance estimates based on coupling or geometric inequalities that usually need $\phi$ to be $1$-Lipschitz, whereas nonparametric statistical estimation requires $\phi = \phi_N$ to mimic a Dirac mass as $N \rightarrow \infty$ that can be controlled in $L^2$-norm but behaves badly in Lipschitz norm. Bernstein's inequality for a range of deviation valid for all $x \geq 0$ is also the gateway to nonparametric adaptive estimation; it is not provided by concentration in Wasserstein distance like in Bolley {\it et al.} \cite{bolley2007quantitative} that moreover has the drawback of adding an additional unavoidable dimensional penalty in the rates of convergence.\\

Our method of proof avoids coupling by relying on a Girsanov argument, following classical ideas, recently revisited for instance by Lacker \cite{Lacker}, a key reference for our work.  
We classically require strong ellipticity for the diffusion coefficient and Lipschitz continuity in the space variable. As for the drift, we assume at least Lipschitz continuity 
$$|b(t,x,\mu)-b(t,x',\mu')| \leq C\big(|x-x'|+\mathcal W_1(\mu,\mu')\big)$$
in Wasserstein-1 metric $\mathcal W_1$. In our approximation argument, the logarithm of the Girsanov density between the law of the data and a companion coupled system of independent particles is of order  $N\int_0^T \int_{\R^d}\big| b(s,x,\mu^N_s)- b(s,x,\mu_s)\big|^2\mu_s^N(dx)ds \lesssim N \sup_{0 \leq t \leq T}\mathcal W_1(\mu_t^N,\mu_t)^2$ for which we need  sharp integrability properties uniformly in $N$. In order to circumvent the unavoidable dimensional effect of the approximation $\mathcal W_1(\mu_t^N,\mu_t) \approx N^{-1/\max(2,d)}$, see Fournier and Guillin \cite{FournierGuillin}, we assume moreover $k$-linear differentiability for $b$ in the measure variable, see Assumption \ref{ass: basic lip} in Section \ref{sec: assumptions}. This enables us to control by a sub-Gaussianity argument for $U$-statistics each term of a Taylor-like expansion of $b(t,x,\mu_t^N)-b(t,x,\mu_t)$ in the measure variable; we obtain the desired result, provided $k \geq d/2$ (with a special modification in dimension $d=2$). In particular, in the Vlasov model, we formally have $k=\infty$ and the result is valid in all dimension $d \geq 1$.\\

Section \ref{sec: proof concentration} and \ref{sec: proof nonpara} are devoted to the proofs and an Appendix (Section \ref{sec: appendix}) contains auxiliary technical results.
 
\section{Model assumptions and construction of estimators} \label{sec: assumptions}
\subsection{The system of interacting particles and its limit}
%\subsubsection*{Notation} 
We fix an integer $d \geq 1$ and a time horizon $T>0$. The random processes take their values in $\R^d$.
% equipped with the Euclidean distance and its Borel sigma-field. 
%We write $\mathcal{P}_1(\mathbb{R}^{d})$ for the set of probability measures on $\mathbb{R}^{d}$ with a first moment, endowed with the Wasserstein $1$-metric 
%$$\mathcal W_1(\mu,\nu) = \inf_{m \in \Gamma(\mu,\nu)} \int_{\R^d \times \R^d} \big|x-y\big|m(dx,dy)$$\\
%Mwhere $\Gamma(\mu, \nu)$ denotes the set of probability measures on the product space $\R^d \times \R^d$ with first marginal $\mu$ and second marginal $\nu$. 
We write $|\cdot |$ for the Euclidean distance (on $\R$ or $\R^d$) or sometimes for the modulus of a complex number, and $x^\top x = |x|^2$.

\subsubsection*{Functions}
We consider functions that are mappings defined on products of metric spaces (typically $[0,T] \times \R^d \times \mathcal P_1$ or subsets of these) with values in $\R$ or $\R^d$. Here, $\mathcal P_1$ denotes the set of probability measures on $\mathbb{R}^{d}$ with a first moment, endowed with the Wasserstein $1$-metric
$$\mathcal W_1(\mu,\nu) = \inf_{m \in \Gamma(\mu,\nu)} \int_{\R^d \times \R^d} \big|x-y\big|m(dx,dy) = \sup_{|\phi|_{\mathrm{Lip}}\leq 1}\int_{\R^d}\phi\, d\big(\mu-\nu\big),$$
where $\Gamma(\mu, \nu)$ denotes the set of probability measures on  $\R^d \times \R^d$ with marginals $\mu$ and $\nu$. 
All the functions in the paper are implicitly measurable with respect to the Borel-sigma field induced by the product topology.  
%We use the same notation when no confusion may arise:
%\begin{itemize}
%\item $\langle f,\mu \rangle = \int_{\R^d}f(x)\mu(dx) \in \R^d$ for a real or $\R^d$-valued function $f$ and $\mu \in \mathcal P(\R^d)$ whenever the integral is meaningful. When $f$ is real-valued, we also set $|f|_{L^2(\mu)} = \langle f^2,\mu\rangle^{1/2}$.
%\item $\big(\langle M_\cdot\rangle_t\big)_{0 \leq t \leq T}$ for the predictable  compensator of a $\R$ or $\R^d$-valued continuous semimartingale $(M_t)_{0 \leq t \leq T}$,
%\item $\langle x,y\rangle = x^\top y$ for the Euclidean product between $x$ and $y$ in $\R^d$ and $\big|x\big|^2 = \langle x,x\rangle = x^\top x$. We  also denote by $|a|$ the absolute value of a real number $a$.
%\end{itemize}
A $\R^d$-valued function $f$ is written componentwise as $f = (f^k)_{1 \leq k \leq d}$ where the $f^k$ are real-valued. 
The product $f \otimes g$ of two $\R^d$-valued functions is the $\R^d$-valued function with $\R^{2d}$ variables with components $(f\otimes g)^k(x,y)= f^k(x)g^k(y)$. 
If $f: [0,T] \times \R^d \rightarrow \R$, we set 
$|f|_\infty =\sup_{t,x}|f(t,x)|$ and $|f|_p = (\int_{[0,T]\times \R^d}|f(t,x)|^pdx\,dt\big)^{1/p}$ for $1 \leq p < \infty$. Depending on the context, if $f:[0,T]\rightarrow \R$ or $f:\R^d \rightarrow \R$  is a function of time or space only,
we sometimes write $|f|_p$ for $(\int_0^T |f(t)|^pdt)^{1/p}$ or $(\int_{\R^d} |f(x)|^pdx)^{1/p}$ when no confusion is possible.
\subsubsection*{Constants}
We repeatedly use positive quantities $\kappa_i,\varpi_i,C_i, i=1,2,\ldots$ that do not depend on $N$, that we call constants, but that actually may (continuously) depend on model parameters.  In most cases, they are explicitly computable. We also use special letters like $\kappa, \delta, \tau, c_\pm, \ldots$, but they will only appear once. The generic notation $C$ is sometimes used before it is set depending on a model parameter. 
The notation $\varpi_i$ stands for quantities that need to be tuned in an algorithm (like an estimator).

%the supremum of $|f(z)|$ over all values $z$ of a real-valued function $z \mapsto f(z)$. 
%
%If $g:\R^d \rightarrow \R$, we define $g^{\otimes d}:\R^d \rightarrow \R^d$ by $g^{\otimes d} = (g,\ldots,g)$ ($k$-times). 

%%Let $T>0$ denote a time horizon fixed throughout the paper. 
%\subsubsection*{Setting} 
%\subsubsection*{The homogeneous semilinear case} \subsection{Assumptions on the model}
\subsubsection*{Assumptions} We work under strong ellipticity and Lipschitz smoothness assumptions on the diffusion matrix 
$\sigma: [0,T] \times \R^d \rightarrow \R^d \otimes \R^d$ and the drift $b: [0,T] \times \R^d \times \mathcal P_1 \rightarrow \R^d$, as well as strong integrability properties for the initial condition $\mu_0$. 

\begin{assumption} \label{ass: init condition}
For some $\gamma_0>0,\gamma_1 \geq 1$, the initial condition $\mu_0$ satisfies 
\begin{equation} \label{eq: def subgaussian}
\int_{\R^d}\exp(\gamma_0 |x|^2)\mu_0(dx) \leq \gamma_1.
\end{equation}
\end{assumption}
%The constant $2$ in \eqref{eq: def subgaussian} is arbitrary: any constant bigger than 1 would do. 

\begin{assumption} \label{ass: minimal prop sigma} The diffusion matrix $\sigma$ is measurable and for some $C\geq 0$, we have
$$|\sigma(t,x')-\sigma(t,x)| \leq C|x'-x|.$$
Moreover, $c = \sigma \sigma^\top$ is such that $\sigma_-^2 |y|^2 \leq (c(t,x)y)^\top y  \leq \sigma_+^2|y|^2$ for some $\sigma_\pm >0$.
% and every $t,x,y$, we have
%$$\sigma_-^2 |y|^2 \leq (c(t,x)y)^\top y  < \sigma_+^2|y|^2.$$  
\end{assumption}

As for the regularity of the drift 
$$b: [0,T] \times \R^d \times \mathcal P_1 \rightarrow \R^d,$$
the notion of linear differentiability, commonly used in the literature of mean-field games and McKean-Vlasov equations in order to quantify the smoothness of $\mu \mapsto b(t,x,\mu)$ as a mapping $\mathcal P_1\rightarrow \R^d$ will be the most useful in our setting. We refer in particular to the illuminating section 2.2. in Jourdain and Tse \cite{jourdain2020central} and the references therein.
\begin{defi} \label{def: linear diff}
A mapping $f:\mathcal P_1\rightarrow \R^d$ is said to have a linear functional derivative, if there exists $\delta_\mu f:\R^d\times \mathcal P_1\rightarrow \R^d$ (sometimes denoted by $\frac{\delta f}{\delta \mu}$) such that 
\begin{equation} \label{eq: def derivee mesure}
f(\mu')-f(\mu) = \int_0^1\int_{\R^d} \delta_\mu f(y,(1-\vartheta) \mu+\vartheta\mu')(\mu'-\mu)(dy)d\vartheta
\end{equation}
with the following smoothness properties
%\begin{equation} 
%\left\{
%\begin{array}{l}
\begin{align*} \label{smoothness derivee mesure}
&\displaystyle \big|\delta_\mu f(y',\mu')- \delta_\mu f(y,\mu)\big|  \leq C\big(\mathcal W_1(\mu',\mu)+|y'-y|\big),  \\  
&\displaystyle \,\big|\partial_y \big(\delta_\mu f(y,\mu')- \delta_\mu f(y,\mu)\big)\big|  \leq C\mathcal W_1(\mu',\mu)
%\end{array}
\end{align*}
for some $C \geq 0$.
%\right.
%\end{equation}
%$$\partial_x \tfrac{\delta f}{\delta m}(x',\mu')-$$
\end{defi}
We can iterate the process described in \eqref{eq: def derivee mesure} and obtain a notion of $k$-linear functional derivative via the existence of  
mappings 
$$\delta_\mu^\ell f: (\R^d)^\ell \times \mathcal P_1 \rightarrow \R^d\;\;\text{for}\;\;\ell = 1,\ldots, k$$ 
%$$\frac{\delta^{k-1} f}{\delta m^{k-1}}(x^{\otimes(k-1)},\mu')-\frac{\delta^{k-1} f}{\delta m^{k-1}}(x^{\otimes(k-1)},\mu) = \int_0^1\int_{\R^d} \frac{\delta^{k} f}{\delta m^{k}}(x^{\otimes k},\vartheta \mu+(1-\vartheta)\mu')(\mu'-\mu)(dx_k)d\vartheta$$
defined recursively by $\delta_\mu^\ell f = \delta_\mu \circ \delta_\mu^{\ell-1}f$ and 
enjoying associated smoothness properties. 
%Examples of linearly differentiable mappings are developed for instance in Chaudru de Raynal and Frikha \cite{CF19} and Coghi {\it et al.} \cite{CDFM19}.
%Typical examples include mappings of the form 
%\begin{equation} \label{eq: rep interactions}
%f(\mu) = h\Big(\int_{\R^d} g_1 d\mu, \int_{(\R^d)^2} g_2 d\mu^{\otimes 2},\ldots, \int_{(\R^d)^M} g_M d\mu^{\otimes M}\Big)
%\end{equation} 
%for some integer $M$, where $h: (\R^d)^M \rightarrow \R^d$ and $g_\ell : (\R^d)^\ell \rightarrow \R^d$ {\color{blue}, $\ell = 1,\ldots, M$,} are smooth functions {\color{green}(?)}.
\begin{assumption} \label{ass: basic lip}
The drift $b: [0,T] \times \R^d \times \mathcal P_1 \rightarrow \R^d$ is measurable and 
$$b_0 = \sup_{t \in [0,T]}|b(t,0,\delta_0)| < \infty.$$
%has representation \eqref{eq: explain drift} and
Moreover, one of the following three conditions is satisfied for some $C > 0$:
\begin{enumerate}
\item[(i)] (Lipschitz continuity.) We have $d=1$ and 
%\begin{equation} \label{eq: def lip gen}
$$
\big|b(t,x',\mu')-b(t,x,\mu)\big| \leq C\big(|x'-x|+\mathcal W_1(\mu',\mu)\big).
$$
\item[(ii)] (Existence of a functional derivative of order $k$.) Let $k \geq 1$. For ($d=1$ and $k \geq 1$) or ($d=2$ and $k \geq 2$) or ($d \geq 3$ and $k\geq d/2$), we have $(i)$ and the map  
$$\mu \mapsto b(t,x,\mu)$$
admits a functional derivative of order $k$ in the sense of Definition \ref{def: linear diff}. 
Moreover, the following representation holds
\begin{equation} \label{eq: rep remainder b}
\delta^k_\mu b(t,x,(y_1,\ldots, y_k),\mu) = \sum_{\mathcal I \subset \{1,\ldots, k\}, m \geq 1} \bigotimes_{j \in\mathcal I} (\delta_\mu^k b)_{\mathcal I, j, m}(t,x,y_j,\mu),
\end{equation}
where the sum in $m$ is finite with at most $m_b$ terms and the
mappings $(\delta^k_\mu b)_{\mathcal I, j,m} : [0,T]\times \R^d \times \R^d \times \mathcal P_1 \rightarrow \R^d$ are such that
$$|(\delta^k_\mu b)_{\mathcal I, j,m}(t,x',y',\mu)-(\delta^k_\mu b)_{\mathcal I, j,m}(t,x,y,\mu)| \leq C(|x'-x|+|y'-y|).$$
\item[(iii)] (Vlasov case.)
We have $d \geq 1$ and
$$b(t,x,\mu) = \int_{\R^d}\widetilde{b}(t,x,y)\mu(dy)$$
for some measurable $\widetilde b: [0,T] \times \R^d \times \R^d \rightarrow \R^d$ such that:
$$
\big|\widetilde b(t,x',y')-\widetilde b(t,x,y)\big| \leq C(|x'-x|+|y'-y|).
$$
\end{enumerate}
\end{assumption}
We let $|b|_{\mathrm{Lip}}$ denote the smallest $C\geq 0$ for which Assumption \ref{ass: basic lip} (i) holds and $|\delta_\mu^k b|_{\mathrm{Lip}}$ the smallest constant $C \geq 0$ for which  Assumption \ref{ass: basic lip} (ii) holds for the highest order of differentiability.

\begin{rk} {\bf 1)} The representation \eqref{eq: rep remainder b} in Assumption  \ref{ass: basic lip} (ii) is merely technical and enables one to obtain a control in $\mathcal W_1$-distance of the remainder term in Taylor-like expansions of $b(t,x,\mu)$ in an easy way, see in particular the proof of Proposition \ref{prop: controle change of proba}, Step 2 below. It can presumably be relaxed, but will be sufficient for the level of generality intended in the paper. It accomodates in particular drifts of the form
$$b(t,x,\mu) = \sum_j F_j\Big(t,x,\int_{\R^{q_2}} G_j\big(t,x,\int_{\R^d} H_j(t,x,z)\mu(dz),z'\big)\lambda_j(dz')\Big)$$ for smooth mappings $F_j(t,x,\cdot): \R^{q_3}\rightarrow \R^d,G_j(t,x,\cdot) : \R^{q_1} \times \R^{q_2} \rightarrow \R^{q_3}, H_j(t,x,\cdot):\R^d \rightarrow \R^{q_1}$ and positive measures $\lambda_j$ on $\R^{q_2}$ in some cases and combinations of these, see Jourdain and Tse \cite{jourdain2020central}. Explicit examples of mean-field models where the structure of the drift is of the form \ref{ass: basic lip} (ii) rather than \ref{ass: basic lip} (i) or (iii) are given for instance in \cite{coghi2018pathwise,  MR779460, meleard1996asymptotic, MJ98}.
{\bf 2)}
Condition (iii) is stronger than (ii): under Assumption \ref{ass: basic lip} (iii), 
%$y \mapsto C^{-1}\widetilde b(t,x,y)$ is at most $1$-Lipschitz,  therefore
$$\int_{\R^d}\widetilde b(t,x,\cdot)d(\mu'-\mu) \leq C\sup_{|\phi|_{\mathrm{Lip}}\leq 1}\int_{\R^d} \phi \,d(\mu'-\mu)  = C\mathcal W_{1}(\mu',\mu),$$
thus
%\begin{align*}
$\big|b(t,x',\mu')-b(t,x,\mu) \big|  = \big| \int_{\R^d}(\widetilde b(t,x',\cdot)-\widetilde b(t,x,\cdot))d\mu' - \int_{\R^d} \widetilde b(t,x,\cdot)\,d(\mu-\mu')\big|
 \leq C\big(|x'-x|+\mathcal W_1(\mu',\mu)\big)$
%\end{align*}
and Assumption \ref{ass: basic lip} (i) holds true. Moreover $\delta_\mu b(t,x,y,\mu) = \widetilde b(t,x,y)$ and Assumption \ref{ass: basic lip} (ii) holds true as well.
\end{rk}

%Under Assumptions \ref{ass: init condition}, \ref{ass: minimal prop sigma} and \ref{ass: basic lip}, we have existence of a strong solution of \eqref{eq:diff basique}, following the techniques of {\it e.g.} Sznitman \cite{Sznitman2} or M\'el\'eard \cite{meleard1996asymptotic}, together with an absolutely continuous weak solution of the nonlinear parabolic PDE \eqref{eq: mckv approx}, see {\it e.g.} \cite{BoKry}. 

We let $\mathcal C = \mathcal C([0,T],(\R^d)^N)$ denote the space of continuous functions on $(\R^d)^N$, equipped with the filtration $(\mathcal F_t)_{0 \leq t \leq T}$  induced by our observation, namely the canonical mappings  
$$X_t(\omega)  = \big(X_t^1(\omega), \ldots, X_t^N(\omega)\big) = \omega_t$$
and modified to be right-continuous for safety. For $\mu_0 \in \mathcal P_1$, the probability $\mathbb P^N$ on $(\mathcal C,\mathcal F_T)$ under which  the canonical process $X = (X_t^1, \ldots, X_t^N)_{0 \leq t \leq 1}$ is a weak solution of \eqref{eq:diff basique} for the initial condition $\mu_0^{\otimes N}$ is uniquely defined under Assumptions \ref{ass: init condition}, \ref{ass: minimal prop sigma} and \ref{ass: basic lip}. Recommended reference (that covers our set of assumptions) is the textbook by Carmona and Delarue \cite{carmona2018probabilistic} or the lectures notes of Lacker  \cite{lacker2018mean}).

\subsection{Kernel estimators} \label{sec: construct estim}
We pick two bounded and compactly supported kernel functions 
$H:(0,T) \rightarrow \R$ and $K:\R^d \rightarrow \R$
such that
$$\int_0^T H(s)ds = \int_{\R^d}K(y)dy = 1.$$
%For a bandwidth $h>0$, we set $H_h(t) = h^{-1}H(h^{-1}t)$ and $K_h(x) = h^{-d}K(h^{-1}x)$.
% so that $|H_h|_{1} = |H|_{1}$ and $|K_h|_{1} = |K|_{1}$.
Let $\ell \geq 1$ be an integer. We say that the kernels $H$ or $K$ {\it have order} $\ell$ if, for $k = 0,\ldots, \ell-1$, we have
% moreover
\begin{equation} \label{eq: def order}
\int_0^T s^kH(s)ds = \int_{\R^d}(y^1)^kK(y)dy^1 = \ldots =  \int_{\R^d}(y^d)^kK(y)dy^d  = {\bf 1}_{\{k = 0\}},
\end{equation}
 with $y = (y^1,\ldots, y^d)$.
\subsubsection*{Construction of an estimator of $\mu_t(x) \in \R$}
Let  $(t_0,x_0) \in (0,T] \times \R^d$. For $h >0$ we obtain a family of estimators of $\mu_{t_0}(x_0)$ by setting
\begin{equation} \label{eq: def mu hat}
\widehat \mu_h^N(t_0,x_0) = \int_{\R^d} K_h(x_0-x)\mu_{t_0}^N(dx),
\end{equation}
with $K_h(x) = h^{-d}K(h^{-1}x)$.
\subsubsection*{Construction of an estimator of $b(t,x,\mu_t) \in \R^d$} Let  $(t_0,x_0) \in (0,T) \times \R^d$.
Abusing notation slightly, define the $\R^d$-valued random measure 
% $\pi^N = \big((\pi^N)^1, \ldots, (\pi^N)^d\big)$ as
$$\pi^N(dt,dx) = N^{-1}\sum_{i = 1}^N X^i(dt)\delta_{X_t^i}(dx),$$
defined by
$$\int_{[0,T] \times \R^d} \phi(t,x) \pi^N(dt,dx) = N^{-1}\sum_{i = 1}^N\int_0^T\phi(t,X_t^i)dX_t^i$$
 for a test function $\phi: [0,T] \times \R^d \rightarrow \R$.
Letting
$$(H\otimes K)_{\boldsymbol h}(t,x) = H_{h_1}(t)K_{h_2}(x)\;\;\text{for}\;\;\boldsymbol h = (h_1,h_2),h_i>0,$$ 
with $H_h(t) = h^{-1}H(h^{-1}t)$, we obtain a family of estimators of $\pi(t_0,x_0) = b(t_0,x_0,\mu_{t_0})\mu_{t_0}(x_0)$ by setting
\begin{align}
\widehat \pi_{\boldsymbol h}^N(t_0,x_0)&  = \int_{[0,T] \times \R^d} (H\otimes K)_{\boldsymbol h}(t_0-t,x_0-x)\pi^N(dt,dx) \nonumber \\
& =  \Big(\int_{[0,T] \times \R^d} (H\otimes K)_{\boldsymbol h}(t_0-t,x_0-x)(\pi^N)^k(dt,dx)\Big)_{1 \leq k \leq d}, \label{eq: est pi}
\end{align}
{\it i.e.} by smoothing componentwise $\pi^N(dt,dx)$.  We finally estimate $b(t_0,x_0,\mu_{t_0})$ by 
$$\widehat b_{h,\boldsymbol h}^N(t_0,x_0)_\varpi = \frac{\widehat \pi_{\boldsymbol h}^N(t_0,x_0)}{\widehat \mu_h^N(t_0,x_0) \vee \varpi} \in \R^d,\;\;\boldsymbol h = (h_1,h_2),$$
with $h,h_i>0$  and some threshold $\varpi >0$ that prevents the estimator to blow-up for small values of $\widehat \mu_h^N(t_0,x_0)$.

 %\subsubsection*{Estimation of $\mu_t(x)$, data driven bandwidth}
 
 \section{Nonparametric oracle estimation} \label{sec: esti mu et b}
 
 Our results involve positive quantities that continuously depend on real-valued parameters of the problem, namely
$$\mathfrak b = \big(\gamma_0, \gamma_1,b_0,|b|_{\mathrm{Lip}}, |\delta^k_{\mu} b|_{\mathrm{Lip}}, m_b, \sigma_{\pm}, T, d\big),$$
defined in Assumptions  \ref{ass: init condition}, \ref{ass: minimal prop sigma} and \ref{ass: basic lip}, together with the dimension $d \geq 1$ of the ambient space and the value of the terminal time $T>0$. In the following, the notation $A_N \lesssim B_N$ means the existence of $C>0$ (possibly depending on $\mathfrak b$ but not $N$) such that $A_N \leq CB_N$ for every $N \geq 1$. 

%and a dimension restriction related to Assumption \ref{ass: basic lip}, namely
%\begin{itemize}
%\item Assumption \ref{ass: basic lip} (i): $d=1$,
%\item Assumption \ref{ass: basic lip} (ii): $d=1$ or ($d=2$ and $k \geq 2$) or ($d \geq 3$ and $k\geq d/2$),  
%\item Assumption \ref{ass: basic lip}  (iii): $d \geq 1$,
%\end{itemize}
%where $k \geq 1$ is the order of linear differentiability.

\subsection{Oracle estimation of $\mu_t(x)$} \label{sec: oracle mu}
 We fix $(t_0,x_0) \in (0,T]\times \R^d$ and implement a variant of the Goldenshluger-Lepski's algorithm \cite{GL08, GL11, GL14} for pointwise estimation.
Pick a discrete set
 %\begin{equation} \label{eq: condition grille 1}
 $$
 \mathcal H_1^N \subset \big[N^{-1/d}(\log N)^{2/d}, 1\big],
 $$
% \end{equation}
 of admissible bandwidths such that $\mathrm{Card}(\mathcal H_1^N) \lesssim N$. The algorithm, based on Lepski's principle, requires the family of estimators
$$\Big(\widehat \mu^N_h(t_0,x_0),h\in \mathcal H_1^N\Big)$$
defined in \eqref{eq: def mu hat} and selects an appropriate bandwidth $\widehat h^N$ from data $\mu_{t_0}^N(dx)$. Writing $\{x\}_+ = \max(x,0)$, define
\begin{equation} \label{eq: choix A}
\mathsf A_h^N = \max_{h'\leq h, h' \in \mathcal H_1^N}\Big\{\big(\widehat \mu_{h}^N(t_0,x_0)-\widehat \mu_{h'}^N(t_0,x_0)\big)^2-(\mathsf V_h^N+\mathsf V_{h'}^N)\Big\}_+,
\end{equation}
%The GL algorithm is specified by a variance correction function
where
\begin{equation} \label{eq: var correc mu}
\mathsf V^N_h  = \varpi_1 |K|_2^2(\log N) N^{-1}h^{-d},\;\;\varpi_1 >0.
\end{equation}
%calibrated with a (known) upper bound $C^\star \geq 16\kappa_3/\kappa_2$, where $\kappa_2 = \kappa_2(\mathfrak b)$ is defined in Theorem \ref{thm: concentration} and $\kappa_3 = \kappa_3(\mathfrak b)$ is a locally uniform upper bound of $\mu_{t_0}(x)$ in a neighbourhood of $x_0$, see Lemma \ref{lem: loc unif mu above}.
Let
\begin{equation} \label{eq: def GL band mu}
\widehat h^N \in \mathrm{argmin}_{h \in \mathcal H_1^N} \big(\mathsf A_h^N +\mathsf V_h^N\big).
\end{equation}
The data driven Goldenshluger-Lepski estimator of $\mu_{t_0}(x_0)$ defined by
$$\widehat \mu_{\mathrm{GL}}^N(t_0,x_0) = \widehat \mu_{\widehat h^N}^N(t_0,x_0)$$
is specified by $K$ and $\varpi_1$. 

\begin{rk} \label{rk Glepski}
The choice of the penalty $\mathsf A_h^N$ and the threshold $\mathsf V_h^N$ in \eqref{eq: choix A} and \eqref{eq: var correc mu} are standard in the GL methodology: $\mathsf A_h^N$ is  a kind of proxy for the estimation of the variance of $\widehat \mu_{h}^N(t_0,x_0)$ while $\mathsf V_h^N$ is the exact penalty needed in order to balance the size of the variance of the estimator in $h$, of order $|K|_2^2N^{-1}h^{-d}$, inflated by a logarithmic term $\log N$ and tuned with $\varpi_1>0$. This enables one to control all the stochastic deviation terms. See in particular the proof of Theorem \ref{thm: GL mu}. We also refer to the original sources in Goldenshluger and Lepski \cite{GL08, GL11, GL14}. 
\end{rk}

\subsubsection*{Oracle estimate} We need some notation. Given a kernel $K$, the {\it bias at scale} $h>0$ of $\mu$ at point $(t_0, x_0)$ is defined as   
\begin{equation} \label{eq: bias mu}
\mathcal B_h^N(\mu)(t_0,x_0) = \sup_{h' \leq h, h' \in \mathcal H_1^N}\Big|\int_{\R^d}K_{h'}(x_0-x)\mu_{t_0}(x)dx-\mu_{t_0}(x_0)\Big|.
\end{equation}
We are ready to give the performance of our estimator of $\mu_t(x)$, by means of an oracle inequality.

\begin{thm} \label{thm: GL mu}
Work under Assumptions  \ref{ass: init condition}, \ref{ass: minimal prop sigma} and \ref{ass: basic lip}. Let $(t_0,x_0) \in (0,T] \times \R^d$. The following oracle inequality holds true:
$$
 \E_{\PP^N}\big[\big(\widehat \mu_{\mathrm{GL}}^N(t_0,x_0)-\mu_{t_0}(x_0)\big)^2\big] 
 \lesssim 
 \min_{h \in \mathcal H_1^N} \big(\mathcal B_{h}^N(\mu)\big(t_0,x_0)^2+\mathsf V_h^N\big),
$$
for large enough $N$, up to a constant depending on $(t_0,x_0)$, $|K|_\infty$ and $\mathfrak b$, provided $\widehat \mu_{\mathrm{GL}}^N(t_0,x_0)$ is calibrated with
%$\varpi_1 \geq 64|K|_\infty^2/\kappa_2^2$ and  $\varpi_2 \geq \max(16\kappa_3/\kappa_2,1)$, 
$\varpi_1 \geq 16\kappa_2^{-1}\kappa_3$,
where $\kappa_2$ is specified in Theorem \ref{thm: concentration} and $\kappa_3$ is a (local) upper bound of $\mu_{t_0}$, see Lemma \ref{lem: loc unif mu above} below.
\end{thm}
Some remarks are in order:
{\bf 1)} Up to an inessential logarithmic factor, our estimator achieves the optimal bias-variance tradeoff among every possible bandwidth $h \in  \mathcal H_1^N$. {\bf 2)} The requirement $Nh^d \geq (\log N)^2$ for $h \in \mathcal H_1^N$ could be tightened to $Nh^d \geq (\log N)^{1+\epsilon}$ for an arbitrary $\epsilon >0$; this is slightly more stringent than the usual bound $Nh^d \geq \log N$ in the literature \cite{GL08, GL11, GL14}, but this has no consequence for the subsequent minimax results. {\bf 3)} The choice of a pointwise loss function at $(t_0,x_0)$ is inessential here: other integrated norms like $|\cdot|_p$ would work as well, following the general strategies of Lepski's principle. {\bf 4)} The construction of the estimator of $\mu_{t_0}(x_0)$ requires a lower bound on $\varpi_1$ that has to be set prior to the data analysis. The bound we obtain are presumably too large. In practice, $\varpi_1$ has to be tuned by other methods, possibly using data. Such approaches in the context of Lepski's methods have been recently introduced by Lacour {\it et al.} \cite{lacour2017estimator}. This weakness is common to all nonparametric methods that depend on a data-driven bandwidth. 

\subsection{Oracle estimation of $b(t,x,\mu_t)$} \label{sec: oracle b}
Similarly to the estimation of $\mu_t(x)$, we pick a discrete set 
\begin{equation} \label{eq: condition grid 2}
\mathcal H_2^N \subset \big[N^{-1/(d+1)}(\log N)^{2/(d+1)}, (\log N)^{-2} \big] \times \big[N^{-1/(d+1)}(\log N)^{2/(d+1)}, 1\big],
\end{equation}
%such that 
%\begin{equation} \label{eq: condition grid 3}
%Nh_1h_2^d \geq \, N^{\varpi_3},\;\;\varpi_3 \geq 1/2, 
%\end{equation} 
with
cardinality $\mathrm{Card}\, \mathcal H_2^N \lesssim N$. We assume that $\mathcal H_2^N $ is equipped with some ordering $\preceq$ such that for every $\boldsymbol h, \boldsymbol h' \in \mathcal H_2^N$, we have either $\boldsymbol h \preceq \boldsymbol h'$ or $\boldsymbol h' \preceq \boldsymbol h$.
%In particular, we always have $Nh_1h_2^d \gtrsim 1$. 
The construction uses $\widehat \mu_{\mathrm{GL}}^N(t_0,x_0)$, given in addition the family of estimators
$$\big(\widehat \pi_{\boldsymbol h}^N(t_0,x_0), \boldsymbol h\in \mathcal H_2^N\big)$$
defined in \eqref{eq: est pi} and constructed with the kernel $H \otimes K$.  
%For $\boldsymbol h = (h_1,h_2)$ and $\boldsymbol h' = (h_1',h_2')$ in $\mathcal H_2^N$, we say that $\boldsymbol h \leq \boldsymbol h'$ if
%$h_1 \leq h'_1$ and $h_2 \leq h'_2$ hold simultaneously. 
Define
\begin{equation} \label{eq: choice A bis}
\mathsf A_{\boldsymbol h}^N = \max_{\boldsymbol{h'} \preceq \boldsymbol{h}, \boldsymbol h'\in \mathcal H_2^N}\Big\{\big|\widehat \pi_{\boldsymbol h}^N(t_0,x_0)-\widehat \pi_{\boldsymbol h'}^N(t_0,x_0)\big|^2-(\mathsf V_{\boldsymbol h}^N+\mathsf V_{\boldsymbol h'}^N)\Big\}_+,
\end{equation}
where
\begin{equation} \label{def upper bivariance}
\mathsf V_{\boldsymbol h}^N =  \varpi_2 |H\otimes K|_2^2(\log N)N^{-1}h_1^{-1}h_2^{-d},\;\;\varpi_2 >0.
\end{equation}
%and $\varpi_4$ is a (known) upper bound of  {\color{red} [to be completed!!!]}.
Let
$$\widehat {\boldsymbol h}^N\in \text{argmin}_{\boldsymbol h \in \mathcal H_2^N}\big(\mathsf A_{\boldsymbol h}^N+\mathsf V_{\boldsymbol h}^N\big).$$

The data-driven Goldenshluger-Lepski estimator of $b(t_0,x_0,\mu_{t_0})$ is defined as 
\begin{equation} \label{def GL est b}
\widehat b_{\mathrm{GL}}^N(t_0,x_0)  = \widehat b_{\widehat h^N,\widehat {\boldsymbol h}^N}^N(t_0,x_0)_{\varpi_3}  = \frac{\widehat \pi_{\widehat {\boldsymbol h}^N}^N(t_0,x_0)}{\widehat \mu_{\widehat {h}^N}(t_0,x_0)\vee \varpi_3}
\end{equation}
and is specified by $H,K, \varpi_1, \varpi_2$ and the threshold $\varpi_3 >0$ that prevents the estimator to blow-up for small values of $\widehat \mu_{\widehat {h}^N}(t_0,x_0)$.

\begin{rk} The same comments as in Remark \ref{rk Glepski} apply here for the specification of $\mathsf A_{\boldsymbol h}^N$ in \eqref{eq: choice A bis} and $\mathsf V_{\boldsymbol h}^N$ in \eqref{def upper bivariance}, noting that in the anisotropic case, the variance of $\widehat \pi_{\boldsymbol h}^N(t_0,x_0)$ is now of order $|H \otimes K|_2^2N^{-1}h_1^{-1}h_2^{-d}$. 
\end{rk}

\subsubsection*{Oracle estimates} 
%In order to estimate $b(t,x,\mu_t)$ in squared-error loss consistently with the quotient estimator \eqref{def GL est b}, we need a (local) lower bound assumption on $\mu(t_0,x_0)$.
%
%\begin{assumption} \label{assumption minoration g}
%For every $(t,a) \in \mathcal D^-$ there exists an open set $\mathcal U_{(t,a)}$ such that
%\begin{equation} \label{eq condit mino g L}
%\inf_{u \in \mathcal U_{(t,a)}}b(t-a,t-a+u)g_0(u)\geq \delta \;\;\text{if}\;\; (t,a) \in \mathcal D_L
%\end{equation}
%and
%\begin{equation} \label{eq condit mino g U}
%g_0(t-a) \geq \delta \;\;\text{if}\;\; (t,a) \in \mathcal D_U,
%\end{equation}
%for some $\delta >  0$.
%\end{assumption}
Given a kernel $H \otimes K$, the bias at scale $\boldsymbol h$ of $\pi = b\mu $ at point $(t_0,x_0)$ is defined as
\begin{equation} \label{eq: bias pi}
\mathcal B_{\boldsymbol h}^N(\pi)(t_0,x_0) =  \sup_{\boldsymbol h' \preceq \boldsymbol h, \boldsymbol h' \in \mathcal H_2^N}\Big|
%\int_0^T \int_0^\infty
\int_{[0,T] \times \R^d} (H \otimes K)_{\boldsymbol h'}(t_0-t,x_0-x)\pi(t,x)dxdt-\pi(t_0,x_0)\Big|.
\end{equation}
We are ready to give an oracle bound for the estimation of $b(t,x,\mu_t)$. 
\begin{thm} \label{thm: oracle b}
Work under Assumptions \ref{ass: init condition}, \ref{ass: minimal prop sigma} and \ref{ass: basic lip}. Let $(t_0,x_0) \in (0,T) \times \R^d$. The following oracle inequality holds
$$
 \E_{\PP^N}\big[\big|\widehat b_{\mathrm{GL}}^N(t_0,x_0)-  b(t_0,x_0,\mu_{t_0})\big|^2\big] 
 \lesssim  
 \min_{h \in \mathcal H_1^N} \big(\mathcal B_{h}^N(\mu)(t_0,x_0)^2+  \mathsf V_{h}^N\big)+\min_{\boldsymbol h \in \mathcal H_2^N} \big(\mathcal B_{\boldsymbol h}^N(\pi)(t_0,x_0)^2+{\mathsf V}_{\boldsymbol h}^N\big),
$$
for large enough $N$, up to a constant depending on  $(t_0,x_0)$, $|H\otimes K|_\infty$, and $\mathfrak b$, provided 
$$\varpi_2 \geq 12d \kappa_3\max(12T\kappa_2^{-1}\kappa_5^2, 25|\mathrm{Tr}(c)|_\infty)\;\;\text{and}\;\;\varpi_3 \leq \kappa_4,$$
where $\kappa_2$ is defined in Theorem \ref{thm: concentration}, and $\kappa_3, \kappa_4, \kappa_5$ are (local) upper or lower bounds on $\mu$ and $b$ defined in Lemma \ref{lem: loc unif mu above} below.
\end{thm}
Some remarks: {\bf 1)} The same remarks as {\bf 1)}, {\bf 2)}, {\bf 3} and {\bf 4)} after Theorem \ref{thm: GL mu} are in order. This includes the calibration of $\varpi_2, \varpi_3$ and the requirement $Nh_1h_2^d \geq (\log N)^2$ for $h \in \mathcal H_2^N$ that could be tightened to $Nh_1h_2^d \geq (\log N)^{1+\epsilon}$ for an arbitrary $\epsilon >0$. However, the requirement $h_1 \leq (\log N)^{-2}$ is a bit unusual and necessary for technical reason: it enables us to  manage the delicate term $II$ in the proof. Fortunately, this has no consequence for the subsequent minimax results, since we always look for oracle bandwidths of the form $N^{-\epsilon}$ that are much smaller than $(\log N)^{-2}$. {\bf 2)} The estimator $\widehat b_{\mathrm{GL}}^N$ is a quotient estimator that estimates the ratio of $\pi(t,x) = b(t,x,\mu_t)\mu_t(x)$ and $\mu_t(x)$, very much in the sense of a Nadaraya-Watson (NW) type estimator in regression \cite{NADARAYA}. Its performance is similar to the worst performance of the estimation of the product $\pi$ and $\mu$. However, the smoothness of $\mu$ is usually no worse than the smoothness of $b$ and we do not lose in terms of approximation results, see Section \ref{sec: adaptive estimation} and Proposition \ref{prop : reg Holder} below for the formulation of a minimax theory in this setting. 

\section{Adaptive minimax estimation} \label{sec: adaptive estimation}
\subsection{Anisotropic H\"older smoothness classes for McKean-Vlasov models}
\begin{defi}  \label{def: holder space}
Let $x_0 \in \R^d$ and $\mathcal U$ a neighbourhood of  $x_0$.   
We say that $f:\R^d \rightarrow \R$ belongs to $\mathcal H^{\alpha}(x_0)$ with $\alpha >0$
if for every $x,y \in \mathcal U$
\begin{equation} \label{def holder}
|D^s f(y)- D^s f(x)| \leq C|y-x|^{\alpha - \lfloor \alpha \rfloor}
\end{equation}
for any $s$ such that $|s| \leq \lfloor \alpha \rfloor$, the largest integer {\it strictly smaller} than $\alpha$; $s \in \mathbb N^d$ is a multi-index with $|s| = s_1+\ldots+s_d$ and $D^s  = \frac{\partial^{|s|}}{\partial_1^{s_1}\ldots \partial_d^{s_d}}$.

\end{defi}
The definition depends on $x_0$ via $\mathcal U$ but this is further omitted in the notation for simplicity. We obtain a semi-norm by setting 
$$|f|_{\mathcal H^\alpha(x_0)} = 
\sup_{x \in \mathcal U}|f(x)|+C(f),$$
where $C(f)$ is the smallest constant $C$ for which \eqref{def holder} holds. The extension of Definition \ref{def: holder space} for $\R^d$-valued functions is straightforward by considering coordinate functions. For time-varying functions defined on $(0,T)$ we have the
% $f^k : \R^d \rightarrow \R$ in the representation $f = (f^1,\ldots, f^d)$. 
\begin{defi} \label{def anisotropic space} Let $(t_0,x_0) \in (0,T)\times \R^d$ and $\alpha, \beta >0$. 
The function $f: (0,T) \times \R^d \rightarrow \R$ belongs to the anisotropic H\"older class $\mathcal H^{\alpha,\beta}(t_0,x_0)$ if
\begin{equation} \label{eq: def aniso holder}
|f|_{\mathcal H^{\alpha,\beta}(t_0,x_0)} = |f(\cdot,x_0)|_{\mathcal H^{\alpha}(t_0)}+ |f(t_0,\cdot)|_{\mathcal H^{\beta}(x_0)} < \infty.
\end{equation}
\end{defi}
Again, the extension to $\R^d$-valued functions is straightforward: a mapping $f = (f^k)_{1 \leq k \leq d} : (0,T) \times \R^d \rightarrow \R^d$ belongs to $\mathcal H^{\alpha,\beta}(t_0,x_0)$ if $f^k \in \mathcal H^{\alpha,\beta}(t_0,x_0)$ for every $k=1,\ldots, d$.\\

We model H\"older smoothness classes for the density function $\mu_t(x)$ and the drift $b(t,x,\mu_t)$. The McKean-Vlasov model \eqref{eq:diff basique} is parametrised by $(b,\sigma, \mu_0)$, or rather $(b,c,\mu_0)$, with $c = \sigma \sigma^\top$. We denote by $\mathcal P = \mathcal P(\mathfrak b)$ the class of $(b,c,\mu_0)$ satisfying Assumptions \ref{ass: init condition}, \ref{ass: minimal prop sigma} and \ref{ass: basic lip} with model parameter $\mathfrak b$. We let 
$$(b,c,\mu_0) \mapsto \mu = \mathcal S(b,c,\mu_0) $$
denote the solution (or forward) map of \eqref{eq: mckv approx}.
%Denoting by $\mathcal P = \mathcal P(\mathfrak b)$ the class of triples $(b,c,\mu_0)$ satisfying Assumptions \ref{ass: minimal prop sigma}, \ref{ass: basic lip} and \ref{ass: init condition} with model parameter $\mathfrak b$, we have
\begin{defi}
Let $\alpha,\beta >0$. The anisotropic  H\"older  class $\mathcal S^{\alpha,\beta}(t_0,x_0)$ to the solution of \eqref{eq: mckv approx} is defined by
$$\mathcal S^{\alpha,\beta}(t_0,x_0) = \Big\{(b,c,\mu_0) \in \mathcal P,\;\;\mu = \mathcal S(b,c,\mu_0) \in \mathcal H^{\alpha,\beta}(t_0,x_0)\Big\}.$$  
\end{defi}

Before establishing minimax rates of convergence, we briefly investigate how rich is the class $\mathcal S^{\alpha,\beta}(t_0,x_0)$. {\bf 1)}
If $b(t,x,\mu) = b(t,x)$ does not involve an interaction, then we have explicit formulas for $\mu$ in some cases when $c$ is regular, see {\it e.g.} Genon-Catalot and Jacod \cite{GCJ94} and the formulas become relatively tractable in dimension $d=1$, especially when $b(t,x) = b(x)$, $c(t,x)=c(x)$ and $\mu_0$ is the invariant distribution of the diffusion process $X_t^i$ provided it exists. In that case, one can construct $\mu \in \mathcal H^{\alpha,\beta}(t_0,x_0)$ with arbitrary $\alpha, \beta >0$ for specific choices of $c$ and $b$ thanks to Feller's classification of scalar diffusions (see {\it e.g.} Revuz and Yor \cite{RY99}). 
%Also, $\mathcal H^{\alpha,\beta}_{\mathrm{drift}}(t_0,x_0)$ coincide with the set of coefficients $(b,c,\mu_0)$ for which $b \in \mathcal H^{\alpha,\beta}(t_0,x_0)$ and without further conditions on $(c,\mu_0)$. 
{\bf 2)} For a non-trivial representation of $b(t,x,\mu)$ as in the Vlasov model, we have the following result, that shows how versatile the classes $\mathcal H^{\alpha,\beta}(t_0,x_0)$ can be.

\begin{prop} \label{prop : reg Holder}
Let $c(t,x) = \tfrac{1}{2}\sigma^2 \mathrm{Id}$ with $\sigma >0$ and $b(t,x,\mu) = b(x,\mu) = F \star \mu(x)+G(x)$, with $F $ having compact support, $\mu_0 \in \mathcal P_1$ with a continuous bounded density satisfying Assumption \ref{ass: init condition} and
$$|G|_{{\mathcal H}^{\beta}}+|F|_{{\mathcal H}^{\beta'}} + |\mu_0|_{{\mathcal H}^{\beta''}} <\infty,$$
for some $\beta, \beta' >1$ and $\beta'' >0$ (and $\beta$ non-integer for technical reason). Here, $\mathcal H^{\beta}$ denotes the global H\"older space (obtained when taking $\mathcal U = \R^d$ in Definition \ref{def: holder space}). Then, for every $(t_0,x_0) \in (0,T) \times \R^d$, we have $\mu \in \mathcal H^{\alpha,\beta+1}(t_0,x_0)$  and $b(\cdot,\cdot,\mu_\cdot) \in \mathcal H^{\alpha, \beta}(t_0,x_0)$ with $\alpha = (\beta +1)/2$.

\end{prop}

The proof relies on classical estimates for parabolic equations, see {\it e.g.} the textbook by Bogatchev {\it et al.} \cite{BoKry}. See also M\'el\'eard and Jourdain \cite{MJ98}  for analogous results. It is sketched in Appendix. 
\ref{proof : reg Holder}.

\subsection{Minimax adaptive estimation of $\mu_t(x)$} For $\alpha,\beta, L >0$, we set
\begin{equation} \label{eq: def aniso mu}
\mathcal S^{\alpha, \beta}_L(t_0,x_0) = \Big\{(b,c,\mu_0) \in \mathcal P,\;|\mathcal S(b,c,\mu_0)|_{\mathcal H^{\alpha,\beta}(t_0,x_0)} \leq L\Big\},
\end{equation}
that shall serve as a smoothness model for the unknown $\mu$, where the semi-norm $|\cdot|_{\mathcal H^{\alpha,\beta}(t_0,x_0)}$ is defined in \eqref{eq: def aniso holder}.\\

Since the estimator $\widehat \mu_{\mathrm{GL}}^N(t_0,x_0)$ is built on $\mu_{t_0}^N$ solely and not the whole process $(\mu_t^N)_{0 \leq t \leq T}$, we study minimax rates of convergence in restriction to the experiment generated by $\mu_{t_0}^N$. We have the following  adaptive upper bound and accompanying lower bound for estimating $\mu$: 
\begin{thm} \label{thm minimax mu}
Work under Assumptions  \ref{ass: init condition}, \ref{ass: minimal prop sigma} and \ref{ass: basic lip}. Let $\widehat \mu_{\mathrm{GL}}^N(t_0,x_0)$ be specified with a kernel $K$ of order $\ell \geq 1$ as defined in \eqref{eq: def order}. 
% {\color{red} [smoothness restrictions!!]}.
For every $(t_0,x_0) \in (0,T) \times \R^d$, we have, 
\begin{equation} \label{eq: UB mu}
\sup_{(b,c,\mu_0)}\big(\E_{\PP^N}\big[\big(\widehat \mu_{\mathrm{GL}}^N(t_0,x_0)-\mu_{t_0}(x_0)\big)^2\big]\big)^{1/2} \lesssim \Big(\frac{\log N}{N}\Big)^{\beta\wedge \ell/(2\beta\wedge \ell+d)}
\end{equation} 
for large enough $N$, up to a constant that depends on $\mathfrak b$, $K$, $\beta$ and $L$ only. Moreover
\begin{equation} \label{eq: LB mu}
\inf_{\widehat \mu}\sup_{(b,c,\mu_0)}\E_{\PP^N}\big[|\widehat \mu-\mu_{t_0}(x_0)|\big] \gtrsim N^{-\beta/(2\beta+d)}
\end{equation} 
for large enough $N$. The infimum in \eqref{eq: LB mu} is taken over all estimators constructed with $\mu_{t_0}^N$. The supremum in \eqref{eq: UB mu} and \eqref{eq: LB mu} is taken over 
$\mathcal S^{\alpha, \beta}_L(t_0,x_0)$ for arbitrary $\alpha,\beta,L>0$.
%{\color{blue}for arbitrary $\alpha > (1-d/2)_+ {\color{green} (?) },\beta,L>0$. }
\end{thm}
Some remarks: {\bf 1)} We obtain the (nearly) optimal rate of convergence $N^{-\beta/(2\beta+d)}$ for estimating $\mu_{t_0}(x_0)$ as a function of the $d$-dimensional state variable $x_0$ for fixed $t_0$. The result obviously does not depend on the smoothness of $t \mapsto \mu_{t}(x_0)$ and is constructed from data $(X_{t_0}^1,\ldots, X_{t_0}^N)$ for fixed $t_0$. {\bf 2)} The extra logarithmic payement is unavoidable for pointwise estimation, as a result of the classical Lepski-Low phenomenon \cite{LEPSKI, LOW}. 
{\bf 3)} Although the result is stated for arbitrary $\alpha,\beta >0$, the mapping $(t,x) \mapsto \mu_t(x)$ is locally smooth; there is no contradiction and result must be understood as bounds that are valid over smooth functions $\mu$ having prescribed $\mathcal S^{\alpha, \beta}(t_0,x_0)$ semi-norms.
{\bf 4)} Our concentration result Theorem \ref{thm: concentration} for the fluctuation of $\mu_t^N$ around $\mu_t$ enables us to improve on the estimation result in Proposition 2.1 of Bolley {\it et al.}\,\cite{bolley2007quantitative} that achieves the (non-adaptive) suboptimal rate $N^{-\beta/(2\beta + 2d + 2)}$. This is due to the fact that Bolley {\it et al.} \cite{bolley2007quantitative} rely on controlling the fluctuations of $\mu_t^N$ around $\mu_t$ in Wasserstein distance, and therefore have to accomodate a dimension effect that we do not have here. 
%Also, the Wasserstein distance is not well suited for controlling kernel estimators, that have unbounded Lipschitz norm as $N \rightarrow \infty$. 
% {\bf 3)} The same remark as in Theorem \ref{thm: GL mu} for accommodating other loss functions holds here. 

\subsection{Minimax adaptive estimation of $b(t,x,\mu_t)$}
For $\alpha, \beta>0$ and $L>0$, in analogy to the class $\mathcal S^{\alpha,\beta}_L(t_0,x_0)$ defined in \eqref{eq: def aniso mu} above, we model the smoothness of the function $(t,x) \mapsto b(t,x,\mu_t)$ via the class
$$\mathcal D^{\alpha,\beta}_L(t_0,x_0) =  \big\{(b,c,\mu_0) \in \mathcal P,\;|b(\cdot,\cdot,\mathcal S(b,c,\mu_0)_.)|_{\mathcal H^{\alpha,\beta}(t_0,x_0)} \leq L\big\}.$$
Define the effective anisotropic smoothness $s_d(\alpha,\beta)$ by
$$\frac{1}{s_d(\alpha,\beta)} = \frac{1}{\alpha}+\frac{d}{\beta}.$$
 We have the following adaptive upper bound and accompanying lower bound for estimating the drift $b$:

\begin{thm} \label{thm minimax b}
Work under Assumptions  \ref{ass: init condition}, \ref{ass: minimal prop sigma} and \ref{ass: basic lip}. Let $\widehat b_{\mathrm{GL}}^N(t_0,x_0)$ be constructed with kernels $H$ and $K$ of order $\ell \geq 1$. 
%{\color{red} [smoothness restrictions!!]}
For every $(t_0,x_0) \in (0,T) \times \R^d$, we have
\begin{equation} \label{eq: UB b}
\sup_{(b,c,\mu_0)}\big(\E_{\PP^N}\big[|\widehat b_{\mathrm{GL}}^N(t_0,x_0)-b(t_0,x_0,\mu_{t_0})|^2\big]\big)^{1/2}  \lesssim \Big(\frac{\log N}{N}\Big)^{s_d(\alpha,\beta)\wedge \ell_d/(2s_d(\alpha,\beta)\wedge \ell_d+1)},
\end{equation}
for large enough $N$, with $\ell_d = \ell/d$, up to constants that depend on $\mathfrak b$, $H \otimes K$ and $\alpha,\beta, L$ only. Moreover,
\begin{equation} \label{eq: LB b}
\inf_{\widehat b}\sup_{(b,c,\mu_0)}\E_{\PP^N}\big[|\widehat b -b(t_0,x_0,\mu_{t_0})|\big] \gtrsim N^{-s_d(\alpha,\beta)/(2s_d(\alpha,\beta)+1)},
\end{equation} 
for large enough $N$. The infimum in \eqref{eq: LB b} is taken over all estimators constructed with $(\mu_t^N)_{0 \leq t \leq T}$. The supremum in \eqref{eq: UB b} and \eqref{eq: LB b} is taken over $ \mathcal S^{\alpha, \beta}_L(t_0,x_0) \cap \mathcal D^{\alpha,\beta}_L(t_0,x_0)$ for some known non-decreasing parametrisation $\alpha = \alpha(\beta) >0$ such that $\alpha(\beta)/\beta$ is non-increasing, with $\beta >0$ and $L >0$.
%such that $s_d(\alpha,\beta) \geq \varpi_3/2$ and $L>0$.
\end{thm}
Some remarks: {\bf 1)} Theorem \ref{thm minimax b} establishes that estimating $b(t,x,\mu_t)$ has the same complexity as estimating $d$ functions of $1+d$ variables in the time and space domain. Each component $t \mapsto b^k(t,x,\mu_t)$ has smoothness $\alpha$, while  $x^\ell \mapsto b^k(t,(x^1,\ldots, x^\ell,\ldots,x^d),\mu_t)$ has smoothness $\beta$ for $\ell = 1,\ldots, d$, resulting in an anisotropic function $(t,x) \mapsto b^k(t,x,\mu_t)$ of $1+d$ variables with smoothness index $(\alpha,\beta,\ldots,\beta)$. We therefore recover the usual anisotropic minimax rate of convergence, with effective smoothness $s_d(\alpha,\beta)$ obtained as the arithmetico-geometric mean of the smoothness index $(\alpha,\beta,\ldots,\beta)$. {\bf 2)} For technical simplicity, we consider $\alpha$ and $\beta$ to be linked, as for instance in Proposition \ref{prop : reg Holder} where we have $\alpha = \alpha(\beta) = (\beta+1)/2$. This somehow weakens our anisotropic adaptation result, but enables us to easily construct a well-behaved ordering $\preceq$ for $\mathcal H_2^N$ that behaves well with respect to the bias at scale $\boldsymbol h$ as defined in \eqref{eq: bias pi}. Dropping this restriction is possible in principle, as in the original paper of Goldenshluger and Lepski \cite{GL11}, yet at a significant additional technical cost.  
{\bf 3)} The remarks {\bf 2)} and {\bf 3)} of Theorem \ref{thm minimax mu} are valid here as well. 

\section{Estimation of the interaction in the Vlasov model} \label{sec: identif interaction potential}

In this section, we work under Assumptions \ref{ass: init condition} and \ref{ass: minimal prop sigma} with a constant $c(t,x) = \tfrac{1}{2}\sigma^2 \mathrm{Id}$ for some $\sigma >0$ and Assumption \ref{ass: basic lip} (iii), {\it i.e.} in the Vlasov case 
$$b(t,x,\mu) = b(x,\mu) = \int_{\R^d}\widetilde b(x,y)\mu(dy),$$
%the drift $b$ takes the form
for $d\geq 1$ and with a time homogeneous drift kernel
%\begin{equation} \label{eq: form drift potential}
$$
\widetilde b(x,y) = F(x-y)+G(x),\;\;x,y\in \mathbb \R^d.
$$
%\end{equation}
Model \eqref{eq:diff basique} then reads
%\begin{equation} \label{eq:diff basique}
$$
\left\{
\begin{array}{l}
dX_t^i = G(X_t^i)dt + N^{-1}\sum_{j = 1}^N F(X_t^i-X_t^j)dt+\sigma dB_t^i,\;\;1 \leq i \leq N,\; t \in [0,T],\\  \\
\mathcal L(X_0^1,\ldots, X_0^N) = \mu_0^{\otimes N}.
\end{array}
\right.
$$
%\end{equation}
We assume that each component $F^k \in L^1(\R^d)$ for every $k=1,\ldots, d$. We are interested in identifying the interaction function $x \mapsto F(x)$ from data \eqref{eq: data}
and possibly  $x \mapsto G(x)$, rather considered here as a nuisance parameter\footnote{In particular, it is a first step toward the interesting problem of testing the hypothesis 
$F=0$ against a set of local alternatives that quantify how far $F$ is from being constant.}. We have

\begin{align} \label{eq: the deconvol}
b(x,\mu_t)&  = G(x)+\int_{\mathbb R^d} F(x-y)\mu_t(y)dy \nonumber \\
& = G(x)+F\star\mu_t(x),
\end{align}
where $f\star \mu_t(x) = \big(\int_{\mathbb R^d}f^k(x-y)\mu_t(y)dy\big)_{1 \leq k \leq d}$ denotes the convolution  between $f:\mathbb R^d \rightarrow \R^d$ and $\mu_t$.

\subsection{Identification of the interaction $F$} Introduce the linear form $\mathcal L$ acting on test functions $\varphi: [0,T] \rightarrow \mathbb C$ defined by 
\begin{equation} \label{eq: property Lt}
\mathcal L\varphi = \int_{[0,T]}\varphi(t)w(t)\rho(dt),
%|\mathcal L\varphi| \leq C|\varphi|_\infty\;\;\text{and}\;\;\mathcal L 1 = 0
\end{equation}
where $\rho$ is a probability distribution on $[0,T]$ and $w: [0,T]\rightarrow \R$ a bounded weight function such that $\int_{[0,T]}w(t)\rho(dt)=0$. Note that $\mathcal L 1 = 0$
where $1$ denotes the constant function. Applying $\mathcal L$ on both sides of \eqref{eq: the deconvol}, we obtain
\begin{equation} \label{eq: apply L}
\mathcal Lb(x,\mu) = F \star \mathcal L\mu (x)
\end{equation}
by Fubini's theorem. For $f  = (f^k)_{1 \leq k \leq d}: \R^d \rightarrow \R^d$ with  $f^k\in L^1(\R^d)$, define a Fourier transform
$$\mathcal F(f)(\xi) = \big(\int_{\R^d}\mathrm{e}^{-2i\pi \xi^\top x}f^k(x)dx\big)_{1 \leq k \leq d},\;\;\xi \in \R^d,$$
so that whenever $g \in L^1(\R^d)$, we have $\mathcal F(f\star g) = \mathcal F(f)\cdot \mathcal F(g)$. We infer from \eqref{eq: apply L}
%\begin{equation} \label{eq: first convol}
%\mathcal F\big(-b(\cdot,\mu_t)\big) = \mathcal F(\nabla V)+\mathcal F(\nabla W) \cdot \mathcal F(\mu_t).
%\end{equation}
%
%
%%We write $|\mathcal L|_\infty$ for the smallest constant $C$ for which \eqref{eq: property Lt} holds. 
%%Typical examples are given by difference operators of the form 
%%$$\mathcal L\varphi = \int_{[0,T]}\varphi(t)w(t)\rho(dt)$$
%%$\mathcal L\varphi(t) = (\partial_t\varphi)(t_0)$ for some $t_0 \in (0,T)$ if $\varphi$ is smooth enough or
%Applying $\mathcal L$ on both side of \eqref{eq: first convol} and assuming we can interchange $\mathcal F$ and $\mathcal L$ we derive
$$\mathcal F\big(\mathcal L b(\cdot,\mu)\big) = \mathcal F(F) \cdot \mathcal F(\mathcal L\mu).
$$
%and taking time derivative:
%$$\mathcal F\big(-\partial_t b(t,\cdot,\mu_t)\big)_{\ell} = \mathcal F(\nabla W)_{\ell} \mathcal F(\partial_t \mu_t)_{\ell} \in \mathbb C^d.$$
This yields the formal decomposition
\begin{equation} \label{eq: expansion W fourier}
\mathcal F(F) =  \frac{\mathcal F\big(\mathcal L b(\cdot,\mu)\big)}{ \mathcal F(\mathcal L\mu)}, 
%= \frac{\mathcal L \mathcal F\big(-b(\cdot,\mu)\big)}{\mathcal L \mathcal F(\mu)}
\end{equation}
provided the quotient is well defined.

\subsection{Consistent estimation of $F$}
A first estimation strategy  consists in plugging-in our estimators of $b(t,x,\mu_t)$ and $\mu_t$  in \eqref{eq: expansion W fourier} above. A somewhat simpler estimator of $ \mathcal F(\mathcal L\mu)(\xi) = \mathcal L \mathcal F(\mu)(\xi)$ is given by the periodogram
$$\mathcal L\big(\int_{\R^d}\mathrm{e}^{-2i\pi \xi^\top x}\mu^N(dx)\big) =  \int_{\R^d}\mathrm{e}^{-2i\pi \xi^\top x}\mathcal L\mu^N(dx) = \mathcal F(\mathcal L \mu^N)(\xi)$$
for which we need not tune a bandwidth.
% provided we have an extension of $\mathcal L_t$ on probability measures.
Following Johannes \cite{JJ09}, 
%we introduce the regularisation $\ell_\alpha(\xi) = (1+|\xi|^2)^{\alpha/2}$ for $\alpha >0$ and 
we obtain an estimator  of $F$ by the formula
\begin{equation} \label{eq: def est W}
\mathcal F(\widehat{F}^N_{\varpi, \varpi'}) = \frac{\mathcal F\big(\mathcal L\big((\widehat b_{h,\boldsymbol h}^N)_{\varpi'}^r\big)\big)\cdot \overline{\mathcal F(\mathcal L \mu^N)}}{|\mathcal F(\mathcal L \mu^N)|^2} {\bf 1}_{\{|\mathcal F(\mathcal L\mu^N) |^2 \geq  \varpi\}} \in \R^d
\end{equation}
for some threshold $\varpi >0$ vanishing as $N \rightarrow \infty$, with the estimator 
$$\widehat b_{h,\boldsymbol h}^N(t_0,x_0)_{\varpi'}^r = \widehat b_{h,\boldsymbol h}^N(t_0,x_0)_{\varpi'}{\bf 1}_{\{|x| \leq r\}}$$
of $b(t_0,x_0,\mu_{t_0})$,  constructed in Section \ref{sec: construct estim} for some threshold $\varpi'>0$ and  bandwidths $h>0$ and $\boldsymbol h>0$. We also set the estimator to be equal to $0$ outside $|x| \leq r$ for some $r >0$.
We obtain a consistency result under the following additional assumption: 
\begin{assumption} \label{ass: fourier et derivation t}
We have $|\mathcal F(\mathcal L \mu)(\xi)| >0$ $d\xi$-almost everywhere.
\end{assumption}
\begin{thm} \label{thm: identif Vlasov}
Work under the assumptions of Proposition \ref{prop : reg Holder} and Assumption \ref{ass: fourier et derivation t}. Assume moreover that $G$ is in $L^1(\R^d) \cap L^2(\R^d)$ componentwise. If $w$ has compact support in $(0,T)$, there exists a choice of $(\varpi, \varpi', h, \boldsymbol h) = (\varpi_N, \varpi_N', h_N, \boldsymbol h_N) \rightarrow 0$ and $r = r_N \rightarrow \infty$ such that 
%for every $1 \leq k \leq d$, we have
$$\E_{\PP^N}\big[|\widehat{F}^N_{\varpi, \varpi'}-F|_2^2\big] \rightarrow 0\;\;\text{as}\;\;N \rightarrow \infty.$$
\end{thm}

Some remarks: {\bf 1)} Theorem \ref{thm: identif Vlasov} proves that we can reconstruct the interaction force $F$ from data \eqref{eq: data}, while the function $G$ remains a nuisance parameter. This is a first step for the construction of a statistical test based on data \eqref{eq: data} for the presence against the absence of interaction between particles in the Vlasov model.   
{\bf 2)} Although we obtain consistency, we do not have a rate of convergence and our result is not uniform in the model parameter. A glance at the proof of Theorem \ref{thm: identif Vlasov} shows that it is possible to cook-up a result with a rate of convergence and some uniformity in the parameter, provided we have a sharp control from below on the decay $|\mathcal F(\mathcal L\mu)(\xi)|$ or rather $|\mathcal F(\mu)(\xi)|$ as $|\xi| \rightarrow \infty$ as well as the decay of $\inf_{t \in [r_1,r_2], |x| \leq r}\mu_t(x)$ for given $[r_1,r_2] \subset (0,T)$ as $r \rightarrow \infty$. 
This requires the exact knowledge of the smoothness of the solution map $\mu = \mathcal S(b,c,\mu_0)$, and it is a delicate issue, see Proposition \ref{prop : reg Holder}; we can anticipate ill-posedness. 
%Also, in such regimes, the restrictions on $\mathcal H_1^N$ and $\mathcal H_2^N$ could become an issue in order to achieve optimal rates in the minimax sense. 
{\bf 3)} From our estimator of $F$, we may construct a plug-in estimator for the function $G$ by setting $$\widehat{G}^N(x)_{\varpi, \varpi'} = - \widehat b_{h, \boldsymbol h}^N(t,x)_{\varpi'}^r-\widehat{F}
^N_{\varpi, \varpi'} \star_x \widehat \mu_{h}^N(t,x).$$ A consistency result can be obtained in the same way as for Theorem \ref{thm: identif Vlasov}. {\bf 4)}, the deconvolution method we employ here requires  quite a stringent localisation assumption on the external force $G$ and the interaction force $F$. As pointed out by a referee, it does not apply for gradient forces of the form $G=-\nabla V$ and $F=-\nabla W$ where $V$ and $W$ diverge polynomially at infinity like {\it e.g.} in  \cite{benachour1998nonlinear, herrmann2008large, CATTIAUXetal} for which alternative methods yet need to be constructed.

\section{Probabilistic tools: a concentration inequality} \label{sec: proof concentration}

\subsection{A Bernstein inequality} \label{subsec: concentration}
Let $\rho(dt)$ be a probability measure on $[0,T]$. We establish a deviation inequality for the sequence of signed measures
$\nu^N(dt, dx) - \nu(dt,dx),$
%\end{equation}
where 
$$\nu^N(dt, dx) = \mu_t^N(dx) \otimes \rho(dt)\;\;\text{and}\;\;\nu(dt,dx) = \mu_t(x)dx \otimes \rho(dt).$$
We have a Bernstein concentration inequality:
\begin{thm} \label{thm: concentration}
Work under Assumptions  \ref{ass: init condition}, \ref{ass: minimal prop sigma} and \ref{ass: basic lip}.
Let $(\mu_t)_{0 \leq t \leq T}$ denote the unique solution of \eqref{eq: mckv approx} with $\mu_{t=0}=\mu_0$ satisfying Assumption \ref{ass: init condition}. Then there exist  $\kappa_1, \kappa_2 > 0$ depending on $\mathfrak b$ such that 
$$
\mathbb P^N\Big(\int_{[0,T] \times \R^d} \phi(t,y)\big(\nu^N(dt,dy)-\nu(dt, dy)\big) \geq x\Big) \leq  \kappa_1\exp\Big(-\kappa_2\frac{ Nx^2}{|\phi|_{L^2(\nu)}^2+|\phi|_\infty x}\Big)
$$
for every $x \geq 0$, for every bounded $\phi: [0,T] \times \R^d \rightarrow \R$, and for any probability measure $\rho(dt)$ on $[0,T]$.
\end{thm}
As a corollary, for a bounded $\phi: \R^d \rightarrow \R$ and any $0 \leq t_0 \leq T$, picking $\rho(dt) = \delta_{t_0}(dt)$, we obtain
$$
\mathbb P^N\Big(\int_{\R^d} \phi(y)\big(\mu^N_{t_0}(dy)-\mu_{t_0}(y)dy\big) \geq x\Big) \leq  \kappa_1\exp\Big(-\kappa_2\frac{ Nx^2}{|\phi|_{L^2(\mu_{t_0})}^2+|\phi|_\infty x}\Big).$$ 
Several remarks: {\bf 1)} Up to the constants $\kappa_i$, the result is quite satisfactory and comparable to the Bernstein deviation inequality for independent data, see {\it e.g.} Massart \cite{massart2007concentration}. Theorem \ref{thm: concentration} is the gateway to derive sharp nonparametric estimators, although it has an independent interest as a deviation inequality. {\bf 2)} The constants $\kappa_i$ are explicitly computable, but certainly far from being optimal with the method of proof employed here. {\bf 3)} Our method of proof uses a change of measure argument based on Girsanov's theorem, in the spirit of the recent work of Lacker \cite{Lacker}. It can presumably be extended to path dependent coefficients in \eqref{eq:diff basique}, but it is essential that the diffusion coefficient does not depend on $\mu_t^N$. 
{\bf 4)} We have an interplay between the smoothness $k$ of the drift $b$ in its measure argument and the dimension $d$ of the ambient state space. This is explained by the fact that we need to control an exponential moment of $\sum_{i = 1}^N\int_t^{t+\delta}|b(s,X_s^i,\mu^N_s)-b(s,X_s^i,\mu_s)|^2ds$ over small intervals $[t, t+\delta]$ in order to approximate the law of the data by the law of independent particles. This approximation is roughly controlled by $N\mathcal W_1(\mu_t^N,\mu_t)^2$ for which a dimensional effect drastically deteriorates the rate of convergence, see {\it e.g.} Fournier and Guillin \cite{FournierGuillin}. The $k$-linear differentiability of $\mu \mapsto b(t,x,\mu)$ enables us to mitigate this effect. In particular, in the Vlasov case covered by Assumption \ref{ass: basic lip}  (iii), we formally have $k=\infty$, and the result is valid in any dimension $d \geq 1$.\\  

The remainder of Section \ref{sec: proof concentration} is devoted to the proof of Theorem \ref{thm: concentration}.

\subsection{Preparation for the proof of Theorem \ref{thm: concentration}} \label{sec: preparation}
%From Section \ref{subsec: concentration}, $X = (X^1,\ldots, X^N)$ denotes the canonical process on $(\mathcal C, \mathcal F_T)$. 
We let $\overline{\PP}^N$ denote the unique probability measure on $(\mathcal C, \mathcal F_T)$ under which the canonical process  $X = (X^1,\ldots, X^N)$  solves

\begin{equation} \label{eq: diff limite}
\left\{
\begin{array}{l}
dX_t^i = b(t,X_t^i,\mu_t)dt+\sigma(t,X_t^i) d\overline{B}_t^i,\;\;1 \leq i \leq N,\; t \in [0,T],\\  \\
\mathcal L(X_0^1,\ldots, X_0^N) = \mu_0^{\otimes N},
\end{array}
\right.
%dX_t^i = b(t,X_t^i,\mu_{t})dt+\sigma(t,X_t^i) d\overline{B}_t^i,\;\;1 \leq i \leq N,\\ 
\end{equation}
where 
$$
\overline{B}^i_t = \int_0^t c(s,X_s^i)^{-1/2}\big(dX_s^i-b(s,X_s^i,\mu_s) ds\big), \;\;1 \leq i \leq N,
$$
are  independent $d$-dimensional $\overline{\mathbb P}^{N}$-Brownian motions. The existence of $\overline{\PP}^N$ follows from Carmona and Delarue \cite{carmona2018probabilistic} or the lectures notes of Lacker  \cite{lacker2018mean}. In turn, 
the real-valued process 
\begin{equation} \label{def M bar}
\overline{M}^{N}_{t} = \sum_{i = 1}^N\int_0^{t}\big((c^{-1/2}b)(s,X_s^i,\mu^N_s)- (c^{-1/2}b)(s,X_s^i,\mu_s)\big)^\top d\overline{B}_s^i,
\end{equation}
is a $\overline{\mathbb P}^{N}$-local martingale. Here, $c^{-1/2}$ denotes any square-root of $c^{-1} = (\sigma\sigma^\top)^{-1}$.\\

The following estimate is the central result of the section. It is the key ingredient that enables us to implement a change of probability argument in order to obtain our concentration estimates.
Its proof is delayed until Section \ref{sec: proof of proposition fonda}.

\begin{prop} \label{prop: controle change of proba}
Work under Assumptions  \ref{ass: init condition}, \ref{ass: minimal prop sigma} and \ref{ass: basic lip}.
For every $\tau >0$, 
there exists $\delta_0 >0$ depending on $\tau$ and $\mathfrak b$ such that
$$\sup_{N \geq 1}\sup_{t \in [0,T-\delta]}\E_{\overline{\mathbb P}^{N}}\Big[\exp\big(\tau\big(\langle \overline{M}_\cdot^N \rangle_{t+\delta} - \langle \overline{M}_\cdot^N\rangle_t\big)\big)\Big] \leq C_1,
$$
for every $0 \leq  \delta \leq \delta_0$ and some $C_1$ that depends on $\mathfrak b$ and $\tau$. 
\end{prop}

Let  $\mathcal E_t(\overline{M}_\cdot^N) = \exp\big(\overline{M}^N_t-\tfrac{1}{2}\langle \overline{M}_.^N\rangle_t\big)$ denote the (martingale) exponential of $\overline{M}^N_t$ and $\langle \overline{M}^N_.\rangle_t$ its predictable compensator. By Novikov's criterion -- in its version developed in the classical textbook \cite{KS}, Lemma 5.14, p.198 -- Proposition \ref{prop: controle change of proba} shows that the local martingale $\mathcal E_t(\overline{M}_\cdot^{N})$ is indeed a true martingale. This enables us to define a new probability measure on $(\mathcal C, \mathcal F_T)$ by setting
$$\widetilde{\mathbb P}^{N}  = \mathcal E_T(\overline{M}^{N}_\cdot) \cdot \overline{\mathbb P}^{N}.$$
By Girsanov's theorem, under $\widetilde{\mathbb P}^{N} $, the canonical process solves \eqref{eq:diff basique}. By uniqueness of the weak solution of \eqref{eq:diff basique}, this proves $\widetilde \PP^N = \PP^N$ and shows in particular that $\PP^N \ll \overline{\PP}^N$ and
$$\frac{d\PP^N}{d\overline{\PP}^N} =  \mathcal E_T(\overline{M}^{N}_\cdot).$$

\subsection{Proof of Theorem \ref{thm: concentration}}
\noindent {\it Step 1:}
Let $\mathcal A^N \in \mathcal F_T$. Since $\PP^N$ and $\overline{\PP}^N$ coincide on $\mathcal F_0$, we have
$$\PP^N(\mathcal A^N)   = \E_{\PP^N}\big[\PP^N(\mathcal A^N\,|\,\mathcal F_0)\big]  =  \E_{\overline{\PP}^N}\big[\PP^N(\mathcal A^N\,|\,\mathcal F_0)\big].$$
Next, for any subdivision $0 = t_0 < t_1 < \ldots < t_K \leq T$ and any $\mathcal F_T$-measurable event $\mathcal A^N$, we claim 
\begin{equation} \label{eq: the claim}
 \E_{\overline{\PP}^N}\big[\PP^N(\mathcal A^N\,|\,\mathcal F_0)\big]  \leq \E_{\overline{\PP}^N} \big[\PP^N\big(\mathcal A^N\,\big|\,\mathcal F_{t_K}\big)\big]^{1/4^K}\prod_{j = 1}^K \E_{\overline{\PP}^N}\Big[\exp\big(2\big(\langle \overline{M}_\cdot^N\rangle_{t_j}-\langle \overline{M}_\cdot^N\rangle_{t_{j-1}}\big)\big)\Big]^{j/4}.
\end{equation}
It follows that 
\begin{align}
  \PP^N(\mathcal A^N)&   \leq \E_{\overline{\PP}^N}\big[\PP^N\big(\mathcal A^N\,\big|\,\mathcal F_{T}\big)\big]^{1/4^K}\prod_{j = 1}^K \E_{\overline{\PP}^N}\Big[\exp\big(2\big(\langle \overline{M}_\cdot^N\rangle_{t_j}-\langle \overline{M}_\cdot^N\rangle_{t_{j-1}}\big)\big)\Big]^{j/4} \nonumber \\
& \leq  \overline{\mathbb P}^N\big(\mathcal A^N\big)^{1/4^K}\sup_{N \geq 1}\sup_{t \in [0,T-\delta_0]}\Big(\E_{\overline{\mathbb P}^{N}}\Big[\exp\big(2\big(\langle \overline{M}_\cdot^N \rangle_{t+\delta_0} - \langle \overline{M}_\cdot^N\rangle_t\big)\big)\Big]\Big)^{K(K+1)/8} \nonumber \\
& \leq C_1^{K(K+1)/8} \,\overline{\mathbb P}^N\big(\mathcal A^N\big)^{1/4^K} \label{eq: first estimate changproba}
\end{align}
by \eqref{eq: the claim} and Proposition \ref{prop: controle change of proba} with $\tau = 2$, $t_j = jT/K$ and $K$ large enough so that $t_j-t_{j-1} \leq \delta_0$.\\

\noindent {\it Step 2:} Let
\begin{align*}
\mathcal {A}^N & = \Big\{\int_{[0,T] \times \R^d} \phi(t,y)\big(\nu^N(dt,dy)-\nu(dt,dy)\big) \geq x\Big\} \\
&= \Big\{\sum_{i = 1}^N \Big(\int_0^T\phi(t,X_t^i)\rho(dt)-\int_{[0,T]\times \R^d}\phi(t,y)\mu_t(dy)\rho(dt)\Big) \geq Nx\Big\},
\end{align*}
so that $\mathcal {A}^N \in \mathcal F_T$. Recall Bernstein's inequality: if $Z_1, \ldots, Z_N$ are real-valued independent random variables bounded by some constant $Q$ and such that $\E[Z_i]=0$, we have
$$\PP\big(\sum_{i = 1}^NZ_i \geq y\big) \leq \exp\Big(-\frac{y^2}{2(\sum_{i = 1}^N\E[Z_i^2]+\frac{Qy}{3})}\Big)\;\;\text{for every}\;\;y \geq 0.$$
Under $\overline{\PP}^{N}$, the random processes $(X_t^i)_{0 \leq t \leq T}$ are independent and identically distributed processes. 
Noticing that  $ \int_{[0,T] \times \R^d} \phi(t,x)\mu_t(dx)\rho(dt) = \E_{\overline{\PP}^{N}}[\int_0^T\phi(t,X_t^i)\rho(dt)]$, we apply Bernstein inequality with $Z_i = \int_0^T\phi(t,X_t^i)\rho(dt)-\E_{\overline{\PP}^{N}}[\int_0^T\phi(t,X_t^i)\rho(dt)]$,
$y=Nx$ and $Q = 2|\phi|_{\infty}$ to infer
$$\overline{\PP}^{N}\big(\mathcal {A}^N \big) \leq \exp\Big(-\frac{Nx^2}{2(|\phi|^2_{L^2(\nu)}+\frac{2}{3}|\phi|_{\infty}x)}\Big),$$
using $\E[Z_i^2] \leq  \E_{\overline{\PP}^{N}}\big[\big(\int_0^T\phi(t,X_t^i)\rho(dt)\big)^2\big] \leq |\phi|_{L^2(\nu)}^2$ by Jensen's inequality.
By \eqref{eq: first estimate changproba}, we also have
$$\PP^N\big(\mathcal {A}^N\big) \leq  C_1^{K(K+1)/8} \,\overline{\mathbb P}^N\big(\mathcal A^N\big)^{1/4^K} \leq C_1^{K(K+1)/8}\exp\Big(-\frac{Nx^2}{2\cdot4^K\big(|\phi|^2_{L^2(\nu)}+\frac{2}{3}|\phi|_{\infty}x\big)}\Big)$$
and we  obtain Theorem \ref{thm: concentration} with $\kappa_1=  C_1^{K(K+1)/8}$ and $\kappa_2=2^{-1}4^{-K}$.\\

\noindent {\it Step 3}: It remains to prove the key estimate \eqref{eq: the claim}, adapted from large deviation techniques, see {\it e.g.} G\"artner \cite{Ga88} and the estimate (4.2) in Theorem 2.6 in Lacker \cite{Lacker}.
We proceed by induction. First,
\begin{align*}
\E_{\overline{\PP}^N}\big[\mathbb P^N\big(\mathcal A^N\,\big|\,\mathcal F_{t_{j-1}}\big)\big] & = \E_{\overline{\PP}^N}\Big[\E_{\mathbb P^N}\big[\mathbb P^N\big(\mathcal A^N\,\big|\,\mathcal F_{t_j}\big)\,\big|\,\mathcal F_{t_{j-1}}\big]\Big] \\
& =  \E_{\overline{\PP}^N}\Big[\E_{\overline{\mathbb P}^N}\Big[\frac{\mathcal E_{t_j}(\overline{M}_\cdot^N)}{\mathcal E_{t_{j-1}}(\overline{M}_\cdot^N)}\mathbb P^N\big(\mathcal A^N\,\big|\,\mathcal F_{t_j}\big)\,\big|\,\mathcal F_{t_{j-1}}\Big]\Big] \\
& = \E_{\overline{\mathbb P}^N}\Big[\frac{\mathcal E_{t_j}(\overline{M}_\cdot^N)}{\mathcal E_{t_{j-1}}(\overline{M}_\cdot^N)}\mathbb P^N\big(\mathcal A^N\,\big|\,\mathcal F_{t_j}\big)\Big],
\end{align*}
see {\it e.g.} Lemma 3.5.3 p. 193 in \cite{KS}.
Next, 
\begin{equation} \label{eq: weird decomp}
\frac{\mathcal E_{t_j}(\overline{M}_\cdot^N)}{\mathcal E_{t_{j-1}}(\overline{M}_\cdot^N)} = \mathcal E_{t_j}\big(2(\overline{M}_\cdot^N-\overline{M}_{t_{j-1}}^N)\big)^{1/2}\big(\exp\big(\langle \overline{M}_\cdot^N\rangle_{t_j}-\langle \overline{M}_\cdot^N\rangle_{t_{j-1}}\big)\big)^{1/2}.
\end{equation}
As shown before, 
%We then use a classical localisation argument: 
under  $\overline{\PP}^N$, the process $\mathcal E_{t}\big(2(\overline{M}_\cdot^N-\overline{M}_{t_{j-1}}^N)\big)_{t \geq t_{j-1}}$ is a
%local 
 martingale and 
% the localising sequence
%$$\tau_m = \inf\big\{t \geq t_{j-1}, \big|\overline{M}_{t}^N-\overline{M}_{t_{j-1}}^N\big| \geq m\big\},\;\;m\geq 1,$$
%combined with Fatou's lemma yields
%\begin{equation} \label{eq: fatou}
%\E_{\overline{\PP}^N}\big[\mathcal E_{t_j}\big(2(\overline{M}_{\cdot }^N-\overline{M}_{t_{j-1}}^N)\big)\big] \leq \liminf_{m \rightarrow \infty}\E_{\overline{\PP}^N}\big[\mathcal E_{t_j}\big(2(\overline{M}_{\cdot \wedge \tau_m}^N-\overline{M}_{t_{j-1}}^N)\big)\big] = 1.
%\end{equation}  
\begin{equation} \label{eq: fatou}
\E_{\overline{\PP}^N}\big[\mathcal E_{t_j}\big(2(\overline{M}_{\cdot }^N-\overline{M}_{t_{j-1}}^N)\big)\big]=1.
\end{equation}
Using \eqref{eq: weird decomp} and Cauchy-Schwarz's inequality twice together with \eqref{eq: fatou}, we obtain
\begin{align*}
\E_{\overline{\mathbb P}^N}\Big[\frac{\mathcal E_{t_j}(\overline{M}_\cdot^N)}{\mathcal E_{t_{j-1}}(\overline{M}_\cdot^N)}\mathbb P^N\big(\mathcal A^N\,\big|\,\mathcal F_{t_j}\big)\Big] & \leq
\E_{\overline{\mathbb P}^N}\Big[\mathbb P^N\big(\mathcal A^N\,\big|\,\mathcal F_{t_j}\big)^2\exp\big(\langle \overline{M}_\cdot^N\rangle_{t_j}-\langle \overline{M}_\cdot^N\rangle_{t_{j-1}}\big)\Big]^{1/2} \\
& \leq \E_{\overline{\mathbb P}^N}\big[\mathbb P^N\big(\mathcal A^N\,\big|\,\mathcal F_{t_j}\big)^4\big]^{1/4}\E_{\overline{\mathbb P}^N}\Big[\exp\big(2(\langle \overline{M}_\cdot^N\rangle_{t_j}-\langle \overline{M}_\cdot^N\rangle_{t_{j-1}})\big)\Big]^{1/4}.
\end{align*}
By Jensen's inequality, we infer
$$\E_{\overline{\PP}^N}\big[\mathbb P^N\big(\mathcal A^N\,\big|\,\mathcal F_{t_{j-1}}\big)\big] \leq \E_{\overline{\mathbb P}^N}\big[\mathbb P^N\big(\mathcal A^N\,\big|\,\mathcal F_{t_j}\big)\big]^{1/4}\,\E_{\overline{\mathbb P}^N}\Big[\exp\big(2(\langle \overline{M}_\cdot^N\rangle_{t_j}-\langle \overline{M}_\cdot^N\rangle_{t_{j-1}})\big)\Big]^{1/4}.$$
Repeating the argument over the subdivision $0 = t_0 < t_1 < \ldots  < t_K$ proves \eqref{eq: the claim}. The proof of Theorem \ref{thm: concentration} is complete.

\subsection{Proof of Proposition \ref{prop: controle change of proba}}  \label{sec: proof of proposition fonda}
\subsubsection*{Preparation}
Recall that the notation $A_N \lesssim B_N$: it means the existence of $C>0$ possibly depending on $\mathfrak b$ and also $\tau$ in this part of the paper, but not $N$, such that $A_N \leq CB_N$ for every $N \geq 1$. 
The following classical moment estimate will be needed.
\begin{lem} \label{lem moment class}
In the setting of Theorem \ref{thm: concentration}, for every $p\geq 1$, we have
%\begin{equation} \label{eq: moment bound bar X}
$$
\sup_{t \in [0,T]}\E_{\overline{\mathbb P}^N}\big[\big|X_t^i\big|^{2p}\big] \leq p!\, C_2^p.
$$
for some $C_2>0$ that depends on $\mathfrak b$ only.
\end{lem}
In particular, the $X_t^i$ are sub-Gaussian under $\overline{\mathbb P}^N$ and satisfy
\begin{equation} \label{eq: cond FG}
\E_{\overline{\mathbb P}^N}\big[\mathrm{e}^{\tfrac{1}{2C_2}|X_t^i|^{2}}\big]  = 1+ \sum_{p \geq 1}\frac{2^{-p}}{p! C_2^p}\E_{\overline{\mathbb P}^N}\big[\big|X_t^i\big|^{2p}\big] \leq 2.
% \int_{\R^d}\mathrm{e}^{\tfrac{1}{2C_2}|x|^{2}}\mu_t(dx)  \leq 2,\;\;t\in [0,T],
\end{equation}
 %since $\mathrm{Law}(X_t^i\,|\,\overline{\PP}^N) = \mu_t$. 
The proof is classical (see {\it e.g.} estimates of this type in M\'el\'eard \cite{meleard1996asymptotic} or Sznitman \cite{Sznitman2}) and postponed to Appendix \ref{proof lem exp}.

\subsubsection*{Completion of Proof of Proposition \ref{prop: controle change of proba}}
Let $\tau >0$. All we need to show is that for small enough $\delta$, we have
\begin{equation} \label{eq: prop fund}
 \sup_{t \in [0,T-\delta]}\E_{\overline{\PP}^N}\Big[\exp\big(\tau\big(\langle \overline{M}_\cdot^N\rangle_{t+\delta}-\langle \overline{M}_\cdot^N\rangle_{t}\big)\big)\Big] \lesssim 1.
 \end{equation}
We start with a useful estimate. 
\begin{lem} \label{lem: dec diff drifts} Let $\xi_s^N(x) = b(s,x,\mu^N_s)- b(s,x,\mu_s)$. 
We have
$$
\big|\xi_s^N(X_s^i)\big|^2 \leq C
 \left\{
\begin{array}{ll}
 \mathcal W_1(\mu_s^N,\mu_s)^2 & under\; Assumption\; \ref{ass: basic lip} \mathrm{(i)},\\ \\
\sum_{\ell=1}^{k-1} \big|\int_{(\R^d)^{\ell}}  \delta_\mu^\ell b(s,X_s^i,y^{\ell}, \mu_s)(\mu_s^N-\mu_s)^{\otimes \ell}(dy^{\ell})\big|^2 &  \\
+\mathcal W_1(\mu_s^N, \mu_s)^2 \wedge  \mathcal W_1(\mu_s^N, \mu_s)^{2k} & under\;Assumption\; \ref{ass: basic lip} \mathrm{(ii)},\\ \\
\big|\int_{\R^d} \widetilde b(s,X_s^N,y)(\mu_s^N-\mu_s)(dy)\big|^2 &  under\; Assumption\; \ref{ass: basic lip} \mathrm{(iii)},
\end{array}
\right.
$$
for some explicit $C \geq 1$ depending on $\mathfrak b$.
\end{lem}
\begin{proof} The estimate follows from the Lipschitz continuity of $b$ under Assumption\; \ref{ass: basic lip} (i)  with $C=|b|_{\mathrm{Lip}}^2$ and from the definition of the Vlasov case  with $C = 1$ under Assumption\; \ref{ass: basic lip} (iii). We turn to the estimate under (ii).
The $k$-linear differentiability of $b$ enables us to write
\begin{equation} \label{eq: taylor linear diff} 
\xi_s^N(X_s^i) = \sum_{\ell = 1}^{k-1} \frac{1}{{\ell !}}\int_{(\R^d)^{\ell}}  \delta_\mu^\ell b(s,X_s^i,y^{\ell}, \mu_s)(\mu_s^N-\mu_s)^{\otimes \ell}(dy^{\ell}) + \mathcal R_k,
\end{equation}
where $y^{\ell} = (y_1,\ldots, y_\ell) \in (\R^d)^\ell$ and 
$$\mathcal R_k = \frac{1}{(k-1)!}\int_0^1(1-\vartheta)^{k-1}\int_{(\R^d)^k}\delta_\mu^k b(s,X_s^i,y^{k}, [\mu_s^N,\mu_s]_\vartheta)(\mu_s^N-\mu_s)^{\otimes k}(dy^k)d\vartheta,$$
% =  &\;\frac{1}{(k-1)!}\int_0^1(1-\vartheta)^{k-1}\sum_{\mathcal I \subset \{1,\ldots, k\}, m \geq 1} \int_{(\R^d)^k}\bigotimes_{j \in \mathcal I}(\delta_\mu^k b)_{\math$$
%\mathcal R_k(s,X_s^i,\mu_s^N,\mu_s) = \int_{(\R^d)^k}\psi(s,X_s^i,y^k,
%\mu_s^N,\mu_s)(\mu_s^N-\mu_s)^{\otimes k}(dy^k),$$
%with
%$$\psi(s,X_s^i, y^k,\mu_s^N,\mu_s)  =  \frac{1}{(k-1)!}\int_0^1(1-\vartheta)^{k-1}\delta_\mu^k b(s,X_s^i,y^{k}, [\mu_s^N,\mu_s]_\vartheta)d\vartheta,$$
with $[\mu_s^N,\mu_s]_\vartheta = (1-\vartheta)\mu_s+\vartheta \mu_s^N$, as stems from the definition of linear differentiability and the iteration $\delta_\mu^\ell b = \delta_\mu \circ \delta_\mu^{\ell-1}b$, see also Lemma 2.2. of  Chassagneux {\it et al.} \cite{Chassagneux} where \eqref{eq: taylor linear diff} is established by induction.\\ 

%In view of the statement of Lemma \ref{lem: dec diff drifts}, we need to prove 
%\begin{equation} \label{eq: controle wass remainder}
%|\mathcal R_k(s,X_s^i,\mu_s^N,\mu_s)|^2 \lesssim \mathcal W_1(\mu_s^N, \mu_s)^2 \wedge  \mathcal W_1(\mu_s^N, \mu_s)^{2k}.
%\end{equation}

Thanks to representation \eqref{eq: rep remainder b}, the remainder term $\mathcal R_k$ equals
\begin{align*}
%&  \;\frac{1}{(k-1)!}\int_0^1(1-\vartheta)^{k-1}\int_{(\R^d)^k}\delta_\mu^k b(s,X_s^i,y^{k}, [\mu_s^N,\mu_s]_\vartheta)(\mu_s^N-\mu_s)^{\otimes k}(dy^k)d\vartheta \\
\frac{1}{(k-1)!}\int_0^1(1-\vartheta)^{k-1}\sum_{\mathcal I \subset \{1,\ldots, k\}, m \geq 1} \int_{(\R^d)^k}\bigotimes_{j \in \mathcal I}(\delta_\mu^k b)_{\mathcal I, j, m}(s,X_s^i,y_j, [\mu_s^N,\mu_s]_\vartheta)(\mu_s^N-\mu_s)^{\otimes k}(dy^k)d\vartheta.
\end{align*}
Note that the product integral vanishes for all terms in the sum in $\mathcal I$ except $\mathcal I=\{1,\ldots, k\}$ since $(\mu_s^N-\mu_s)(\R^d)=0$. By definition,
$$|(\delta_\mu^k b)_{\{1,\ldots, k\}, j, m}(s,X_s^i,\cdot, [\mu_s^N,\mu_s]_\vartheta)|_{\mathrm{Lip}} \leq |\delta_\mu^kb|_{\mathrm{Lip}}\;\;\text{for every}\;\;(j,m)$$ 
and the sum in $m$ has at most $m_b$ terms by assumption. It follows that 
\begin{align*}
\big| \mathcal R_k \big| & =  \frac{1}{(k-1)!}\big|\int_0^1(1-\vartheta)^{k-1}\sum_{m =1}^{m_b}\prod_{j = 1}^k\int_{\R^d}(\delta_\mu^k b)_{\{1,\ldots, k\}, j, m}(s,X_s^i,y_j, [\mu_s^N,\mu_s]_\vartheta)(\mu_s^N-\mu_s)(dy_j) d \vartheta\big|\\
& \leq  \frac{m_b}{(k-1)!}\big|\int_0^1(1-\vartheta)^{k-1} |\delta_\mu^k b|_{\mathrm{Lip}}^k\Big(\sup_{|\varphi|_{\mathrm{Lip}} \leq 1} \int_{\R^d}\varphi \,d(\mu_s^N-\mu_s)\Big)^k d\vartheta\big|\\
&  \leq \frac{m_b  |\delta_\mu^k b|_{\mathrm{Lip}}^k}{k!}\mathcal W_1(\mu_s^N,\mu_s)^k.
\end{align*}
Writing $y^\ell = (y^{\ell-1},y) \in (\R^d)^{\ell-1}\times \R^d$,  
we also have the rough bound
\begin{align*}
& \big| \int_{(\R^d)^{\ell}}\delta_\mu^\ell b(s,X_s^i,y^{\ell}, \mu_s)(\mu_s^N-\mu_s)^{\otimes \ell}(dy^{\ell})\big| \\
&\leq  \int_{(\R^d)^{\ell-1}}\big| \int_{\R^d}\delta_\mu^\ell b(s,X_s^i,(y^{\ell-1},y), \mu_s)(\mu_s^N-\mu_s)(dy)\big| (\mu_s^N+\mu_s)^{\otimes (\ell -1)}(dy^{\ell-1}) \\
&\leq \sup_{y^{\ell-1} \in (\R^d)^{\ell-1}}\big| \delta_\mu^\ell b(s,X_s^i,(y^{\ell-1},\cdot), \mu_s)\big|_{\mathrm{Lip}} \\
&\hspace{3mm}\times \big| \sup_{|\varphi|_{\mathrm{Lip}}\leq 1}\int_{\R^d}\varphi(y)(\mu_s^N-\mu_s)(dy)\big|  \int_{(\R^d)^{\ell-1}}(\mu_s^N+\mu_s)^{\otimes (\ell -1)} (dy^{\ell-1}) \\
& = 2^{\ell-1}\sup_{y^{\ell-1} \in (\R^d)^{\ell-1}}\big| \delta_\mu^\ell b(s,X_s^i,(y^{\ell-1},\cdot), \mu_s)\big|_{\mathrm{Lip}}\mathcal W_1(\mu_s^N, \mu_s)  \\
&\leq 2^{\ell-1} |\delta_\mu^\ell b|_{\mathrm{Lip}}\mathcal W_1(\mu_s^N, \mu_s). 
\end{align*}
Plugging this estimate in \eqref{eq: taylor linear diff} and using the Lipschitz property for $b$, we obtain
$$\big|\mathcal R_k\big| \leq  \big(|b|_{\mathrm{Lip}}+\sum_{\ell=1}^{k-1} \frac{2^{\ell-1}}{\ell !} |\delta_\mu^\ell b|_{\mathrm{Lip}} \big) \mathcal W_1(\mu_s^N, \mu_s)$$
and we conclude
\begin{equation} \label{eq control remainder}
\big|\mathcal R_k\big| = \big|\mathcal R_k(s,X_s^i,\mu_s^N,\mu_s)  \big|  \leq  C'\mathcal W_1(\mu_s^N, \mu_s) \wedge \mathcal W_1(\mu_s^N, \mu_s)^k, 
\end{equation} 
with 
$$C'= \max\Big(\frac{m_b}{k!}|\delta_\mu^k b|_{\mathrm{Lip}}^k, |b|_{\mathrm{Lip}}+\sum_{\ell=1}^{k-1} \frac{2^{\ell-1}}{\ell !} |\delta_\mu^\ell b|_{\mathrm{Lip}}\Big).$$
From \eqref{eq: taylor linear diff} and \eqref{eq control remainder} we conclude
\begin{align*}
|\xi_s^N(X_s^i)|^2 
&\leq k\big(\sum_{\ell=1}^{k-1} \frac{1}{(\ell !)^2}\big|\int_{(\R^d)^{\ell}}  \delta_\mu^\ell b(s,X_s^i,y^{\ell}, \mu_s)(\mu_s^N-\mu_s)^{\otimes \ell}(dy^{\ell})\big|^2+\big|\mathcal R_k(s,X_s^i,\mu_s^N,\mu_s)  \big|^2\big) \\
& \leq C\big(\sum_{\ell=1}^{k-1} \big|\int_{(\R^d)^{\ell}}  \delta_\mu^\ell b(s,X_s^i,y^{\ell}, \mu_s)(\mu_s^N-\mu_s)^{\otimes \ell}(dy^{\ell})\big|^2+\mathcal W_1(\mu_s^N, \mu_s)^2 \wedge  \mathcal W_1(\mu_s^N, \mu_s)^{2k}\big),
\end{align*}
where $C = k\max(1, (C')^2)$ incorporates the constant in \eqref{eq control remainder}.
\end{proof}

We now establish \eqref{eq: prop fund}. By \eqref{def M bar} and Lemma \ref{lem: dec diff drifts}, we have
\begin{align}
&\tau\big(\langle \overline{M}_\cdot^N\rangle_{t+\delta}-\langle \overline{M}_\cdot^N\rangle_{t}\big) \nonumber \\
& = 
\tau\sum_{i = 1}^N\int_{t}^{t+\delta} \big(b(s,X_s^i,\mu^N_s)- b(s,X_s^i,\mu_s)\big)^\top c^{-1}\big(b(s,X_s^i,\mu^N_s)- b(s,X_s^i,\mu_s)\big)ds \nonumber \\
& \leq \tau| \mathrm{Tr}(c^{-1})|_\infty \sum_{i = 1}^N \int_{t}^{t+\delta} |\xi_s^N(X_s^i)|^2ds, \nonumber \\
& \leq \kappa
 \left\{
\begin{array}{ll}
N\int_t^{t+\delta} \mathcal W_1(\mu_s^N,\mu_s)^2ds & \mathrm{under}\;\ref{ass: basic lip} \mathrm{(i)},\\ \\
\int_t^{t+\delta} \sum_{i = 1}^N\sum_{\ell=1}^{k-1} \big|\int_{(\R^d)^{\ell}}  \delta_\mu^\ell b(s,X_s^i,y^{\ell}, \mu_s)(\mu_s^N-\mu_s)^{\otimes \ell}(dy^{\ell})\big|^2ds &  \\
+N\int_t^{t+\delta}\mathcal W_1(\mu_s^N, \mu_s)^2 \wedge  \mathcal W_1(\mu_s^N, \mu_s)^{2k} ds& \mathrm{under}\;\ref{ass: basic lip} \mathrm{(ii)},\\ \\
\int_t^{t+\delta} \sum_{i = 1}^N\big|\int_{\R^d} \widetilde b(s,X_s^N,y)(\mu_s^N-\mu_s)(dy)\big|^2 ds,&  \mathrm{under}\;\ref{ass: basic lip} \mathrm{(iii)}. \label{eq: the big maj}
\end{array}
\right.
\end{align}
with  $\kappa = \tau | \mathrm{Tr}(c^{-1})|_\infty C$, where $C$ is the constant of Lemma \ref{lem: dec diff drifts}. We now heavily rely on the sharp deviation estimate
\begin{equation} \label{eq: fournier guillin}
\sup_{0 \leq s \leq T}\overline{\mathbb P}^{N}(\mathcal W_1\big(\mu^N_s,\mu_s \big) \geq x) \lesssim \varepsilon_N(x),
\end{equation}
with
\begin{align}
 \varepsilon_N(x) =\left\{
\begin{array}{lll}
\exp(-\mathfrak CNx^2) & \mathrm{if} & d=1,\nonumber \\ 
\exp(-\mathfrak CN\tfrac{x^2}{(\log (2+1/x))^2}\big){\bf 1}_{\{x \leq 1\}}+\exp(-\mathfrak CNx^2){\bf 1}_{\{x >1\}} & \mathrm{if} & d=2, \\  \label{eq: FourGhill}
\exp(-\mathfrak CNx^d){\bf 1}_{\{x \leq 1\}}+\exp(-\mathfrak CNx^2){\bf 1}_{\{x >1\}} & \mathrm{if} & d\geq 3, \nonumber
 \end{array}
\right.
\\
\end{align}
extracted from Theorem 2 of Fournier and Guillin \cite{FournierGuillin}. Here $\mathfrak C$ depends on $C_2$ and $d$ only, thanks to \eqref{eq: cond FG} that guarantees that Condition (1) of Theorem 2 in \cite{FournierGuillin} is satisfied, hence the uniformity in $s \in [0,T]$.\\

We complete the proof of \eqref{eq: prop fund} under Assumption \ref{ass: basic lip} (i) that implies in particular $d=1$.
From \eqref{eq: the big maj}, we infer
\begin{align}
& \E_{\overline{\PP}^N}\big[\exp\big(\tau\big(\langle \overline{M}_\cdot^N\rangle_{t+\delta}-\langle \overline{M}_\cdot^N\rangle_{t}\big)\big)\big]  \nonumber \\
& \leq \delta^{-1}\int_t^{t+\delta}\E_{\overline{\PP}^N}\big[\exp\big(\kappa \delta N\mathcal W_1(\mu_s^N,\mu_s)^2\big)\big]ds \nonumber \\
& \leq  \sup_{s \in [0,T]}\E_{\overline{\PP}^{N}}\Big[\exp\big(\kappa \delta N\mathcal W_1(\mu^N_s,\mu_s)^2  \big)\Big] \nonumber\\
& \leq 1+\kappa \delta \sup_{s \in [0,T]}\int_0^\infty\exp(\kappa \delta  z)\overline{\mathbb P}^{N}\big(\mathcal W_1(\mu^N_s,\mu_s) \geq N^{-1/2}z^{1/2}\big)dz \nonumber\\
& \lesssim \int_0^\infty \exp\big((\kappa \delta- \mathfrak C)z\big)dz, \label{eq: integral dimen 1}
\end{align}
where 
$\mathfrak C$ is the constant in \eqref{eq: FourGhill}. The integral in \eqref{eq: integral dimen 1} is finite as soon as 
$\delta \leq \tau^{-1}|\mathrm{Tr}(c^{-1})|_\infty^{-1}|b|_{\mathrm{Lip}}^{-2}\mathfrak C$ and \eqref{eq: prop fund} follows.\\
% and we obtain Proposition \ref{prop: controle change of proba} under Assumption \ref{ass: basic lip} (i).

We next complete the proof of \eqref{eq: prop fund} under Assumption \ref{ass: basic lip} (ii).  When $d=1$, we can rely on the previous case. Assume now that ($d=2$ and $k \geq 2$) or ($d \geq 3$ and $k \geq d/2$).
By Jensen's inequality
$$
\E_{\overline{\PP}^N}\big[\exp\big(\tau\big(\langle \overline{M}_\cdot^N\rangle_{t+\delta}-\langle \overline{M}_\cdot^N\rangle_{t}\big)\big)\big] \leq I+II,
$$
with
\begin{align*}
I & = \frac{1}{\delta k N}\int_t^{t+\delta}\sum_{i = 1}^N\sum_{\ell = 1}^{k-1}\E_{\overline{\PP}^N}\big[\exp\big(\kappa\delta kN\big|\int_{(\R^d)^{\ell}}  \delta_\mu^\ell b(s,X_s^i,y^{\ell}, \mu_s)(\mu_s^N-\mu_s)^{\otimes \ell}(dy^{\ell})\big|^2\big)\big]ds, \\
II & =  \frac{1}{\delta k}\int_t^{t+\delta}\E_{\overline{\PP}^N}\big[\exp\big(\kappa\delta kN\mathcal W_1(\mu_s^N, \mu_s)^2 \wedge  \mathcal W_1(\mu_s^N, \mu_s)^{2k}\big)\big]ds.
\end{align*}
We first estimate the remainder term $II$: by inequality \eqref{eq: fournier guillin}, we have
\begin{align*}
II  \leq \frac{1}{k}\Big(1+& \sup_{t \in [0,T]}\kappa \delta k\int_0^{\infty}  \mathrm{e}^{\kappa \delta kz} \overline{\PP}^{N}\big(N\mathcal W_1(\mu_t^N, \mu_t)^2 \wedge  \mathcal W_1(\mu_t^N, \mu_t)^{2k}\geq z\big)dz\Big) \\
& \lesssim 1+\sup_{t \in [0,T]}\int_0^N \mathrm{e}^{\kappa \delta kz} \overline{\PP}^{N}\big(\mathcal W_1(\mu^N_t,\mu_t)\geq  N^{-1/(2k)}z^{1/(2k)}\big)dz \\
&+ \sup_{t \in [0,T]}\int_N^\infty \mathrm{e}^{\kappa \delta kz} \overline{\mathbb P}^{N}\big(\mathcal W_1(\mu^N_t,\mu_t)\geq N^{-1/2}z^{1/2}\big)dz. 
\end{align*}

%as soon as $2\delta \kappa \leq \mathfrak C$. 
We first estimate the integral over $[0,N]$:
% With the notation $\phi(x) = (\log(2+x^{-1}))^{{\bf 1}_{\{d=2\}}}$, we have 
%$\varepsilon_N(x){\bf 1}_{\{x \leq 1\}} \leq \exp(-\mathfrak CNx^d \phi(x)^{-2})$ for $d \geq 2$ by the estimate \eqref{eq: fournier guillin}. It follows that 
\begin{align*}
 \int_0^N \mathrm{e}^{\kappa \delta kz} 
\overline{\mathbb P}^{N}\big(\mathcal W_1(\mu^N_t,\mu_t)\geq   N^{-1/(2k)}z^{1/(2k)}\big)dz 
& \lesssim \int_0^N \exp\big(\kappa \delta kz-\mathfrak CN^{1-d/(2k)}z^{d/(2k)}\big)dz \\
& = N  \int_0^1 \exp\big(N(\kappa \delta kz-\mathfrak Cz^{d/(2k)})\big)dz \lesssim 1
\end{align*}
for $2 < d \leq 2k$ as soon as $\delta \leq k^{-1}\kappa^{-1}\mathfrak C$. The case ($d=2$ and $k \geq 2$) is slightly more technical but elementary and we omit it.
For the integral over $[N,\infty)$, we proceed as under Assumption \ref{ass: basic lip} (i) to obtain
$$\int_N^\infty \mathrm{e}^{\kappa \delta kz} \overline{\mathbb P}^{N}\big(\mathcal W_1(\mu^N_t,\mu_t)\geq N^{-1/2}z^{1/2}\big)dz \lesssim \int_0^\infty \exp\big((\kappa \delta k -\mathfrak C)z\big)dz \lesssim 1$$
for $\delta \leq k^{-1}\kappa^{-1}\mathfrak C$ and we conclude $II \lesssim 1$ in that case.\\

We next turn to the term $I$. Observe first that by exchangeability
\begin{align*}
I & \leq \sup_{1 \leq \ell \leq k-1, t \in [0,T]} \E_{\overline{\PP}^N}\big[\exp\big(\kappa\delta kN\big|\int_{(\R^d)^{\ell}}  \delta_\mu^\ell b(t,X_t^N,y^{\ell}, \mu_t)(\mu_t^N-\mu_t)^{\otimes \ell}(dy^{\ell})\big|^2\big)\big] \\
 & = 1+\sum_{p \geq 1}\frac{(\kappa\delta kN)^p}{p!}\sup_{1 \leq \ell \leq k-1, t \in [0,T]} \E_{\overline{\PP}^N}\Big[\big|\int_{(\R^d)^{\ell}}  \delta_\mu^\ell b(t,X_t^N,y^{\ell}, \mu_t)(\mu_t^N-\mu_t)^{\otimes \ell}(dy^{\ell})\big|^{2p}\Big].
\end{align*}
We then use the following estimate, reminiscent of moment bounds for $U$-statistics, however in a weaker and simpler form in our context. For an integer $\ell \geq 1$, we call $\mathcal G_\ell$ the class of functions $f :[0,T]\times \R^d \times (\R^d)^\ell \rightarrow \R^d$ 
that are Lipschitz continuous in the space variables.

% and bounded in its last two arguments. 
\begin{lem} \label{lem rosenthal}
Let $p \geq 1$. For $1 \leq \ell \leq k$, $f \in \mathcal G_\ell$ and $N \geq k+1$, we have
\begin{equation} \label{eq: rosen}
\mathcal V_{2p,\ell}^N\big(f(t,\cdot)\big)  = \E_{\overline{\PP}^N}\Big[\big|\int_{(\R^d)^{\ell}}  f(t,X_t^N, y^{\ell})(\mu_t^N-\mu_t)^{\otimes \ell}(dy^{\ell})\big|^{2p}\Big]\leq \frac{p! K_\ell^p}{(N-k)^p}|f(t,\cdot)|_{\mathrm{Lip}}^{2p},
\end{equation}
for some explicitly computable $K_\ell = K_\ell(\mathfrak b) >0$.
\end{lem}

The proof of Lemma \ref{lem rosenthal} is quite elementary, yet technical, and is delayed until Appendix \ref{sec: proof of lemma rosen}.
The remainder of the proof of \eqref{eq: prop fund} is then straightforward: by Lemma \ref{lem rosenthal} with $f(t,x,y^\ell) =  \delta_\mu^\ell b(t,x,y^{\ell}, \mu_t),$ it follows that
$$I \leq 1+\sum_{p \geq 1}(\kappa\delta kN)^p\frac{K_\ell^p}{(N-k)^p}\sup_{t \in [0,T]}|\delta_{\mu}^{\ell}b(t,\cdot, \mu_t)|_{\mathrm{Lip}}^{2p} \lesssim 1$$
as soon as $\delta < \kappa^{-1}k^{-1}(k+1)^{-1}K_{\ell}^{-1}\sup_{t \in [0,T]}|\delta_{\mu}^{\ell}b(t,\cdot, \mu_t)|_{\mathrm{Lip}}^{-2}$.
Since $II \lesssim 1$ is established as well, we obtain Proposition \ref{prop: controle change of proba} under Assumption \ref{ass: basic lip} (ii) provided Lemma \ref{lem rosenthal} is proved.\\

We finally prove \eqref{eq: prop fund} under Assumption  \ref{ass: basic lip} (iii) when $d \geq 1$ is arbitrary.
By \eqref{eq: the big maj} and Jensen's inequality together with the exchangeability  of $(X_t^i)_{1 \leq i \leq N}$, we have
\begin{align*}
&\E_{\overline{\PP}^N}\big[\exp\big(\tau\big(\langle \overline{M}_\cdot^N\rangle_{t+\delta}-\langle \overline{M}_\cdot^N\rangle_{t}\big)\big)\big] \\
& \leq \frac{1}{\delta}\int_t^{t+\delta}\E_{\overline{\PP}^N}\big[\exp\big(\kappa\delta \sum_{i = 1}^N\big|\int_{(\R^d)^{\ell}} \widetilde b(s,X_s^i,y)(\mu_s^N-\mu_s)(dy)\big|^2\big)\big]ds \\
&  \leq \sup_{s \in [0,T]}\E_{\overline{\PP}^N}\big[\exp\big(\kappa\delta N\big|\int_{\R^d} \widetilde b(s,X_s^N,y)(\mu_s^N-\mu_s)(dy)\big|^2\big)\big] \\
& \leq 1+\sum_{p \geq 1}\frac{(\kappa\delta N)^p}{p!}\sup_{s \in [0,T]} \E_{\overline{\PP}^N}\Big[\big|\int_{\R^d} \widetilde b(s,X_s^N,y)(\mu_s^N-\mu_s)(dy)\big|^{2p}\Big] \lesssim 1
\end{align*}
as soon as $\delta < \tfrac{1}{2}\kappa^{-1}K_{1}^{-1}\sup_{s \in [0,T]}|\widetilde b(s,\cdot)|_{\mathrm{Lip}}^{-2}$ by Lemma \ref{lem rosenthal}. Therefore \eqref{eq: prop fund} is established under Assumption \ref{ass: basic lip} (iii) and this completes the proof of Proposition \ref{prop: controle change of proba}.

\section{Proof of the nonparametric estimation results} \label{sec: proof nonpara}

%\subsection{Preliminaries}
We will repeatedly use estimates of the form
\begin{equation} \label{eq: dev exp classique}
\int_{\nu}^\infty \exp(-z^r) dz\leq 2r^{-1}\nu^{1-r}\exp(-\nu^r),\;\;\nu,r >0,\;\;\nu \geq (2/r)^{1/r}.
\end{equation} 
and
\begin{equation} \label{eq: eq exp int}
\int_0^\infty \exp\Big(-\frac{az^p}{b+cz^{p/2}}\Big)dz \leq C_p \max\Big(\Big(\frac{a}{b}\Big)^{-1/p},\Big(\frac{a}{c}\Big)^{-2/p}\Big),\;\;a,b,c,p>0,
\end{equation}
with $C_p = 2\int_0^\infty \exp(-\tfrac{1}{2}(\min(\sqrt{z},z))^p)dz$, stemming from the rough bound
$$\exp\Big(-\frac{az^p}{b+cz^{p/2}}\Big)\leq \exp\Big(-\frac{az^p}{2b}\Big)+\exp\Big(-\frac{az^{p/2}}{2c}\Big),\;\;z>0.$$
The estimate \eqref{eq: eq exp int} is far from being optimal, but will be sufficient for our purpose.
\subsection{Proof of Theorem \ref{thm: GL mu}} 
\subsubsection*{Preliminaries} 
We first state local upper and lower estimates on $(t,x) \mapsto \mu_t(x)$. 
\begin{lem} \label{lem: loc unif mu above}
Work under Assumptions \ref{ass: init condition}, \ref{ass: minimal prop sigma} and \ref{ass: basic lip}. Let $(t_0, x_0) \in (0,T] \times \R^d$. Let $r>0$ and $[r_1,r_2] \subset (0,T)$.
%For every $r>0$, 
\begin{itemize}
\item[(i)] There exists $\kappa_5$ depending on $(t_0,x_0), r$ and $\mathfrak b$ such that
% $\kappa_3, \kappa_4,\kappa_5>0$ depending on $r, t_0, x_0$ and $\mathfrak b$ such that 
\begin{equation} \label{eq: bound loc unif b}
\sup_{t \in [0,T], |x-x_0| \leq r}|b(t,x,\mu_t)| \leq \kappa_5.
\end{equation}
\item[(ii)] There exist $\kappa_3,\kappa_4$ depending on $x_0, r_1, r_2,r$ and $\mathfrak b$  such that
\begin{equation} \label{eq: upperlower gaussian tail}
0 < \kappa_4 \leq \inf_{t \in [r_1,r_2], |x-x_0| \leq r} \mu_{t}(x) \leq \sup_{t \in [r_1,r_2], |x-x_0| \leq r} \mu_{t}(x) \leq \kappa_3.
\end{equation}
\end{itemize}
\end{lem}
In turn, for a compactly supported kernel $K$, this implies the existence of $r = r(K)$ such that the estimate
\begin{equation} \label{eq: borne var mu}
|K_h(x_0-\cdot)|_{L^2(\mu_{t_0})}^2 = \int_{\R^d} h^{-2d}K\big(h^{-1}x\big)^2\mu_{t_0}(x_0-x)dx \leq \kappa_3(r) h^{-d}|K|_2^2 \\
\end{equation}
holds true.
\begin{proof}
For $x$ such that $|x-x_0| \leq r$, we have
\begin{align}
|b(t,x,\mu_t)| & \leq |b(t,0,\delta_0)|+|b|_{\mathrm{Lip}}\big(|x|+\mathcal W_1(\mu_t,\delta_0)\big) \nonumber \\
& \leq \sup_{t \in [0,T]}|b(t,0,\delta_0)|+|b|_{\mathrm{Lip}}\big(|x_0|+r+\sup_{t \in [0,T]}\int_{\R^d}|y|\mu_t(dy)\big)  \label{eq: lin growth}
\end{align}
that defines $\kappa_5$ thanks to Assumption \ref{ass: basic lip} and Lemma \ref{lem moment class}. This establishes \eqref{eq: bound loc unif b}. The estimate \eqref{eq: upperlower gaussian tail} follows from classical 
%We refer un our context to Theorem 7.3.3 and Example 8.3.10 in \cite{BoKry}.
Gaussian tail estimates for the solution of parabolic equations. We refer for example in our context to Corollary 8.2.2 of \cite{BoKry}: for every compact interval $[r_1,r_2]\subset (0,T)$, there exist constants $\mathfrak c_\pm >0$ depending on $r_1,r_2$ and $\mathfrak b$ only such that
$$\exp\big(-\mathfrak c_-(1+|x|^2)\big)\leq \mu_t(x) \leq \exp\big(\mathfrak c_+\big(1+|x|^2)\big)$$
for every $(t,x) \in [r_1,r_2]\times \R^d$. This establishes \eqref{eq: upperlower gaussian tail}. 
Actually, if we moreover have $1/2$-H\"older smoothness in time for the diffusion coefficient, investigating further Theorem 7.3.3 and Example 8.3.10 of \cite{BoKry}, 
it is possible to prove $\sup_{x \in \R^d}\mu_t(x) <\infty$ uniformly in $t \in [r_1,r_2]$, hence $\kappa_3$ can be taken independently of $r$.

\end{proof}
We next prove a standard bias-variance estimate for the quadratic risk of  $\widehat \mu_h^N(t_0,x_0)$.
\begin{lem} \label{lem: bias var mu}
 In the setting of Theorem \ref{thm: GL mu}, if $K$ is a bounded and compactly supported kernel and $h \in \mathcal H_1^N$, we have
$$
  \E_{\PP^N}\big[\big(\widehat \mu_{h}^N(t_0,x_0)-\mu_{t_0}(x_0)\big)^2\big] 
 \lesssim 
 \mathcal B_{h}^N(\mu)\big(t_0,x_0)^2+\mathsf V_h^N,
$$
up to a constant that depends (continuously) on $(t_0,x_0)$, $|K|_\infty$ and $\mathfrak b$, and where $ \mathcal B_{h}^N(\mu)\big(t_0,x_0)$ is defined in \eqref{eq: bias mu} and $\mathsf V_h^N$ in \eqref{eq: var correc mu}.
\end{lem}
\begin{proof}
Write $\widehat \mu_{h}^N(t_0,x_0)-\mu_{t_0}(x_0) = I+II$, with
$$
I = \int_{\R^d} K_h(x_0-x)\mu_{t_0}(x)dx-\mu_{t_0}(x_0)
$$
and
$$
II = \int_{\R^d} K_h(x_0-x)\big(\mu_{t_0}^N(dx)-\mu_{t_0}(x)dx\big).
$$
We have $I^2 \leq \mathcal B_h^N(\mu)(t_0,x_0)^2$ for the squared bias term. For the variance term, using successively Theorem \ref{thm: concentration} and the estimate \eqref{eq: borne var mu} 
%and 
%the abbreviation $\alpha_1 = \kappa_2, \alpha_2 = \kappa_3 |K|_2^2$ and $\alpha_3 = |K|_\infty $ together with the elementary estimate
%$$\exp(-\frac{\alpha_1x^2}{\alpha_2+\alpha_3x}) \leq \exp\big(-\tfrac{1}{2}\alpha_1(x/\alpha_2\wedge x^2/\alpha_3)\big)$$
we have
\begin{align*}
 \E_{\PP^N}\big[II^2\big] & = \int_0^\infty \PP^N(\big| II \big| \geq z^{1/2})dz \\
& \leq 2\kappa_1 \int_0^\infty  \exp\Big(-\frac{\kappa_2 Nz}{|K_h(x_0-\cdot)|_{L^2(\mu_{t_0})}^2+|K_h(x_0-\cdot)|_\infty z^{1/2}}\Big)dz \\
& \leq 2\kappa_1 \int_0^\infty  \exp\Big(-\frac{\kappa_2 Nh^dz}{\kappa_3 |K|_2^2+|K|_\infty  z^{1/2}}\Big)dz \\
% & \leq 2\kappa_1\Big(\int_0^{\kappa_3^{2}|K|_2^4/|K|_\infty^{2}}  \exp\Big(-\frac{\kappa_2 Nh^dz}{2\kappa_3 |K|_2^2}\Big)dz
%+\int_{\kappa_3^{2}|K|_2^4/|K|_\infty^{2}}^\infty \exp\Big(-\frac{\kappa_2 Nh^dz^{1/2}}{2|K|_\infty}\Big)dz\Big)\\
%& \leq 2\kappa_1(Nh^d)^{-1}\Big(\int_0^{\infty}  \exp\Big(-\frac{\kappa_2 z}{2\kappa_3 |K|_2^2}\Big)dz
%+(Nh^d)^{-1}\int_{0}^\infty \exp\Big(-\frac{\kappa_2 z}{2|K|_\infty}\Big)2zdz\Big)\\
& \lesssim (Nh^d)^{-1}(1+(Nh^d)^{-1}) \\
& \lesssim \mathsf V_h^N
\end{align*}
where we used \eqref{eq: eq exp int} and the fact that $\max_{h \in \mathcal H_1^N} (Nh^d)^{-1} \lesssim 1$.
\end{proof}

\subsubsection*{Completion of proof of Theorem \ref{thm: GL mu}} We essentially repeat the main argument of the Goldenshluger-Lepski method (see {\it e.g.} \cite{GL08, GL11, GL14} for the pointwise risk). We nevertheless give a proof for sake of completeness. Recall that $\widehat h^N$ denotes the data-driven bandwidth defined in \eqref{eq: def GL band mu}.\\ 
 
\noindent {\it Step 1:} For $h \in \mathcal H_1^N$, we successively have
\begin{align*}
&  \E_{\PP^N}\big[\big(\widehat \mu_{\mathrm GL}^N(t_0,x_0)-\mu_{t_0}(x_0)\big)^2\big]\\
 & \lesssim  \E_{\PP^N}\big[\big(\widehat \mu_{\mathrm GL}^N(t_0,x_0)-\widehat \mu_h^N(t_0,x_0)\big)^2\big]+  \E_{\PP^N}\big[\big(\widehat \mu_h^N(t_0,x_0)-\mu_{t_0}(x_0)\big)^2\big]  \\
 & \lesssim  \E_{\PP^N}\big[\big\{\big(\widehat \mu_{\widehat h^N}^N(t_0,x_0)-\widehat \mu_h^N(t_0,x_0)\big)^2-\mathsf V_h^N-\mathsf V_{\widehat h^N}^N\big\}_+  +\mathsf V_h^N+\mathsf V_{\widehat h^N}^N \big]  +  \E_{\PP^N}\big[\big(\widehat \mu_h^N(t_0,x_0)-\mu_{t_0}(x_0)\big)^2\big]  \\
&  \lesssim  \E_{\PP^N}\big[\mathsf A_{\max(\widehat h^N,h)}^N +\mathsf V_h^N+\mathsf V_{\widehat h^N}^N \big] +  \E_{\PP^N}\big[\big(\widehat \mu_h^N(t_0,x_0)-\mu_{t_0}(x_0)\big)^2\big]  \\
 & \lesssim  \E_{\PP^N}\big[\mathsf A_{h}^N\big]+\mathsf V_h^N + \E_{\PP^N}\big[\mathsf A_{\widehat h^N}^N+\mathsf V_{\widehat h^N}^N \big] +  \E_{\PP^N}\big[\big(\widehat \mu_h^N(t_0,x_0)-\mu_{t_0}(x_0)\big)^2\big] \\
 & \lesssim  \E_{\PP^N}\big[\mathsf A_{h}^N\big]+\mathsf V_h^N + \mathcal B_h^N(\mu)(t_0,x_0)^2,  
\end{align*}
where we applied Lemma \ref{lem: bias var mu} in order to obtain the last line.\\

\noindent {\it Step 2:} We first estimate $\mathsf A_{h}^N$. Write $\mu_h(t_0,x_0)$ for $\int_{\R^d}K_h(x_0-x)\mu_{t_0}(x)dx$. For $h,h'\in \mathcal H_1^N$ with $h' \leq h$, since
\begin{align*}
& \big(\widehat \mu_h^N(t_0,x_0)-\widehat \mu^N_{h'}(t_0,x_0)\big)^2 \\
&\leq  4\big(\widehat \mu_h^N(t_0,x_0)-\mu_{h}(t_0,x_0)\big)^2+4\big(\mu_h(t_0,x_0)-\mu_{t_0}(x_0)\big)^2 +4\big(\mu_{h'}(t_0,x_0)-\mu_{t_0}(x_0)\big)^2\\
&+4\big(\widehat \mu_{h'}^N(t_0,x_0)-\mu_{h'}(t_0,x_0)\big)^2,  
\end{align*}
we have
\begin{align*}
  \big(\widehat \mu_h^N(t_0,x_0)-\widehat \mu^N_{h'}(t_0,x_0)\big)^2-\mathsf V_h^N-\mathsf V_{h'}^N  & \leq 8\mathcal B_h^N(\mu)(t_0,x_0)^2+\big(4(\widehat \mu_h^N(t_0,x_0)-\mu_{h}(t_0,x_0))^2-\mathsf V_h^N\big)\\
&+\big(4(\widehat \mu_{h'}^N(t_0,x_0)-\mu_{h'}(t_0,x_0))^2-\mathsf V_{h'}^N\big)
\end{align*}
using $h' \leq h$ in order to bound $(\widehat \mu_{h'}^N(t,a)-\mu_{h'}(t_0,x_0))^2$ by the bias at scale $h$. Taking maximum over $h'\leq h$, we obtain
\begin{align}
& \max_{h' \leq h}\big\{\big(\widehat \mu_h^N(t_0,x_0)-\widehat \mu^N_{h'}(t_0,x_0)\big)^2-\mathsf V_h^N-\mathsf V_{h'}^N\big\}_+  \label{eq end lepski}\\
\leq & \;8\mathcal B_h^N(\mu)(t_0,x_0)^2+  \big\{4\big(\widehat \mu_h^N(t_0,x_0)-\mu_{h}(t_0,x_0)\big)^2-\mathsf V_h^N\big\}_+ \nonumber \\
&+\max_{h' \leq h}\big\{4\big(\widehat \mu_{h'}^N(t_0,x_0)-\mu_{h'}(t_0,x_0)\big)^2-\mathsf V_{h'}^N\big\}_+. \nonumber%\label{eq end lepski}
\end{align}

\noindent {\it Step 3:} We estimate the expectation of the first stochastic term in the right-hand side of \eqref{eq end lepski}. We refine the computation of the term $II$ in the proof of Lemma \ref{lem: bias var mu}. By Theorem \ref{thm: concentration} and using estimates of the form  \eqref{eq: dev exp classique} and \eqref{eq: eq exp int}, we have
\begin{align*}
  &\E_{\PP^N}\big[\big\{4\big(\widehat \mu_h^N(t_0,x_0)-\mu_{h}(t_0,x_0)\big)^2-\mathsf V_h^N\big\}_+\big] \\
  &  = \int_0^\infty \PP^N\big(4\big(\widehat \mu_h^N(t_0,x_0)-\mu_{h}(t_0,x_0)\big)^2-\mathsf V_h^N \geq z\big)dz \\  
 & = \int_0^\infty \PP^N\big(|\widehat \mu_h^N(t_0,x_0)-\mu_h(t_0,x_0)| \geq \tfrac{1}{2}(\mathsf V_h^N+z)^{1/2}\big)dz \\
&  \leq 2\kappa_1 \int_{\mathsf V_h^N}^\infty  \exp\Big(-\frac{\kappa_2 Nh^d \tfrac{1}{4}z}{\kappa_3 |K|_2^2+|K|_\infty  \tfrac{1}{2}z^{1/2}}\Big)dz \\
& \lesssim \int_{\mathsf V_h^N}^\infty  \exp\Big(-\frac{\kappa_2Nh^dz}{8\kappa_3 |K|_2^2}\Big)dz+\int_{\mathsf V_h^N}^\infty \exp\Big(-\frac{\kappa_2 Nh^dz^{1/2}}{4|K|_\infty}\Big)dz\\
& \lesssim (Nh^d)^{-1}\exp\Big(-\frac{\kappa_2Nh^d \mathsf V_h^N}{8\kappa_3 |K|_2^2}\Big)+(Nh^d)^{-2}Nh^d (\mathsf V_h^N)^{1/2} \exp\Big(-\frac{\kappa_2 Nh^d (\mathsf V_h^N)^{1/2}}{4|K|_\infty}\Big)\\
& \lesssim (Nh^d)^{-1}N^{-\varpi_1 \kappa_2/(8\kappa_3)}+(Nh^d)^{-3/2}(\log N)^{1/2}\exp\big(-{\tfrac{\kappa_2|K|_2\varpi_1^{1/2}}{4|K|_\infty}(\log N)^{5/2}}\big),\\
& \lesssim N^{-2}
\end{align*}
as soon as 
%$\varpi_2 \geq 16\kappa_3/\kappa_2$ and $\varpi_1\varpi_2 \geq 64\kappa_2^{-2}|K|_\infty^2|K|_2^{-2}$. Assuming moreover $\varpi_2 \geq 1$ and using $|K|_2 \geq 1$, a sufficient condition is $\varpi_1 \geq 64\kappa_2^{-2}|K|_\infty^2$.
$\varpi_1 \geq 16\kappa_2^{-1}\kappa_3$, thanks to 
%thanks to estimates of the form  \eqref{eq: dev exp classique} and \eqref{eq: eq exp int},
%for $\varpi_1 \geq 4|K|_\infty^2|K|_2^{-2}(\kappa_2\kappa_3)^{-1}$ and  $\varpi_2 \geq 16\kappa_3/\kappa_2$, 
 $\max_{h \in \mathcal H_1^N} (Nh^d)^{-1} \lesssim 1$, and using $\min_{h \in \mathcal H_1^N}h \geq (N^{-1}(\log N)^2)^{1/d}$ to show that the second term is negligible in front of $N^{-2}$.\\

\noindent {\it Step 4:} For the second stochastic term, we have the rough estimate
\begin{align*}
&  \E_{\PP^N}\big[\max_{h' \leq h}\big\{4\big(\widehat \mu_{h'}^N(t_0,x_0)-\mu_{h'}(t_0,x_0)\big)^2-\mathsf V_{h'}^N\big\}_+\big] \\
& \leq \sum_{h'\leq h} \E_{\PP^N}\big[\big\{4\big(\widehat \mu_{h'}^N(t_0,x_0)-\mu_{h'}(t_0,x_0)\big)^2-\mathsf V_{h'}^N\big\}_+\big]  \lesssim \mathrm{Card}(\mathcal H_1^N) N^{-2} \lesssim N^{-1} 
\end{align*}
where we used Step 3 to bound each term $ \E_{\PP^N}\big[\big\{4\big(\widehat \mu_{h'}^N(t_0,x_0)-\mu_{h'}(t_0,x_0)\big)^2-\mathsf V_{h'}^N\big\}_+\big]$ independently of $h$ together with $\mathrm{Card}(\mathcal H_1^N) \lesssim N$. In conclusion, we have through Steps 2-4 that $ \E_{\PP^N}\big[\mathsf A_{h}^N\big] \lesssim N^{-1} +\mathcal B_h^N(\mu)(t_0,x_0)^2$. Therefore, from Step 1, we conclude
$$ \E_{\PP^N}\big[\big(\widehat \mu_{\mathrm GL}^N(t_0,x_0)-\mu_{t_0}(x_0)\big)^2\big] \lesssim \mathcal B_h^N(\mu)(t_0,x_0)^2 + \mathsf V_h^N + N^{-1}$$  
for any $h \in \mathcal H_1^N$. Since $N^{-1} \lesssim \mathsf V_h^N $ always, the proof of Theorem \ref{thm: GL mu} is complete.

\subsection{Proof of Theorem \ref{thm: oracle b}} \label{sec proof of oracle b}

\subsubsection*{Preliminaries}
The assumptions of Theorem \ref{thm: oracle b} are in force in this section.
We first study the fluctuations of the random measure $\pi^N(dt,dx) - \pi(t,x)dtdx$, where $\pi^N(dt,dx) = N^{-1}\sum_{i = 1}^N\delta_{X_t^i}(dx)X^i(dt)$. 
%\begin{equation} \label{eq: conv emp measure 2}
%\pi^N(dt,dx) - \pi(t,x)dtdx,
%\end{equation}
%where
%\begin{equation} \label{eq: random measure pi}
%\pi^N(dt,dx) = N^{-1}\sum_{i = 1}^N\delta_{X_t^i}(dx)X^i(dt) 
%\end{equation}

%If $\phi:[0,T]\times \R^d \rightarrow \R$ is a test function, we set $|\phi|_{p,1} = \big(\int_0^T|\phi(t,\cdot)|_1^{p}dt\big)^{1/p}$ for $1 \leq p \leq \infty$.

\begin{lem} \label{lem: dev b}
Let $\phi: (0,T] \times \R^d  \rightarrow \R$ be bounded and compactly supported. The following decomposition holds
\begin{align*}
\int_{[0,T] \times \R^d}&\phi(t,x)(\pi^N(dt,dx)-\pi(t,x)dtdx\big) \nonumber \\
&=  \int_0^T\int_{\R^d}\phi(t,x)\big(b(t,x,\mu_t)(\mu_t^N(dx)-\mu_t(x)dx)+\xi_t^N(x)\mu_t^N(x)dx\big)dt+\mathcal M_T^N(\phi), \label{eq: decomp mg}
\end{align*}
where $\xi_t^N(x) = b(t,x,\mu_t^N)-b(t,x,\mu_t)$ and $\mathcal M_t^N(\phi) = \big(\mathcal M_t^N(\phi)^1,\ldots, \mathcal M_t^N(\phi)^d\big)$ is a $d$-dimensional $\mathbb P^N$-continuous martingale with predictable compensator such that 
\begin{equation} \label{eq: controle bracket}
\langle \mathcal M_.^N(\phi)^k\rangle_t \leq N^{-1}|\mathrm{Tr}(c)|_\infty \int_0^t \int_{\R^d}\phi(s,x)^2\mu_s^N(dx)ds.
\end{equation}

%Moreover, for every $u\geq 0$, we have
%\begin{equation} \label{eq: dev bracket}
%\PP^N\big(\big\langle \mathcal M_.^N(\phi)^k \big\rangle_T \geq u\big)  \lesssim {\bf 1}_{\big(N^{-1}|\phi|^2_2\kappa_6\geq u\big)} + 
%\exp \Big(-\frac{\kappa_7 Nu^2}{N^{-1}|\phi|_4^4+|\phi|_\infty^2u}\Big)
%\end{equation}
%with $\kappa_6 = 2|\mathrm{Tr}(c)|_\infty\kappa_3$ and $\kappa_7 = \kappa_2 \big(2|\mathrm{Tr}(c)|_\infty\max(2|\mathrm{Tr}(c)|_\infty T^{-1}\kappa_3,1\big)\big)^{-1}$. The constant $\kappa_3$ is defined in Lemma \ref{lem: loc unif mu above} and is calibrated with a radius $r>0$ such that $\mathrm{Supp}(\phi) \subset (0,r] \times \{|x| \leq r\}$.

\end{lem}

\begin{proof}
we have
\begin{align*}
& \int_{[0,T] \times \R^d}\phi(t,x)\big(\pi^N(dt,dx) -\pi(t,x)dtdx\big) \\
& =  N^{-1}\sum_{i = 1}^N \int_0^T \phi(t,X_t^i)\big(\sigma(t,X_t^i)dB_t^i+b(t,X_t^i,\mu_t^N)dt\big)-  \int_{[0,T] \times \R^d}\phi(t,x)b(t,x,\mu_t)\mu_t(x)dtdx \\
& = \int_0^T\int_{\R^d} \phi(t,x)\big( b(t,x,\mu_t^N)\mu_t^N(dx)-b(t,x,\mu_t)\mu_t(x)dx\big)dt+ \mathcal M_T^N(\phi), 
\end{align*}
where 
$$\mathcal M_t^N(\phi) =  N^{-1}\sum_{i = 1}^N\int_0^t\phi(s,X_s^i)\sigma(s,X_s^i)dB_s^i$$
is a martingale with bracket satisfying  \eqref{eq: controle bracket}. The result follows.

\end{proof}

We next have a bias-variance estimate for the quadratic risk of $\widehat \pi_{\boldsymbol h}(t_0,x_0)$, in the same spirit as in Lemma \ref{lem: bias var mu}.

\begin{lem} \label{lem: bias var pi}
Assume that $H\otimes K$ is a bounded and compactly supported kernel on $(0,T) \times \R^d$. Let $\boldsymbol h  \in \mathcal H_2^N$. Then
$$
\E_{\PP^N}\big[\big|\widehat \pi_{\boldsymbol h}^N(t_0,x_0)-\pi(t_0,x_0)\big|^2\big] 
 \lesssim 
 \mathcal B_{\boldsymbol h}^N(\pi)(t_0,x_0)^2+\mathsf V_{\boldsymbol h}^N,
$$
up to a constant that (continuously) depends on $(t_0,x_0)$, $|H\otimes K|_\infty$ and $\mathfrak b$, and where $ \mathcal B_{\boldsymbol h}^N(\pi)\big(t_0,x_0)$ is defined in \eqref{eq: bias pi} and $\mathsf V_{\boldsymbol h}^N$ in \eqref{def upper bivariance}.
\end{lem}

\begin{proof}
Write $\widehat \pi_{\boldsymbol h}^N(t_0,x_0)-\pi(t_0,x_0) = I+II$, with
$$
I = \int_0^T\int_{\R^d} (H\otimes K)_{\boldsymbol h}(t_0-t,x_0-x)\pi(t,x)dxdt-\pi(t_0,x_0)
$$
and
$$
II = \int_0^T\int_{\R^d}(H\otimes K)_{\boldsymbol h}(t_0-t,x_0-x)(\pi^N(dt,dx)-\pi(t,x)dtdx\big).
$$
We have $\big|I\big|^2 \leq \mathcal B_{\boldsymbol h}^N(\pi)(t_0,x_0)^2$ for the squared bias term. For the variance term, applying the decomposition of Lemma \ref{lem: dev b} with test function $\phi(t,x) = (H\otimes K)_{\boldsymbol h}(t_0-t,x_0-x)$, we obtain
$$
\big|II\big|^2  \lesssim III + IV + V,
$$
with
\begin{align*}
III & = \Big|\int_0^T\int_{\R^d}\phi(t,x)b(t,x,\mu_t)(\mu_t^N(dx)-\mu_t(x)dx\big)dt\Big|^2,\\
IV & = \Big|\int_0^T\int_{\R^d}\phi(t,x)\xi^N_t(x)\mu_t^N(dx)dt\Big|^2,\\
V& = \big|\mathcal M_T^N(\phi)\big|^2,
\end{align*}
%where we set $\xi_t^N(x) =$. 
where $\xi_t^N(x) = b(t,x,\mu_t^N)- b(t,x,\mu_t)$. Writing $b = (b^1,\ldots, b^d)$ in components, note first that for $\nu(dt,dx) = \mu_t(dx)T^{-1}dt$, we have
\begin{equation} \label{eq: estimes phi}
\big|\phi \,b^k(\cdot,\mu_\cdot)\big|^2_{L^2(\nu)} \leq \kappa_3\kappa_5^2T^{-1}|H\otimes K|_2^2 (h_1h_2^d)^{-1},\;\;\big|\phi \, b^k(\cdot,\mu_\cdot)\big|_\infty \leq \kappa_5|H \otimes K|_\infty (h_1h_2^d)^{-1},
\end{equation}
by Lemma \ref{lem: loc unif mu above} and the compactness of the support of $\phi$. By Theorem  \ref{thm: concentration} applied to $\nu^N(dt,dx)-\nu(dt,dx) = (\mu_t^N(dx)-\mu_t(dx))T^{-1}dt$, it follows that
\begin{align*}
\E_{\PP^N}[III] & \lesssim \sum_{k = 1}^d \int_0^\infty \PP^N\big(\big|\int_{[0,T] \times \R^d}\phi(t,x)b^k(t,x,\mu_t)(\mu_t^N(dx)-\mu_t(x)dx)T^{-1}dt\big| \geq z^{1/2}\big)dz \\
 & \leq 2d\kappa_1  \int_0^\infty \exp\Big(-\frac{\kappa_2 Nh_1h_2^dz}{\kappa_3\kappa_5^2T^{-1}|H \otimes K|_2^2 +\kappa_5|H \otimes K|_\infty z^{1/2}}\Big)dz \\
 & \lesssim  (Nh_1h_2^d)^{-1}(1+(Nh_1h_2^d)^{-1}) \\
 & \lesssim \mathsf{V}_{\boldsymbol h}^N,
\end{align*}
using $\max_{\boldsymbol h \in \mathcal H_2^N} (Nh_1h_2^d) \lesssim 1$.
We conclude
\begin{equation} \label{eq: first esti var}
\E_{\PP^N}\big[III\big]
 \lesssim  \mathsf V_{\boldsymbol h}^N.
\end{equation}
We next turn to the term $IV$. We need a deviation result for the fluctuation $\xi_t^N(x) = b(t,x,\mu_t^N)- b(t,x,\mu_t)$ that will also be helpful later.
\begin{lem} \label{lem: fluctuation xi}
There exist positive numbers $\kappa_6, \kappa_7$ and $\kappa_8$, depending on $\mathfrak b$,  such that for large enough $N$
%\begin{equation} 
%\label{eq: deviation xi}
$$
\sup_{0 \leq t \leq T}\PP^N\big(|\xi_t^N(X_t^N)| \geq u\big) \leq \kappa_6\exp\Big(-\frac{\kappa_7 Nu^2}{1+N^{1/2}u}\Big)\;\;\text{for}\;\;u \geq \kappa_8 N^{-1/2}.
$$
%\end{equation}
\end{lem}

\begin{proof}
Writing $\xi_t^N(X_t^N) = \big(\xi_t^N(X_t^N)^1,\ldots, \xi_t^N(X_t^N)^d\big)$ in components, we have
$$\PP^N\big(|\xi_t^N(X_t^N)| \geq u\big) \leq \sum_{\ell = 1}^d \PP^N\big(|\xi_t^N(X_t^N)^\ell| \geq u d^{-1}\big)$$
It suffices thus to prove the result for each component $\xi_t^N(X_t^N)^\ell$, substituting $\kappa_6$ and $\kappa_8$ by $d\kappa_6$ and $d\kappa_8$ to obtain the general case. For notational simplicity, we drop the superscript $\ell$ and prove the result for $\xi_t^N(X_t^N)$ instead of $|\xi_t^N(X_t^N)|$, up to a inflation of $\kappa_6$ by a factor 2.\\

By \eqref{eq: taylor linear diff} in the proof of Proposition \ref{prop: controle change of proba}, we may write
$$\xi_t^N(X_t^N)  = \zeta_t^N(X_t^N)+ \mathcal R_k,$$
with
$$\zeta_t^N(X_t^N) = \sum_{\ell = 1}^{k-1} \frac{1}{{\ell !}}\int_{(\R^d)^{\ell}}  \delta_\mu^\ell b(t,X_t^N,y^{\ell}, \mu_t)(\mu_t^N-\mu_t)^{\otimes \ell}(dy^{\ell}),$$
having $k=1$ under Assumption \ref{ass: basic lip} (i), with $\big|\mathcal R_k\big| \lesssim \mathcal W_1(\mu_t^N, \mu_t)  \wedge  \mathcal W_1(\mu_t^N, \mu_t)^{k}$ under Assumption \ref{ass: basic lip} (ii) by \eqref{eq control remainder}, and having $k=1$ with $\delta_\mu^1b=\widetilde b$ and $\mathcal R_k = 0$ under Assumption \ref{ass: basic lip} (iii). It is enough to prove the deviation bound for each term separately.\\ 

Let $u \geq 0$. We first bound the remainder term $\mathcal R_k$. Applying \eqref{eq: first estimate changproba} in the proof of Theorem \ref{thm: concentration} for the event $\mathcal A^N = \{ \mathcal W_1(\mu_t^N, \mu_t) \wedge \mathcal W_1(\mu_t^N, \mu_t)^k \geq u\}$  we obtain
\begin{align*}
\PP^N( \mathcal W_1(\mu_t^N, \mu_t) \wedge \mathcal W_1(\mu_t^N, \mu_t)^k \geq u) & \leq C_1^{K(K+1)/8}\overline{\PP}^N\big( \mathcal W_1(\mu_t^N, \mu_t) \wedge \mathcal W_1(\mu_t^N, \mu_t)^k \geq u\big)^{4^{-K}} \\
& \leq C_1^{K(K+1)/8}\big(\varepsilon_N(u^{1/k})\wedge \varepsilon_N(u)\big)^{4^{-K}}
\end{align*} 
where the last estimate stems from the deviation inequality \eqref{eq: fournier guillin} of Fournier and Guillin \cite{FournierGuillin}. Under Assumption \ref{ass: basic lip} (i), with ($d=1$ and $k=1$) or under Assumption \ref{ass: basic lip} (ii) with ($d=2$ and $k \geq 2$) or ($d \geq 3$ and $k \geq d/2$), 
$$\varepsilon_N(u^{1/k})\wedge \varepsilon_N(u)\lesssim \exp(-\mathfrak CN u^{2})\;\;\text{for every}\;u \geq 0$$
as follows from the definition of $\varepsilon_N(x)$ in \eqref{eq: fournier guillin}. Therefore $\mathcal R_k$ has the right order.
%$$
%\lesssim \left\{
%\begin{array}{ll}
%\exp(-\mathfrak C4^{-K}N u^{2/k}) & \mathrm{for} \;  u > 1\\
%\exp(-\mathfrak C4^{-K}N u^{2/k}) & \mathrm{for} \; u \leq 1 \; \mathrm{and} \; d=1\\
%\exp(-\mathfrak C4^{-K}N\tfrac{u^{2/k}}{(\log (2+u^{-1/k}))^2}\big) & d=2 \\
%\exp(-\mathfrak C4^{-K}Nu^{d/k})) &  d \geq 3 
%\end{array}
%\right.
%$$
%%\end{align*}
%
%
%
%{\color{blue} For Assumption {\color{blue}\ref{ass: basic lip}} (i), $d=1$, so we take $k'=1$ and we have }
%$$ 
%{\color{blue}
%\PP^N(\mathcal W_1(\mu_t^N, \mu_t) \geq u) \lesssim \exp(-\mathfrak C4^{-K}N u^{2})}
%$$
%and this term has the right order.
%
%{\color{blue} For Assumption {\color{blue}\ref{ass: basic lip}} (ii), the case $d=1$ is already done as Assumption {\color{blue}\ref{ass: basic lip}} (ii) implies Assumption {\color{blue}\ref{ass: basic lip}} (i) and so, for $d\geq 2$, under the same restrictions on the dimension $d$ as in Theorem \ref{thm: concentration}, we have  }
%$$
%{\color{blue}\PP^N(\mathcal W_1(\mu_t^N, \mu_t) \geq u){\bf 1}_{u>1} \lesssim \exp(-\mathfrak C4^{-K}N u^{2}){\bf 1}_{u>1}}
%$$
%$$ 
%{\color{blue}
%\PP^N(\mathcal W_1^k(\mu_t^N, \mu_t) \geq u){\bf 1}_{u\leq1} \lesssim ( {\bf 1}_{d=2}
%\exp(-\mathfrak C4^{-K}N\tfrac{u^{2/k}}{(\log (2+u^{-1/k}))^2}\big) + {\bf 1}_{d\geq 3}
%\exp(-\mathfrak C4^{-K}Nu^{d/k}) ){\bf 1}_{u\leq1}}
%$$
%{\color{blue}which have the right order.}
As for the main term, we first note that for every $p \geq 2$, we have
\begin{align*}
\E_{\overline{\PP}^N}\big[\big| \zeta_t^N(X_t^N)\big|^p\big] 
 & \leq   e^{p-1}\sum_{\ell = 1}^{k-1} \frac{1}{{\ell !}}\, \E_{\overline{\PP}^N}\Big[\Big|\int_{(\R^d)^{\ell}}  \delta_\mu^\ell b(t,X_t^N,y^{\ell}, \mu_t)(\mu_t^N-\mu_t)^{\otimes \ell}(dy^{\ell})\Big|^p\Big]  \nonumber \\ & \leq N^{-p/2}p! C_5^p, \label{eq: moment p et C5}
\end{align*}
by Lemma \ref{lem rosenthal} and Cauchy-Schwarz's inequality, for large enough $N$ and some $C_5$ depending on $\mathfrak b$. With no loss of generality, we take $C_5 \geq 1$. In particular, by Cauchy-Schwarz's inequality,
$$\big|\E_{\overline{\PP}^N}[\zeta_t^N(X_t^N)]\big| \leq \sqrt{2}C_5N^{-1/2} = \kappa_8N^{-1/2}$$
that defines the constant $\kappa_8$. We next use the following version of Bernstein inequality that can be found in Lemma 8 in Birg\'e and Massart \cite{birge1998minimum}: if $Z$ is a real-valued random variable such that $\E[|Z|^p] \leq \frac{p!}{2}v^2c^{p-2}$ for $c, v >0$ and every $p \geq 2$, then
\begin{equation} \label{eq: bernstein birge}
\PP(Z-\E[Z] \geq u) \leq \exp\Big(-\frac{u^2/2}{v^2+cu}\Big)\;\;\text{for every}\;\;u \geq 0.
\end{equation}
We then apply \eqref{eq: bernstein birge} to $Z = \zeta_t^N(X_t^N)$ with $c=C_5N^{-1/2}$ and $v=\sqrt{2}N^{-1/2}C_5$ and obtain
\begin{align*}
\overline{\PP}^N\big(\zeta_t^N(X_t^N)-\E_{\overline{\PP}^N}[\zeta_t^N(X_t^N)] \geq u\big) &  \leq \exp\Big(-\frac{C_6Nu^2}{1+N^{1/2}u}\Big)\;\;\text{ for every}\;\;u \geq 0,
%\exp\Big(-\frac{Nu^2}{2(2C_5^2+C_5N^{1/2}u)}\Big) \\
\end{align*}
with $C_6 = (4C_5^2)^{-1}$ using $C_5 \geq 1$. Finally, for $u \geq \kappa_8N^{-1/2}$, setting $u'=u-\kappa_8N^{-1/2}$ and applying \eqref{eq: first estimate changproba} in the proof of Theorem \ref{thm: concentration} for the event $\{ \zeta_t^N(X_t^N) \geq u\}$, we derive
\begin{align*}
\PP^N\big(\zeta_t^N(X_t^N) \geq u\big) & \lesssim  \overline{\PP}^N\big(\zeta_t^N(X_t^N) \geq u\big)^{4^{-K}} \\
& \leq \overline{\PP}^N\big(\zeta_t^N(X_t^N)-\E_{\overline{\PP}^N}[\zeta_t^N(X_t^N)]  \geq u'\big)^{4^{-K}}  \\
& \leq  \exp\Big(-\frac{{C_64^{-K}}N(u')^2}{1+N^{1/2}u'}\Big) \\
& \lesssim  \exp\Big(-\frac{C_64^{-K}Nu^2}{1+N^{1/2}u}\Big)
\end{align*}
and the lemma follows with $\kappa_7 = 4^{-K}C_6$.
\end{proof}
We are ready to bound the term $IV$. Applying Cauchy-Schwarz's inequality twice, we have 
\begin{align*}
\E_{\PP^N}\big[IV\big] & \leq \E_{\PP^N}\Big[\Big(\int_0^T\int_{\R^d}|\phi(t,x)|^2\mu_t^N(dx)dt\Big)^{2}\Big]^{\frac{1}{2}}\E_{\PP^N}\Big[\Big(\int_0^T\int_{\R^d}|\xi_t^N(x)|^2\mu_t^N(dx)dt\Big)^2\Big]^\frac{1}{2}. 
\end{align*}
%Applying H\" older's inequality twice, we have
%\begin{align*}
%\E_{\PP^N}\big[IV\big] & \lesssim \E_{\PP^N}\Big[\Big(\int_0^T\int_{\R^d}|\phi(t,x)|^p\mu_t^N(dx)dt\Big)^{\frac{4}{p}}\Big]^{\frac{1}{2}}\E_{\PP^N}\Big[\Big(\int_0^T\int_{\R^d}|\xi_t^N(x)|^q\mu_t^N(dx)dt\Big)^\frac{4}{q}\Big]^\frac{1}{2} 
%\end{align*}
%where $p,q \geq 1$ are such that $p^{-1}+q^{-1}=1$. 
On the one hand, by exchangeability and Lemma \ref{lem: fluctuation xi}, we have
\begin{align}
 \E_{\PP^N}\Big[\Big(\int_0^T\int_{\R^d}|\xi_t^N(x)|^2\mu_t^N(dx)dt\Big)^2\Big]^\frac{1}{2} 
& \leq \Big(T\int_0^T\E_{\PP^N}\Big[|\xi_t^N(X_t^N)|^4\Big]dt
\Big)^\frac{1}{2} \nonumber\\
& \lesssim \sup_{0 \leq t \leq T} \Big(\int_0^\infty \PP^N\big(\big|\xi_t^N(X_t^N)\big| \geq z^{1/4})dz\Big)^{1/2} \nonumber \\
& \lesssim \Big(\kappa_{8}^{4}N^{-2}+\kappa_6\int_0^\infty \exp\Big(-\frac{\kappa_7Nz^{1/2}}{1+N^{1/2}z^{1/4}}\Big)dz\Big)^{1/2} \nonumber \\
&  \lesssim N^{-1}. \label{eq: controle xi holder}
\end{align} 
On the other hand
\begin{align}
&  \E_{\PP^N}\Big[\Big(\int_0^T\int_{\R^d}\phi(t,x)^2\mu_t^N(dx)dt\Big)^{2}\Big]^{1/2} \nonumber\\
& \lesssim  \int_0^T\int_{\R^d}\phi(t,x)^2\mu_t(x)dxdt + \E_{\PP^N}\Big[\Big|\int_{[0,T] \times \R^d}\phi(t,x)^2(\mu_t^N(dx)-\mu_t(dx))dt\Big|^{2}\Big]^{1/2} \nonumber \\
& \lesssim  |\phi|_2^2+\Big(\int_0^\infty\PP^N\big(\big|\int_{[0,T] \times \R^d}\phi(t,x)^2(\mu_t^N(dx)-\mu_t(dx))T^{-1}dt\big| \geq z^{1/2}\big)\Big)^{1/2} \nonumber \\
& \lesssim  |\phi|_{2}^2+\Big(2\kappa_1\int_0^\infty \exp\Big(-\frac{\kappa_2Nz}{\kappa_3T^{-1}|\phi|_{4}^{4}+|\phi|_\infty^2 z^{1/2}}\Big)dz\Big)^{1/2} \nonumber \\
& \lesssim  |\phi|_{2}^2+\big(N^{-1}|\phi|_4^4+N^{-2}|\phi|_\infty^4\big)^{1/2} \lesssim   |\phi|_{2}^2+N^{-1/2}|\phi|_4^2+N^{-1}|\phi|_\infty^2, \label{eq: holder 1}
\end{align}
using Lemma \ref{lem: loc unif mu above} and the fact that $\phi$ is compactly supported to obtain the first term and Theorem \ref{thm: concentration} applied to $\nu^N(dt,dx)-\nu(dt,dx)=(\mu^N_t(dx)-\mu_t(dx))T^{-1}dt$ together with $|\phi^2|_{L^2(\nu)}^2 \leq \kappa_3T^{-1}|\phi|_4^4$ to obtain the second term.
%\begin{align*}
%& \int_0^T\int_0^\infty \exp\Big(-\frac{\kappa_2Nz^{p/2}}{\kappa_3|\phi(t,\cdot)|_{2p}^{2p}+|\phi(t,\cdot)|_\infty^p z^{p/4}}\Big)dz dt\\
%%& \leq \int_0^T\Big(\int_0^{|\phi(t,\cdot)|_{2p}^8/|\phi(t,\cdot)|_\infty^4} \exp\Big(-\frac{\kappa_2Nz^{p/2}}{2|\phi(t,\cdot)|_{2p}^{2p}}\Big)dz + \int_{|\phi(t,\cdot)|_{2p}^8/|\phi(t,\cdot)|_\infty^4}^\infty \exp\Big(-\frac{\kappa_2Nz^{p/4}}{2|\phi(t,\cdot)|_\infty^p}\Big)dz \Big)dt\\
%& \lesssim N^{-2/p}\int_0^T|\phi(t,\cdot)|_{2p}^4dt+N^{-4/p}\int_0^T|\phi(t,\cdot)|_\infty^4dt  \\
%& \lesssim  N^{-2/p}|\phi|_{2p}^4+N^{-4/p}|\phi|_\infty^4. 
%\end{align*}
%We derive, for $p \geq 2,$ 
%\begin{equation} \label{eq: holder 1}
% \E_{\PP^N}\Big[\Big(\int_0^T \int_{\R^d}|\phi(t,x)|^p\mu_t^N(dx)dt\Big)^{4/p} \Big]^{1/2} \lesssim  |\phi|_{p}^{2}+N^{-1/p}|\phi|_{2p}^2+N^{-2/p}|\phi|_\infty^2.
%\end{equation}
Putting together \eqref{eq: controle xi holder} and \eqref{eq: holder 1}, we conclude
\begin{align}
\E_{\PP^N}\big[IV\big] & \lesssim N^{-1}\big(   |\phi|_{2}^2+N^{-1/2}|\phi|_4^2+N^{-1}|\phi|_\infty^2 \big) \nonumber \\
& \lesssim  (Nh_1h_2^d)^{-1}(1+(Nh_1h_2^d)^{-1/2}+(Nh_1h_2^d)^{-1}\big)
 \lesssim \mathsf V_{\boldsymbol h}^N. \label{eq: last remainder}
\end{align}
%using $\inf_{\boldsymbol h \in \mathcal H_2^N} h_1h_2^d \gtrsim N^{-1}$ (and actually bigger). 
Finally, by Lemma \ref{lem: dev b} and Theorem \ref{thm: concentration} again
%applied to $\nu^N(dt,dx)-\nu(dt,dx) = (\mu_t^N(dx)-\mu_t(dx))T^{-1}dt$ and using $|\phi^2|_{L^2(\nu)}^2 \leq \kappa_3T^{-1}|\phi|_4^4$, 
we have
\begin{align}
\E_{\mathbb P^N}\big[V\big] & = \sum_{k = 1}^d\E_{\mathbb P^N}\big[\langle \mathcal M_\cdot^N(\phi)^k\rangle_T\big] \nonumber \\
& \leq d N^{-1}|\mathrm{Tr}(c)|_\infty  \E_{\mathbb P^N}\Big[\int_{[0,T] \times \R^d}\phi(s,x)^2\mu_s^N(dx)ds\Big] \nonumber \\
& \lesssim N^{-1}|\phi|_2^2 + N^{-1}\E_{\mathbb P^N}\Big[\Big|\int_{[0,T] \times \R^d}\phi(s,x)^2\big(\mu_s^N(dx)-\mu_s(dx)\big)T^{-1}ds\Big|\Big] \nonumber \\
& \lesssim N^{-1}|\phi|_2^2+N^{-1} \int_0^\infty \PP^N\big(\big|\int_{[0,T] \times \R^d}\phi(s,x)^2\big(\mu_s^N(dx)-\mu_s(x)\big)T^{-1}ds \big| \geq z\big)dz \nonumber \\
& \leq  N^{-1}|\phi|^2_2 + N^{-1}2\kappa_1 \int_0^\infty \exp\Big(-\frac{\kappa_2Nz^2}{\kappa_3T^{-1}|\phi|_{4}^{4}+|\phi|_\infty^2 z}\Big)dz \nonumber \\
& \lesssim N^{-1}|\phi|^2_2 + N^{-3/2}|\phi|_4^2 + N^{-2}|\phi|_\infty^2 \nonumber \\
& \lesssim (Nh_1h_2^d)^{-1}\big(1+ (Nh_1h_2^d)^{-1/2}+ (Nh_1h_2^d)^{-1}\big)
 \lesssim \mathsf V_{\boldsymbol h}^N. \label{eq: mg bound last}
\end{align}
Putting together \eqref{eq: first esti var}, \eqref{eq: last remainder} and \eqref{eq: mg bound last} establishes $\E_{\mathbb P^N}\big[II^2\big] \lesssim \mathsf V_{\boldsymbol h}^N$ and concludes the proof of Lemma \ref{lem: bias var pi}.
\end{proof}

\subsubsection*{Completion of proof of Theorem \ref{thm: oracle b}}
Let $(h,\boldsymbol h) \in \mathcal H_1^N \times \mathcal H_2^N$ and $(t_0,x_0) \in (0,T) \times \R^d$. Remember that we set $\pi(t,x) = b(t,x,\mu_t)\mu_t(x)$.\\

\noindent {\it Step 1:} We plan to use the decomposition
\begin{align*}
\widehat b_{h,\boldsymbol h}^N(t_0,x_0)_{\varpi_3} - b(t_0,x_0,\mu_{t_0}) = I+II,
\end{align*}
with
$$I = \frac{\pi(t_0,x_0)\big(\mu_{t_0}(x_0)-\widehat \mu_h^N(t_0,x_0) \vee \varpi_3\big)}{\mu_{t_0}(x_0)\widehat \mu_h^N(t_0,x_0)\vee \varpi_3}  
$$
and
$$II = \frac{\big(\widehat \pi_{\boldsymbol h}^N(t_0,x_0)-\pi(t_0,x_0)\big)\mu_{t_0}(x_0)}{\mu_{t_0}(x_0)\widehat \mu_h^N(t_0,x_0)\vee \varpi_3}.
$$
First, we have 
\begin{align*}
|I| & \leq \frac{\kappa_5}{\varpi_3}|\mu_{t_0}(x_0)-\widehat \mu_h^N(t_0,x_0) \vee \varpi_3 | \lesssim |\mu_{t_0}(x_0)-\widehat \mu_h^N(t_0,x_0)| 
\end{align*}
as soon as $\varpi_3 \leq \kappa_4$
by Lemma \ref{lem: loc unif mu above},  for some (small) $r>0$ fixed throughout. 
In the same way, 
$$|II| \leq \varpi_3^{-1}|\widehat \pi_{\boldsymbol h}^N(t_0,x_0)-\pi(t_0,x_0)|.$$
Picking $h = \widehat h^N$, $\boldsymbol h = \widehat {\boldsymbol h}^N$, taking square and expectation, we have thus established
\begin{align} 
& \E_{\PP^N}\big[\big|\widehat b_{\mathrm{GL}}^N(t_0,x_0)- b(t_0,x_0,\mu_{t_0})\big|^2\big] \nonumber \\
& \lesssim \E_{\PP^N}\big[\big(\widehat \mu_{\mathrm{GL}}^N(t_0,x_0)-\mu_{t_0}(x_0)\big)^2\big]+ \E_{\PP^N}\big[\big|\widehat \pi_{\widehat {\boldsymbol h}^N}^N(t_0,x_0)-\pi(t_0,x_0)\big|^2\big] \label{eq first maj mu}
\end{align}
as soon as $\varpi_3 \leq \kappa_4$. By Theorem \ref{thm: GL mu}, we already have the desired bound for the first term.\\

\noindent {\it Step 2:} We study the second term in the right-hand side of \eqref{eq first maj mu}. For any $\boldsymbol h \in \mathcal H_2^N$, similarly to the proof of Step 1 in Theorem \ref{thm: GL mu}, we have
$$\E_{\PP^N}\big[\big|\widehat \pi_{\widehat {\boldsymbol h}^N}^N(t_0,x_0)-\pi(t_0,x_0)\big|^2\big] \lesssim \E_{\PP^N}\big[\mathsf A_{\boldsymbol h}^N\big]+\mathsf V_{\boldsymbol h}^N + \mathcal B_{\boldsymbol h}^N(\pi)(t_0,x_0)^2$$
thanks to Lemma \ref{lem: bias var pi}. In order to estimate $\E_{\PP^N}\big[\mathsf A_{\boldsymbol h}^N\big]$, we repeat Step 2 of the proof of Theorem \ref{thm: GL mu} and obtain
\begin{align}
& \max_{\boldsymbol {h}' \preceq \boldsymbol h}\big\{\big|\widehat \pi_{\boldsymbol h}^N(t_0,x_0)-\widehat \pi_{{\boldsymbol h}'}(t_0,x_0)\big|^2-\mathsf V_{\boldsymbol h}^N-\mathsf V_{{\boldsymbol h}'}^N\big\}_+  \label{eq controle sup idem} 
\\
& \lesssim  \mathcal B_{\boldsymbol h}^N(\pi)(t_0,x_0)^2+  \big\{4\big|\widehat \pi_{\boldsymbol h}^N(t_0,x_0)-\pi_{\boldsymbol h}(t_0,x_0)\big|^2-\mathsf V_{\boldsymbol h}^N\big\}_+ \nonumber \\
&\;\;\;\;+\max_{{\boldsymbol h}' \preceq {\boldsymbol h}}\big\{4\big|\widehat \pi_{{\boldsymbol h}'}^N(t_0,x_0)-\pi_{{\boldsymbol h}'}(t_0,x_0)\big|^2-\mathsf V_{{\boldsymbol h}'}^N\big\}_+\nonumber 
\end{align}
with the notation $\pi_{\boldsymbol h}(t_0,x_0) = \int_0^T\int_{\R^d} (H \otimes K)_{\boldsymbol h}(t_0-t, x_0-x)\pi(t,x)dxdt$.\\

\noindent {\it Step 3:} We estimate the expectation of the first stochastic term in the right-hand side of \eqref{eq controle sup idem}. 
%Noting that 
%$$\widehat \pi_{\boldsymbol h}^N(t_0,x_0)-\pi_{\boldsymbol h}(t_0,x_0) = \int_0^T\int_{\R^d}(H\otimes K)_{\boldsymbol h}(\pi^N(dt,dx)-\pi(t,x)dxdt),$$
By Lemma \ref{lem: dev b}, setting $\phi(t,x) = (H\otimes K)_{\boldsymbol h}(t_0-t,x_0-x)$, we have
$$\big\{4\big|\widehat \pi_{\boldsymbol h}^N(t_0,x_0)-\pi_{\boldsymbol h}(t_0,x_0)\big|^2-\mathsf V_{\boldsymbol h}^N\big\}_+ \leq I + II+III,$$
with
\begin{align*}
I & = \Big\{12\big| \int_{[0, T] \times \R^d}\phi(t,x)b(t,x,\mu_t)(\mu_t^N(dx)-\mu_t(x)dx\big)dt\big|^2-\tfrac{1}{3}\mathsf V_{\boldsymbol h}^N\Big\}_+, \\
II & =  \Big\{12\big| \int_{[0,T] \times \R^d}\phi(t,x)\xi_t^N(x)\mu_t^N(dx)dt\big|^2-\tfrac{1}{3}\mathsf V_{\boldsymbol h}^N\Big\}_+,\\
III & = \Big\{ 12\big|\mathcal M_T^N\big(\phi\big)\big|^2-\tfrac{1}{3}\mathsf V_{\boldsymbol h}^N\Big\}_+.
\end{align*}
For the term $I$, writing $b=(b^1,\ldots, b^d)$ in components, we have
\begin{align*}
I  
\leq 12T^2\sum_{k = 1}^d \Big\{\big(\int_{[0,T] \times \R^d}\phi(t,x)b^k(t,x,\mu_t)(\mu_t^N(dx)-\mu_t(x)dx)T^{-1}dt\big)^2-\tfrac{1}{36dT^2}\mathsf V_{\boldsymbol h}^N\Big\}_+.
\end{align*}
Applying Theorem \ref{thm: concentration} to $(\mu_t^N(dx)-\mu_t(x)dx)T^{-1}dt$ and using the estimates \eqref{eq: estimes phi} of Lemma \ref{lem: bias var pi} above, we infer
\begin{align*}
& \E_{\PP^N}\Big[\Big\{\big(\int_{[0,T] \times \R^d}\phi(t,x)b^k(t,x,\mu_t)(\mu_t^N(dx)-\mu_t(x)dx)T^{-1}dt\big)^2-\tfrac{1}{36dT^2}\mathsf V_{\boldsymbol h}^N\Big\}_+\Big] \\
& \lesssim \int_0^\infty \PP^N\Big(\big|\int_{[0,T] \times \R^d}\phi(t,x)b^k(t,x,\mu_t)(\mu_t^N(dx)-\mu_t(x)dx)T^{-1}dt\big| \geq \big(z+\tfrac{1}{36dT^2}\mathsf V_{\boldsymbol h}^N\big)^{1/2}\Big)dz \\
& \lesssim 2\kappa_1\int_{\mathsf V_h^N/(36dT^2)}^\infty \exp\Big(-\frac{\kappa_2Nh_1h_2^dz}{\kappa_3\kappa_5^2T^{-1}|H \otimes K|_2^2+\kappa_5|H \otimes K|_\infty z^{1/2}}\Big)dz \\
& \lesssim (Nh_1h_2^d)^{-1} N^{-(72dT)^{-1}\kappa_2\kappa_3^{-1}\kappa_5^{-2}\varpi_2} \\
&\;\;+ (\log N)^{1/2}(Nh_1h_2^d)^{-3/2}\exp\big(-\frac{\kappa_2\kappa_5^{-1}\varpi_2^{1/2}|H\otimes K|_2}{12d^{1/2}T|H\otimes K|_\infty}(\log N)^{1/2}(Nh_1h_2^d)^{1/2}\big)\\
& \lesssim N^{-2}
\end{align*}
as soon as $\varpi_2 \geq 144dT\kappa_2^{-1}\kappa_3\kappa_5^{2}$, thanks to $\max_{(h_1,h_2) \in \mathcal H_2^N}(Nh_1h_2^d)^{-1} \lesssim 1$. We also use the assumption $\min_{(h_1,h_2) \in \mathcal H_2^N}Nh_1h_2^d \geq (\log N)^2$ given by \eqref{eq: condition grid 2} to show that the second term is negligible in front of $N^{-2}$. We conclude 
 \begin{align} 
  \E_{\PP^N}\big[I\big] 
 \lesssim N^{-2}. \label{eq: conclusion I}
 \end{align}

We next consider the term $II$. For $\tau >0$, introduce the event
$$\mathcal B_\tau^N = \Big\{\int_{[0,T] \times \R^d}|\phi(t,x)|\mu_t^N(dx)dt \leq \tau\Big\}.$$
By Cauchy-Schwarz's inequality, on $\mathcal B_\tau^N$, we have
\begin{align*}
& \Big|\int_{[0,T] \times \R^d}\phi(t,x)\xi_t^N(x)\mu_t^N(dx)dt\Big|^2 \\
& \leq \int_{[0,T] \times \R^d}|\phi(t,x)|\mu_t^N(dx)dt \int_{[0,T] \times \R^d}\big|\phi(t,x)|\big|\xi_t^N(x)\big|^2\mu_t^N(dx)dt \\
& \leq \tau \int_{[0,T] \times \R^d}|\phi(t,x)|\big|\xi_t^N(x)\big|^2\mu_t^N(dx)dt.
%& \leq  \tau \int_{[0,T] \times \R^d}|\phi(t,x)|\big|\xi_t^N(x)\big|^2\mu_t^N(dx)dt 
\end{align*}
By Lemma \ref{lem: fluctuation xi} and exchangeability, we also have the rough bound
\begin{align}
\E_{\PP^N}\big[ \big|\int_{[0,T] \times \R^d}\phi(t,x)\xi_t^N(x)\mu_t^N(dx)dt\big|^4\big] 
& \leq |\phi|_\infty^4T^3\sup_{0 \leq t \leq T}\E_{\PP^N}\big[\big|\xi_t^N(X_t^N)\big|^4\big] \nonumber \\ 
& \lesssim  (h_1h_2^d)^{-4}N^{-2}, \label{eq: maj grossiere}
\end{align}
where we estimate $\E_{\PP^N}\big[\big|\xi_t^N(X_t^N)\big|^4\big]$ as in \eqref{eq: controle xi holder} above.
It follows that $II \leq IV + V$, with
\begin{align*}
IV & = 12\tau\Big\{\int_{[0,T] \times \R^d}|\phi(t,x)|\big|\xi_t^N(x)\big|^2\mu_t^N(dx)dt-\tfrac{1}{36\tau}\mathsf V_{\boldsymbol h}^N\Big\}_+ {\bf 1}_{\mathcal B_\tau^N}\\
& \leq 12\tau |K_{h_2}|_\infty \int_0^T |H_{h_1}(t_0-t)|\Big\{\int_{\R^d}\big|\xi_t^N(x)\big|^2\mu_t^N(dx)-\tfrac{1}{36\tau|H|_1 |K_{h_2}|_\infty}\mathsf V_{\boldsymbol h}^N\Big\}_+dt \; {\bf 1}_{\mathcal B_\tau^N},\\
V & = 12\big|\int_{[0,T] \times \R^d}\phi(t,x)\xi_t^N(x)\mu_t^N(dx)dt\big|^2{\bf 1}_{(\mathcal B_\tau^N)^c}.
\end{align*}
Taking expectation and using exchangeability, we further have
\begin{align*}
\E_{\PP^N}\big[IV\big] &\lesssim   |K_{h_2}|_\infty\int_0^T |H_{h_1}(t_0-t)|\E_{\PP^N}\Big[\Big\{\big|\xi_t^N(X_t^N)\big|^2-\tfrac{1}{36\tau|H|_1|K_{h_2}|_\infty}\mathsf V_{\boldsymbol h}^N\Big\}_+\Big]dt\\
& \lesssim  |K_{h_2}|_\infty\int_0^T |H_{h_1}(t_0-t)|\int_{\tfrac{1}{36\tau|H|_1|K_{h_2}|_\infty}\mathsf V_{\boldsymbol h}^N}^\infty \PP^N\big(|\xi_t^N(X_t^N)| \geq z^{1/2}\big)dzdt \\
&  \lesssim  |K_{h_2}|_\infty \int_{\tfrac{1}{36\tau |H|_1|K_{h_2}|_\infty}\mathsf V_{\boldsymbol h}^N}^\infty \exp\Big(-\frac{\kappa_7Nz}{1+N^{1/2}z^{1/2}}\Big)dz \\
& \lesssim  N^{-1}(h_2^{d})^{-1}\exp\Big(-\frac{\kappa_7\varpi_2|H\otimes K|_2^2}{72 \tau |H|_1|K|_\infty}h_1^{-1}\log N\Big)\\
&\;\;+ N^{-1}(h_2^{d})^{-1}h_1^{-1/2}(\log N)^{1/2}
\exp\Big(-\frac{\kappa_7\varpi_2^{1/2}|H\otimes K|_2}{12\tau^{1/2}|H|_1^{1/2}|K|_\infty^{1/2}}h_1^{-1/2}(\log N)^{1/2}\Big),
\end{align*}
by Lemma \ref{lem: fluctuation xi}, using in particular the fact that $\tfrac{1}{36\tau |H|_1|K_{h_2}|_\infty}\mathsf V_{\boldsymbol h}^N \gtrsim N^{-1}(\log N)^3 \geq \kappa_8^2N^{-1}$ for large enough $N$. Since $\max_{(h_1,h_2) \in \mathcal H_2^N}h_1 \leq (\log N)^{-2}$ by assumption \eqref{eq: condition grid 2}, both terms are negligible in front of $N^{-2}$ and we conclude $\E_{\PP^N}\big[IV\big] \lesssim N^{-2}$.\\ 

We turn to the term $V$. By Cauchy-Schwarz's inequality and \eqref{eq: maj grossiere}, we have
\begin{align*}
\E_{\PP^N}[V] & \leq 12\, \E_{\PP^N}\Big[\Big|\int_{[0,T] \times \R^d}\phi(t,x)\xi_t^N(x)\mu_t^N(dx)dt\Big|^4\Big]^{1/2}\PP^N\big((\mathcal B_\tau^N)^c\big)^{1/2} \\
& \lesssim N^{-1}(h_1h_2^d)^{-2} \PP^N\big(\int_{[0,T] \times \R^d}|\phi(t,x)|\mu_t^N(dx)dt > \tau\big)^{1/2}.
\end{align*}
Note that $\int_{[0,T] \times \R^d}|\phi(t,x)|\mu_t(dx)dt \leq \kappa_3 |\phi|_1$ hence, for the choice $\tau \geq 2\kappa_3|H \otimes K|_1$, that we make from now on and that does not depend on $N$, we have
$$\int_{[0,T] \times \R^d}|\phi(t,x)|\mu_t(dx)dt  \leq \tfrac{1}{2}\tau.$$
By triangle inequality and a union bound argument, it follows that 
\begin{align*}
& \PP^N\big(\int_{[0,T] \times \R^d}|\phi(t,x)|\mu_t^N(dx)dt > \tau \big)^{1/2}  \\
%& \leq {\bf 1}_{\big(|\int_{[0,T] \times \R^d}\phi(t,x)\mu_t(dx)dt | \geq \tfrac{1}{2}\tau^{1/2}\big)} + \PP^N\big(\big|\int_{[0,T] \times \R^d}\phi(t,x)(\mu_t^N(dx)-\mu_t(dx))dt\big| \geq \tfrac{1}{2}\tau^{1/2}\big)^{1/2} \\
& \leq \PP^N\big(\big|\int_{[0,T] \times \R^d}|\phi(t,x)|(\mu_t^N(dx)-\mu_t(dx))T^{-1}dt\big| \geq \tfrac{1}{2}T^{-1}\tau\big)^{1/2} \\
& \leq (2\kappa_1)^{1/2} \exp\Big(-\tfrac{1}{2}\frac{\kappa_2Nh_1h_2^d\tfrac{1}{4}T^{-2}\tau^2}{\kappa_3T^{-1}|H\otimes K|_2^2+|H\otimes K|_\infty\tfrac{1}{2}T^{-1}\tau}\Big),
\end{align*}
%as soon as $\tau \geq 4\kappa_3T^2|H \otimes K|_1^2$ using that $|\int_{[0,T] \times \R^d}\phi(t,x)\mu_t(dx)dt| \leq T\kappa_3 |\phi|_1$ in order to get rid of the indicator, and
where we applied Theorem \ref{thm: concentration}. Using $\min_{(h_1,h_2) \in \mathcal H_2^N}Nh_1h_2^d \geq (\log N)^2$ granted by \eqref{eq: condition grid 2}, we obtain $\E_{\PP^N}[V] \lesssim N^{-2}$. Putting together our estimates for $IV$ and $V$, we conclude
\begin{equation} \label{eq: conclusion I bis}
\E_{\PP^N}[II] \lesssim N^{-2}.
\end{equation}

Finally, we consider the term $III$. The classical following deviation bound holds for continuous martingales: 
$$\PP^N\big(\mathcal M_T^N(\phi)^k\geq u, \big\langle \mathcal M_.^N(\phi)^k \big\rangle_T \leq v\big) \leq \exp\big(-\frac{u^2}{2v}\big)$$
for every $u,v \geq 0$, see {\it e.g.} \cite{RY99}. Let $\kappa >0$ to be tuned below. The choice  $v = \kappa \big(\mathsf V_{\boldsymbol h}^N\big)^{1/2}(\log N)^{-1}u$  entails
\begin{align*}
 \PP^N\big(\mathcal M_T^N(\phi)^k \geq u\big) & \leq 
 \exp\big(-\tfrac{1}{2}\kappa^{-1}\big(\mathsf V_{\boldsymbol h}^N\big)^{-1/2}(\log N)u\big)+\PP^N\big(\big\langle \mathcal M_.^N(\phi)^k \big\rangle_T  \geq \kappa \big(\mathsf V_{\boldsymbol h}^N\big)^{1/2}(\log N)^{-1}u\big)  \label{eq: dev mg renorm}.
\end{align*}
It follows that 
\begin{align*}
\E_{\PP^N}\big[III\big] & \leq 12\sum_{k =1}^d\E_{\PP^N}\big[\big\{\big(\mathcal M_T^N(\phi)^k\big)^2-\tfrac{1}{36d}\mathsf V_{\boldsymbol h}^N\big\}_+\big]  \\
& \lesssim \sum_{k = 1}^d \int_0^\infty \PP^N\big(\big|\mathcal M_T^N(\phi)^k\big| \geq (z+\tfrac{1}{36d}\mathsf V_{\boldsymbol h}^N)^{1/2}\big)dz  \lesssim VI +VII,
\end{align*}
with
\begin{align*}
VI & =  \int_{\tfrac{1}{36d}\mathsf V_{\boldsymbol h}^N}^\infty \exp\big(-\tfrac{1}{2} \kappa^{-1}\big(\mathsf V_{\boldsymbol h}^N\big)^{-1/2}(\log N)z^{1/2}\big)dz, \\
VII & = \sum_{k = 1}^d\int_{\tfrac{1}{36d}\mathsf V_{\boldsymbol h}^N}^\infty\PP^N\big(\big\langle \mathcal M_.^N(\phi)^k \big\rangle_T  \geq \kappa\big(\mathsf V_{\boldsymbol h}^N\big)^{1/2}(\log N)^{-1}z^{1/2}\big)dz. 
\end{align*}
Taking for instance $\kappa =  (25\sqrt{d})^{-1} < (24\sqrt{d})^{-1}$, we obtain 
\begin{align*}
VI \lesssim \mathsf V_{\boldsymbol h}^N (\log N)^{-1} N^{-1/(12\kappa d^{1/2})}  \lesssim N^{-2}.
\end{align*}
In order to bound the term $VII$, we first notice that by Lemma \ref{lem: dev b}, we have
\begin{align*}
\langle \mathcal M_.^N(\phi)^k \big\rangle_T  & \leq N^{-1}|\mathrm{Tr}(c)|_\infty \int_{[0,T]\times \R^d}\phi(t,x)^2\mu_t^N(dx)dt \\
& \leq  N^{-1}|\mathrm{Tr}(c)|_\infty \Big(\kappa_3 (h_1h_2^d)^{-1} |H\otimes K|_2^2 +\big| \int_{[0,T]\times \R^d}\phi(t,x)^2(\mu_t^N(dx)-\mu_t(dx))dt\big|\Big).
\end{align*}
Next, the condition
$$N^{-1}|\mathrm{Tr}(c)|_\infty \kappa_3 (h_1h_2^d)^{-1} |H\otimes K|_2^2 \geq \tfrac{1}{2}\kappa\big(\mathsf V_{\boldsymbol h}^N\big)^{1/2}(\log N)^{-1}z^{1/2}$$
is equivalent to
\begin{align*}
z & \leq  4\kappa^{-2} |\mathrm{Tr}(c)|_\infty^2 \kappa_3^2|H\otimes K|_2^4 N^{-2}(h_1h_2^d)^{-2} \big(\mathsf V_{\boldsymbol h}^N\big)^{-1}(\log N)^{2} \\
& \leq  4\kappa^{-2} |\mathrm{Tr}(c)|_\infty^2 \kappa_3^2\varpi_2^{-2} \mathsf V_{\boldsymbol h}^N\\
& <  \frac{1}{36d}\mathsf V_{\boldsymbol h}^N
\end{align*}
as soon as $\varpi_2 \geq 300 d|\mathrm{Tr}(c)|_\infty \kappa_3$. It follows that $VII$ is of order
\begin{align*}
& \int_{\tfrac{1}{36d}\mathsf V_{\boldsymbol h}^N}^\infty\PP^N\Big(\big| \int_{[0,T]\times \R^d}\phi(t,x)^2(\mu_t^N(dx)-\mu_t(dx))dt\big| \geq \tfrac{1}{2}N|\mathrm{Tr}(c)|_\infty^{-1}\kappa\big(\mathsf V_{\boldsymbol h}^N\big)^{1/2}(\log N)^{-1}z^{1/2}\Big)dz \\
& \lesssim  \int_{\tfrac{1}{36d}\mathsf V_{\boldsymbol h}^N}^\infty\exp\Big(-\frac{\kappa_2\tfrac{1}{4}N^3|\mathrm{Tr}(c)|_\infty^{-2}\kappa^2\big(\mathsf V_{\boldsymbol h}^N\big)(\log N)^{-2}T^{-2}z}{\kappa_3T^{-1}|\phi|_4^4+|\phi|_\infty^2 \tfrac{1}{2}N|\mathrm{Tr}(c)|_\infty^{-1}\kappa\big(\mathsf V_{\boldsymbol h}^N\big)^{1/2}(\log N)^{-1}T^{-1}z^{1/2}}\Big)du \\
& \lesssim (Nh_1h_2^d)^{-2}(\log N)\exp\Big(-\frac{\kappa^2\kappa_2|\mathrm{Tr}(c)|_\infty^{-2}\varpi_2^2|H \otimes K|_2^4}{288 \kappa_3 d T |H\otimes K|_4^4}Nh_1h_2^d\Big)\\
&\;\;+(Nh_1h_2^d)^{-2}(\log N)\exp\Big(-\frac{\kappa_2\kappa|\mathrm{Tr}(c)|_\infty^{-1}\varpi_2|H \otimes K|_2^2}{24d^{1/2}T|H \otimes K|_\infty^2}Nh_1h_2^d\Big) \\
& \lesssim N^{-2}
\end{align*}
where we applied again Theorem \ref{thm: concentration} and used $\min_{(h_1,h_2) \in \mathcal H_2^N}Nh_1h_2^d \geq (\log N)^2$ granted by \eqref{eq: condition grid 2}. We infer $VII \lesssim N^{-2}$ and
conclude 
\begin{equation} \label{eq: conclusion I ter}\E_{\PP^N}[III] \lesssim N^{-2}
\end{equation}
and putting together \eqref{eq: conclusion I}, \eqref{eq: conclusion I bis} and \eqref{eq: conclusion I ter}, we have established
$$\E_{\PP^N}\big[\big\{4\big|\widehat \pi_{\boldsymbol h}^N(t_0,x_0)-\pi_{\boldsymbol h}(t_0,x_0)\big|^2-\mathsf V_{\boldsymbol h}^N\big\}_+\big] \lesssim N^{-2}.$$

\noindent {\it Step 4:} The control of the second term in the right-hand side of \eqref{eq controle sup idem} is done in the same way as in Step 4 
of the proof of Theorem \ref{thm: GL mu} and only inflates the previous bound by a factor or order $\mathrm{Card}(\mathcal H_2^N) \lesssim N$.
In turn $\E_{\PP^N}\big[\mathsf A_{\boldsymbol h}^N\big] \lesssim N^{-1} + \mathcal B_{\boldsymbol h}^N(\pi)(t_0,x_0)^2$ and we have established by Step 2 that for any $\boldsymbol h \in \mathcal H_2^N,$
%\begin{equation} \label{eq oracle gamma}
$$
\E_{\PP^N}\big[\big|\widehat \pi_{\widehat{\boldsymbol h}^N}^N(t_0,x_0)-\pi(t_0,x_0)\big|^2\big] \lesssim \mathcal B_{\boldsymbol h}^N(\pi)(t_0,x_0)^2+\mathsf V_{\boldsymbol h}^N + N^{-1}
%\end{equation}
$$
holds true. Putting together Step 1 and Theorem \ref{thm: GL mu} and using $N^{-1} \lesssim \mathsf V_{\boldsymbol h}^N$ completes the proof of Theorem \ref{thm: oracle b}.

\subsection{Proof of Theorem \ref{thm minimax mu}}

\subsubsection*{Proof of the lower bound \eqref{eq: LB mu} } 
%{\color{red} [word about strategy of proof]}.\\

\subsubsection*{Step 1} Pick an infinitely many times differentiable  function $V_1 :\R^d \rightarrow \R$ such that 
\begin{itemize}
\item[(i)] $\nabla V_1$ is Lipschitz continuous, 
\item[(ii)]  $\limsup_{|x| \rightarrow \infty} -\nabla V_1(x)^\top \big(x/|x|^2\big) <0$,
\item[(iii)] $V_1 = 0$ in a neighbourhood of $x_0$.
\end{itemize}
Let $C_{V_1} = \int_{\R^d} \exp(-2V_1(x))dx$ and define
$$\nu_1(x) = C_{V_1}^{-1}\exp\big(-2V_1(x)\big),\;\;x \in \R^d.$$
From the classical theory of multidimensional diffusion processes (see {\it e.g.} the classical textbook of Stroock and Varadhan \cite{SV79}), properties (i) and (ii) imply that 
$\nu_1$ is the unique invariant measure of the diffusion process $d\xi_t = -\nabla V_1(\xi_t)dt+dW_t$ for some Brownian motion $W$ on $\R^d$. In turn $\nu_1(t,x) = \nu_1(x)$ as a function defined on $[0,T]\times \R^d$ satisfies
$$\nu_1 = \mathcal S(b_1,\mathrm{Id},\nu_1)\;\;\text{with}\;\;b_1(t,x,\mu) = -\nabla V_1(x),$$
assuming morever that $V_1$ is such that  $\nu_1$ satisfies Assumption \ref{ass: init condition}, a choice which is obviously possible. 
Since $\nu_1$ is constant in a neighbourhood of $(t_0,x_0)$ we may (and will) assume that $(b_1,\mathrm{Id},\nu_1) \in \mathcal S_{L/2}^{\alpha,\beta}(t_0,x_0)$. Next, we set
$$V_2^N(x) = V_1(x)+\varpi  C_{V_1}N^{-1/2}\tau_N^{d/2}\psi\big(\tau_N(x-x_0)\big),\;\;0 < \tau_N\rightarrow \infty,$$
for some $0 < \varpi \leq 1$, where $\psi : \R^d \rightarrow \R$ is  infinitely many times differentiable, compactly supported and satisfies 
$$\psi(0)=1,\;\;|\psi|_\infty \leq 1,\;\;\int_{\R^d} \psi(x)dx = 0,\;\;|\psi|_2 = 1,$$
%Let $\mathcal V_{x_0}$ be a neighbourhood of $x_0$ such that $V_1(x) = 0$ on $\mathcal V_{x_0}$. For large enough $N$ we have $\psi\big(\tau_N(x-x_0)\big)=0$ outside $\mathcal V_{x_0}$. In turn, 
%\begin{equation} \label{eq: rep locale}
%V_2^N(x) = V_1(x)+\varpi C_{V_1}N^{-1/2}\tau_N^{d/2}\psi\big(\tau_N(x-x_0)\big)\;\;\text{for every}\;\;x \in \R^d,
%\end{equation}
hence $\nabla V_2^N$ satisfies (i) and (ii) and (iii).
It defines in turn a solution 
$$\nu_2^N = \mathcal S(b_2^N,\mathrm{Id},\nu_2^N)\;\;\text{with}\;\;b_2^N(t,x,\mu) = -\nabla V_2^N(x),$$
having
$$\nu_2^N(x) = C_{V_2^N}^{-1}\exp\big(-2V_2^N(x)\big)\;\;\text{with}\;\;C_{V_2^N} = \int_{\R^d} \exp(-2V_2^N(x))dx$$ 
and $\nu_2^N(t,x) = \nu_2^N(x)$ is understood as a function defined on $[0,T] \times \R^d$.\\

\noindent {\it Step 2:} We claim that for every $\beta >0$, setting $\tau_N = N^{1/(2\beta+d)}$ and taking $\varpi$ sufficiently small, we have $(b_2^N,\mathrm{Id},\nu_2^N) \in \mathcal S_L^{\alpha,\beta}(t_0,x_0)$ 
for large enough $N$. Indeed, in a neighbourhood of $x_0$, we have $V_1=0$ hence for $x$ in such a neighbourhood, we have
%for $\beta = k+\{\beta\}$ having $0< \{\beta\} \leq 1$ and $k \geq 0$, we have
%$$\frac{d^k}{dx^k}$$
\begin{align*}
\nu_2^N(x) 
%& = C_{V_2^N}^{-1}\exp(-2V_1(x)-2\varpi \nu_1(x)^{-1}N^{-1/2}\tau_N^{d/2}\psi\big(\tau_N(x-x_0)\big) \\
& = C_{V_2^N}^{-1}\exp\big(-2\varpi C_{V_1}N^{-1/2}\tau_N^{d/2}\psi\big(\tau_N(x-x_0)\big)\big).
\end{align*}
%by \eqref{eq: rep locale}.
On the one hand, $\varpi N^{-1/2}\tau_N^{d/2}\big|\psi\big(\tau_N(\cdot-x_0)\big)|_{\mathcal H^\beta(x_0)} \lesssim \varpi|\psi|_{\mathcal H^\beta(x_0)}$ which can be taken arbitrarily small. 
On the other hand, we also have $C_{V_2^N} \rightarrow C_{V_1}$ as $N \rightarrow \infty$, see in particular \eqref{eq: conv V2 vers C1} below, and the claim follows.

\subsubsection*{Step 3} For $(b,c,\mu_0) \in \mathcal P$, we write $\PP^N_{b,c,\mu_0}$ for $\PP^N$ to emphasise the model parameter $(b,c,\mu_0)$.  For data extracted from $\mu^N_{t_0}$ solely, we restrict the model to  
$$\PP^N_{b,c,\mu_0}(t_0) = (X_{t_0}^1,\ldots, X_{t_0}^N) \circ \PP^N_{b,c,\mu_0},$$
the law of $(X_{t_0}^1,\ldots, X_{t_0}^N)$ under $\PP^N_{b,c,\mu_0}$.  
Note that for a drift $b(t,x,\mu) = b(t,x)$ independent of an interaction measure term $\mu$, we have
$$\PP^N_{b,c,\mu_0} = \overline{\PP}_{b,c,\mu_0}^N\;\;\text{hence}\;\;\PP^N_{b,c,\mu_0}(t_0) = \mu_{t_0}^{\otimes N}.$$
By Pinsker's inequality, it follows that 
\begin{align*}
\|\PP^N_{b_1,\mathrm{Id},\nu_1}(t_0) -\PP^N_{b_2^N,\mathrm{Id},\nu_2^N}(t_0) \|_{TV}^2 & = \|\nu_1^{\otimes N}-  (\nu_2^N)^{\otimes N}\|_{TV}^2 \\
& \leq \frac{N}{2}\int_{\R^d}\nu_1(x)\log \frac{\nu_1(x)}{\nu_2^N(x)}dx \\
& = N\int_{\R^d} \nu_1(x)\big(V_2^N(x)-V_1(x)\big)dx+\frac{N}{2}\log \frac{C_{V_2^N}}{C_{V_1}} \\
& = \frac{N}{2}\log \frac{C_{V_2^N}}{C_{V_1}},
\end{align*}
for large enough $N$, where $\|\cdot\|_{TV}$ denotes the total variation distance, using successively
\begin{align*}
V_2^N(x)-V_1(x) & = \varpi  C_{V_1}N^{-1/2}\tau_N^{d/2}\psi\big(\tau_N(x-x_0)\big)  =  \varpi  \nu_1(x)^{-1}N^{-1/2}\tau_N^{d/2}\psi\big(\tau_N(x-x_0)\big), 
\end{align*}
since $\nu_1(x)^{-1} = C_{V_1}$ in a neighbourhood of $x_0$ and the fact that $\psi\big(\tau_N(x-x_0)\big) =0$ outside this neighbourhood, for large enough $N$, thanks to the fact that $\mathrm{Supp}(\psi)$ is compact, together with 
the cancellation $\int_{\R^d}\psi(x)dx=0$. Moreover, a Taylor's expansion yields
\begin{align*}
\frac{C_{V_2^N}}{C_{V_1}}-1 & = \int_{\R^d} \nu_1(x)\exp(-2\varpi \nu_1(x)^{-1}N^{-1/2}\tau_N^{d/2}\psi(\tau_N(x-x_0))dx-1 \\ 
& = 2\varpi ^2N^{-1} \tau_N^d \int_{\R^d} \nu_1(x)^{-1}\psi(\tau_N(x-x_0))^2\vartheta^N(x)dx,  
\end{align*}
thanks to the cancellation property of $\psi$ again, with a remainder term satisfying
$$0 \leq \vartheta^N(x) \leq \exp\big(2\varpi N^{-1/2}\tau_N^{d/2}|\psi|_\infty\sup_{x \in \mathrm{Supp}(\psi)}\nu_1(x)^{-1}\big) \leq 2$$ 
for large enough $N$. It follows that
\begin{equation} \label{eq: conv V2 vers C1}
\Big|\frac{C_{V_2^N}}{C_{V_1}}-1 \Big|  \leq 4\varpi ^2N^{-1}\sup_{x \in \mathrm{Supp}(\psi)}\nu_1(x)^{-1} \lesssim N^{-1}.
\end{equation}
The inequality $\log(1+x) \leq x$ for $x \geq -1$ enables us to conclude
\begin{equation} \label{eq: controle pinsker}
\|\PP^N_{b_1,c,\nu_1}(t_0) -\PP^N_{b_2^N,c,\nu_2^N}(t_0) \|_{TV}^2 \leq 2\varpi^2\sup_{x \in \mathrm{Supp}(\psi)}\nu_1(x)^{-1} \leq \tfrac{1}{2}
\end{equation}
for large enough $N$ by taking $\varpi >0$ sufficiently small. 
\subsubsection*{Step 4}
We conclude by a classical two-point lower bound argument using Le Cam's lemma: if $\mathbb P_i$, $i=1,2$ are two probability measures defined on the same probability space and $\Psi(\PP_i) \in \R$ is a functional of $\mathbb P_i$, we have
\begin{equation} \label{eq LeCam}
\inf_{\widehat \Psi}\max_{i = 1,2}\E_{{\mathbb P}_i}\big[|\widehat \Psi-\Psi(\PP_i)|\big] \geq \tfrac{1}{2}|\Psi(\PP_1)-\Psi(\PP_2)|(1-\|\mathbb P_1-\mathbb P_2\|_{TV}),
\end{equation}
where the infimum is taken over all estimators of $\Psi(\PP_i)$, see {\it e.g.} \cite{LECAM} among many other references. We let
$$\Psi\big(\PP_{b_1,\mathrm{Id},\nu_1}^N(t_0)\big) = \nu_1(t_0,x_0),\;\;\Psi\big(\PP_{b_2^N,\mathrm{Id},\nu_2^N}^N(t_0)\big) = \nu_2^N(t_0,x_0),$$
so that
\begin{align*}
\big|\Psi\big(\PP_{b_1,\mathrm{Id},\nu_1}^N(t_0)\big) -\Psi\big(\PP_{b_2^N,\mathrm{Id},\nu_2^N}^N(t_0)\big)\big| & \gtrsim  \nu_1(x_0)\Big(\frac{C_{V_2^N}}{C_{V_1}}-\exp\big(2\varpi \nu_1(x)^{-1}N^{-1/2}\tau_N^{d/2}\psi(0)\big)\Big) \\
& \gtrsim N^{-1/2}\tau_N^{d/2} = N^{-\beta/(2\beta+d)}
\end{align*}
in the same way as before, using the properties of $\psi$ and \eqref{eq: conv V2 vers C1}. We conclude by applying Le Cam's lemma together with \eqref{eq: controle pinsker}. The proof of the lower bound \eqref{eq: LB mu} is complete.

\subsubsection*{Proof of the upper bound \eqref{eq: UB mu}} The argument is classical (see {\it e.g.} \cite{GL08, GL11, GL14}). Pick 
$$\mathcal H_1^N = \big\{\mathrm{e}^{-k}, 1 \leq k \leq d^{-1}\log N - 2d^{-1}\log \log N\big\}.$$
We have $\mathrm{Card}\, \mathcal H_1^N \lesssim N$ holds as well (and is actually much smaller). Moreover, for every $h \in \mathcal H_1^N$:
$$\mathcal B_h^N(\mu)(t_0,x_0)^2 \lesssim h^{2\beta\wedge \ell}\;\;\text{and}\;\;\mathsf{V}_h^N \lesssim N^{-1}h^{-d}(\log N).$$
Applying Theorem \ref{thm: GL mu}, we obtain 
\begin{align*}
 \E_{\PP^N}\big[\big(\widehat \mu_{\mathrm{GL}}^N(t_0,x_0)-\mu_{t_0}(x_0)\big)^2\big] & \lesssim 
 \min_{h \in \mathcal H_1^N} \big(h^{2\beta \wedge \ell}+N^{-1}h^{-d}(\log N)\big) \\
 & \lesssim \Big(\frac{\log N}{N}\Big)^{2\beta\wedge \ell/(2\beta\wedge \ell+d)},
 \end{align*}
for large enough $N$, since for every $\beta \in (0,\ell]$, we have $\mathrm{e}^{-(k^N+1)} \leq (N/\log N)^{-1/(2\beta+d)} \leq \mathrm{e}^{-k^N}$ with $k^N = \lfloor \tfrac{1}{2\beta+d}(\log N-\log\log N)\rfloor$. This proves \eqref{eq: UB mu} and completes the proof of Theorem \ref{thm minimax mu}.

\subsection{Proof of Theorem \ref{thm minimax b}}
\subsubsection*{Proof of the lower bound \eqref{eq: LB b}} We apply the same strategy as for Theorem \ref{thm minimax mu}, establishing a two-point inequality and applying Le Cam's lemma for two drift functions that have no interaction. We write  $\PP_{b,c,\mu_0}^N$ for the law of $(X_t^1, \ldots, X_t^N)_{0 \leq t \leq T}$ parametrised by $(b,c,\mu_0)$. We start with the following simple consequence of Girsanov's theorem:
\begin{lem} \label{eq: two point girsanov}
For $i=1,2$, let $b_i(t,x,\mu) = b_i(t,x)$ be two drift functions that satisfy Assumption \ref{ass: basic lip} with no interaction. Set $\Delta(t,x) = b_2(t,x)-b_1(t,x)$. We have
$$\|\PP_{b_2,\mathrm{Id},\mu_0}^N-\PP_{b_1,\mathrm{Id},\mu_0}^N\|_{TV}^2 \leq \tfrac{N}{4}\int_0^T|\Delta(t,\cdot)|_{L^2(\mu_t)}^2dt,$$
where $\mu = \mathcal S(b_1, \mathrm{Id}, \mu_0)$ is a solution of \eqref{eq: mckv approx} with parameter $(b_1, \mathrm{Id}, \mu_0)$, up to an explicitly computable constant that only depends on $\mu_0$ and $b_1$.
\end{lem}
\begin{proof}
With the notation of Section \ref{sec: preparation},
 by Girsanov's theorem,
 $$
\frac{d\PP_{b_2,\mathrm{Id},\mu_0}^N}{d\PP_{b_1,\mathrm{Id},\mu_0}^N} = \mathcal E_T\big(M_\cdot^N(\Delta)\big),\;\;\text{where}\;\;
M_t^N(\Delta) = \sum_{i = 1}^N\int_0^t \Delta(s,X_s^i)^\top dB_s^i$$ 
is a $\PP^N_{b_1,\mathrm{Id},\mu_0}$-martingale, with $B^i_t = \int_0^t \big(dX_s^i - b_1(s,X_s^i)ds\big)$.
% are independent $d$-dimensional Brownian motions under Thus $M_\cdot^N(b_1,b_2)$ is 
Moreover, since there is no interaction term in the drift $b_1$, we have
% by Theorem \ref{thm: concentration}
%\begin{align*}
$$
\E_{\PP_{b_1,\mathrm{Id},\mu_0}^N}\big[\big\langle M_\cdot^N(\Delta)\big\rangle_T\big]  =  N\int_0^T\E_{\PP_{b_1,\mathrm{Id},\mu_0}^N}\big[\big|\Delta(t,X_t^1)\big|^2\big]dt  =   N\int_0^T\big|\Delta(t,\cdot)\big|^2_{L^2(\mu_t)}dt.
$$
%& \leq \tfrac{N}{4}\int_0^T\Big(\int_{\R^d}\big|\Delta(t,x)\big|^2\mu_t(dx) + \E_{\PP_{b_1,\mathrm{Id},\mu_0}^N}\Big[\big|\int_{\R^d}\big|\Delta(t,x)\big|^2(\mu_t^N-\mu_t)(dx)\big|\Big]\Big)dt \\
%& \lesssim N\int_0^T\big|\Delta(t,\cdot)\big|_{L^2(\mu_t)}^2dt+ N\int_0^T\int_0^\infty \PP_{b_1,\mathrm{Id},\mu_0}^N\Big(\big|\int_{\R^d}\big|\Delta(t,x)\big|^2(\mu_t^N-\mu_t)(dx)\big| \geq z\Big)dzdt \\
%& \lesssim N\int_0^T\big|\Delta(t,\cdot)\big|_{L^2(\mu_t)}^2dt+ N\int_0^T\int_0^\infty \exp\Big(-\frac{\kappa_2Nz^2}{|\Delta(t,\cdot)|_{L^4(\mu_t)}^4+|\Delta(t,\cdot)|_{\infty}^2z}\Big)dz dt\\
%& \lesssim \int_0^T \big(N|\Delta(t,\cdot)|_{L^2(\mu_t)}^2 + N^{1/2}|\Delta(t,\cdot)|_{L^4(\mu_t)}^4+|\Delta(t,\cdot)|_\infty^2\big)dt.
%\end{align*}
By Pinsker's inequality $\|\PP_{b_1,\mathrm{Id},\mu_0}^N-\PP_{b_2,\mathrm{Id},\mu_0}^N\|_{TV}^2   \leq \tfrac{1}{2}\E_{\PP_{b_1,\mathrm{Id},\mu_0}^N}\big[\log \frac{d\PP_{b_1,\mathrm{Id},\mu_0}^N}{d\PP_{b_2,\mathrm{Id},\mu_0}^N}\big]$. The conclusion follows from
\begin{align*}
 %\|\PP_{b_1,\mathrm{Id},\mu_0}^N-\PP_{b_2,\mathrm{Id},\mu_0}^N\|_{TV}^2  & \leq \tfrac{1}{2}
 \E_{\PP_{b_1,\mathrm{Id},\mu_0}^N}\Big[\log \frac{d\PP_{b_1,\mathrm{Id},\mu_0}^N}{d\PP_{b_2,\mathrm{Id},\mu_0}^N}\Big]   = -\E_{\PP_{b_1,\mathrm{Id},\mu_0}^N}\Big[\log \frac{d\PP_{b_2,\mathrm{Id},\mu_0}^N}{d\PP_{b_1,\mathrm{Id},\mu_0}^N}\Big] 
 = \tfrac{1}{2}\E_{\PP_{b_1,\mathrm{Id},\mu_0}^N}\big[\big\langle M_\cdot^N(\Delta)\big\rangle_T\big] \\
%& = N\tfrac{1}{4}\int_0^T\E_{\PP_{b_1,\mathrm{Id},\mu_0}^N}\Big[\int_{\R^d}\big|b_2(t,x)-b_1(t,x)\big|^2\mu_t^N(dx)\Big]dt
\end{align*}
and the previous estimate on $\big\langle M_\cdot^N(\Delta)\big\rangle_T$ under $\PP_{b_1,\mathrm{Id},\mu_0}^N$.
\end{proof}
\subsubsection*{Step 1} Pick now a function $b_1: [0,T] \times \R^d \rightarrow \R^d$ satisfying Assumption \ref{ass: basic lip} and any initial condition $\mu_0$ such that $(b_1,\mathrm{Id}, \mu_0) \in  \mathcal S^{\alpha, \beta}_{L/2}(t_0,x_0) \cap \mathcal D^{\alpha,\beta}_{L/2}(t_0,x_0)$.\\

 Let $\psi  = (\psi^1,\ldots, \psi^d) : (0,T) \times \R^d \rightarrow \R^d$ be infinitely many times differentiable, compactly supported and such that 
for every $1 \leq k \leq d$, we have
$$\psi^k(0) = 1,\;\;|\psi^k|_\infty \leq 1,\int_{[0,T] \times \R^d} \psi(t,x)\, dtdx=0,\;\;|\psi^k|_2 = 1.$$
For $N \geq 1$ and some $0 < \varpi \leq 1$, we define
$$b_2^N(t,x) =b_1(t,x)+\varpi N^{-1/2}\tau_N^{1/2}(\widetilde \tau_N)^{1/2}\psi\big(\tau_N(t-t_0), \widetilde \tau_N(x-x_0)\big),$$
where $\tau_N$ and $\widetilde \tau_N$ are defined via
$$\tau_N^{\alpha} = (\widetilde \tau_N)^\beta = N^{s_d(\alpha,\beta)/(2s_d(\alpha,\beta)+1)}.$$
By accomodating $b_1$ and $\varpi>0$, we may (and will) assume that $(b_2^N,\mathrm{Id}, \mu_0) \in \mathcal D^{\alpha,\beta}_L \cap \mathcal D^{\alpha,\beta}_{L}(t_0,x_0)$  and $b_2^N \in \mathcal P$ for every $N \geq 1$. Setting $\Delta^N(t,x) = b_2^N(t,x)-b_1(t,x)$ and noting that 
$$\int_0^T|\Delta^N(t,\cdot)|_{L^2(\mu_t)}^2dt \lesssim \varpi^2N^{-1},
%\;\;\int_0^T|\Delta^N(t,\cdot)|_{L^4(\mu_t)}^4dt \lesssim \varpi^4N^{-1},\;\;|\Delta^N|_\infty^2 \lesssim \varpi^2
$$
thanks to the compactness of the support of $\psi$ and Lemma \ref{lem: loc unif mu above},
we obtain
\begin{align*}
\|\PP_{b_1,\mathrm{Id},\mu_0}^N-\PP_{b_2^N,\mathrm{Id},\mu_0}^N\|_{TV}^2 \lesssim \varpi^2\leq \tfrac{1}{2} 
\end{align*}
by Lemma \ref{eq: two point girsanov}, for a suitable choice of $\varpi >0$.
\subsubsection*{Step 2} We conclude in the same way as in the proof of the lower bound of Theorem \ref{thm minimax mu}: by Le Cam's lemma \eqref{eq LeCam}, with
$$\Psi(\PP_{b_1,\mathrm{Id},\mu_0}^N) = b_1(t_0,x_0)\;\;\text{and}\;\;\Psi(\PP_{b_2^N,\mathrm{Id},\mu_0}^N) = b_2^N(t_0,x_0),$$
we have
$$\big|\Psi(\PP_{b_1,\mathrm{Id},\mu_0}^N)-\Psi(\PP_{b_2^N,\mathrm{Id},\mu_0}^N)\big| \gtrsim N^{-1/2}\tau_N^{1/2}(\widetilde \tau_N)^{1/2} = N^{-s_d(\alpha,\beta)/(2s_d(\alpha,\beta)+1)}$$
and the conclusion follows. The proof of the lower bound \eqref{eq: LB b} is complete.

\subsubsection*{Proof of the upper bound \eqref{eq: UB b}}
%We repeat the argument of the proof of Theorem \ref{thm minimax mu} following \cite{GL08, GL11, GL14}. 
Let 
$$\delta^N(k) = \frac{g^{-1}(k/\log N)}{\alpha(g^{-1}(k/\log N))}k,$$
where $g^{-1}$ is the inverse of the function $g(\beta) = \frac{1}{\beta} \frac{s_d(\alpha(\beta), \beta)}{2s_d(\alpha(\beta), \beta)+1}$, which is non-increasing for $\beta >0$ thanks to the assumption that $\beta \mapsto \alpha(\beta)$ is non-decreasing and $\beta \mapsto \alpha(\beta)/\beta$ is non-increasing. 
%The condition \eqref{eq: condition grid 2} for the grid $\mathcal H_2^N$ is satisfied and $\mathrm{Card}\, \mathcal H_2^N \lesssim N$ holds as well (and is actually much smaller). 
Pick 
\begin{align*}
\mathcal H_2^N&  = \Big\{(\mathrm{e}^{-k_1},\mathrm{e}^{-k_2}), k_1 = \delta^N(k_2), 1 \leq  k_2 \leq (d+1)^{-1}(\log N -2\log\log N),
% 2 \log \log N \leq k_1 \leq (d+1)^{-1}(\log N -2\log\log N),
\\
&\;\;\;\;\;\;\;\;\;2\log \log N \leq \delta^N(k_2) \leq (d+1)^{-1}(\log N -2\log\log N)\Big\}.
\end{align*}
The condition \eqref{eq: condition grid 2} for the grid $\mathcal H_2^N$ is satisfied and $\mathrm{Card}\, \mathcal H_2^N \lesssim N$ holds as well (and is actually much smaller). 
Now, set $\beta_{k_2} = g^{-1}(k_2/\log N)$. We define an ordering $\preceq$ on $\mathcal H_2^N$ that has the right behaviour with respect to the bias of $\pi$ at scale $\boldsymbol h$. We say that $\boldsymbol h = (h_1,h_2) \preceq \boldsymbol h' = (h_1',h_2')$ if
\begin{equation} \label{eq: bias order}
h_1(k_2)^{\alpha(\beta_{k_2})}+h_2(k_2)^{\beta_{k_2}} \leq h_1'(k_2')^{\alpha(\beta_{k_2'})}+h_2'(k_2')^{\beta_{k_2'}},
\end{equation}
where we write $\boldsymbol h = (h_1,h_2) = (h_1(k_2), h_2(k_2)) = (\mathrm{e}^{-\delta^N(k_2)},\mathrm{e}^{-k_2})$ and likewise for $\boldsymbol h'$.
We have that $\boldsymbol h \preceq \boldsymbol h'$ is equivalent to $k_2 \geq k'_2$ since
$$h_1(k_2)^{\alpha(\beta_{k_2})}+h_2(k_2)^{\beta_{k_2}} = 2N^{-s_d(\alpha(\beta_{k_2}), \beta_{k_2})/(2s_d(\alpha(\beta_{k_2}), \beta_{k_2})+1)}$$
and the fact that $\beta \mapsto s_d(\alpha(\beta), \beta)$ is non-decreasing, following from the assumption that $\beta \mapsto \alpha(\beta)$ is non-decreasing. Hence $\boldsymbol h \preceq \boldsymbol h'$ or $\boldsymbol h' \preceq \boldsymbol h$. Moreover,
$$(b,c,\mu_0) \in  \mathcal S^{\alpha, \beta}_L(t_0,x_0) \cap \mathcal D^{\alpha,\beta}_L(t_0,x_0)\;\;\text{implies}\;\;\pi \in \mathcal H^{\alpha,\beta}(t_0,x,_0).$$
Therefore, for every $\boldsymbol h = (h_1,h_2) \in \mathcal H_2^N$:
$$\mathcal B_{\boldsymbol h}^N(\pi)(t_0,x_0)^2 \lesssim h_1^{2\alpha\wedge \ell}+h_2^{2\beta \wedge \ell}\;\;\text{and}\;\;\mathsf{V}_{\boldsymbol h}^N \lesssim N^{-1}h_1^{-1}h_2^{-d}(\log N)$$
thanks to the definition of the ordering $\preceq$ in \eqref{eq: bias order}.
It follows that 
for every $s_d(\alpha,\beta) \in (0,\ell/d]$, we have 
$$\mathrm{e}^{-(k_1^N+1)} \leq N^{-\alpha^{-1}s_d(\alpha,\beta)/(2s_d(\alpha,\beta)+1)} \leq \mathrm{e}^{-k_1^N}$$ with $k_1^N = \lfloor \tfrac{\alpha^{-1}s_d(\alpha,\beta)}{2s_d(\alpha,\beta)+1}\log N\rfloor$ and 
$$\mathrm{e}^{-(k_2^N+1)} \leq N^{-\beta^{-1}s_d(\alpha,\beta)/(2s_d(\alpha,\beta)+1)} \leq \mathrm{e}^{-k_2^N}$$ with $k_2^N = \lfloor \tfrac{\beta^{-1}s_d(\alpha,\beta)}{2s_d(\alpha,\beta)+1}\log N\rfloor$. Applying Theorem \ref{thm: oracle b}, we obtain
\begin{align*}
 \E_{\PP^N}\big[\big|\widehat \pi_{\mathrm{GL}}^N(t_0,x_0)-\pi_{t_0}(x_0)\big|^2\big] & \lesssim 
 \min_{(h_1,h_2) \in \mathcal H_1^N} \big(h_1^{2\alpha\wedge \ell}+h_2^{2\beta \wedge \ell}+N^{-1}h_1^{-1}h_2^{-d}(\log N)\big)+N^{-1} \\
 & \lesssim \Big(\frac{\log N}{N}\Big)^{2s_d(\alpha,\beta)\wedge \ell_d/(2s_d(\alpha,\beta)\wedge \ell_d+1)}.
 \end{align*}
%since for every $s_d(\alpha,\beta) \in {\color{red}(0},\ell/d]$, we have $\mathrm{e}^{-(k_1^N+1)} \leq N^{-\alpha^{-1}s_d(\alpha,\beta)/(2s_d(\alpha,\beta)+1)} \leq \mathrm{e}^{-k_1^N}$ with $k_1^N = \lfloor \tfrac{\alpha^{-1}s_d(\alpha,\beta)}{2s_d(\alpha,\beta)+1}\log N\rfloor$ and $\mathrm{e}^{-(k_2^N+1)} \leq N^{-\beta^{-1}s_d(\alpha,\beta)/(2s_d(\alpha,\beta)+1)} \leq \mathrm{e}^{-k_2^N}$ with $k_2^N = \lfloor \tfrac{\beta^{-1}s_d(\alpha,\beta)}{2s_d(\alpha,\beta)+1}\log N\rfloor$.
This proves \eqref{eq: UB mu} and completes the proof of Theorem \ref{thm minimax mu}.

\subsection{Proof of Theorem \ref{thm: identif Vlasov}}

\subsubsection*{Preliminary results}

\begin{lem} \label{lem: periodogram}
Work under Assumptions   \ref{ass: init condition}, \ref{ass: minimal prop sigma}, and \ref{ass: basic lip}. We have
%\begin{equation} \label{eq: control periodo}
$$
\sup_{\xi \in \R^d} \E_{\PP^N}\big[|\mathcal F\big(\mathcal L(\mu^N-\mu)\big)(\xi)|^2\big] \lesssim N^{-1}.
$$
%\end{equation}
\end{lem}
\begin{proof}
Writing $\varphi_1(x) = \cos(2\pi \xi^\top x)$ and $\varphi_2(x) = \sin(2\pi \xi^\top x)$, we have by Jensen's inequality and Theorem \ref{thm: concentration}
\begin{align*}
\E_{\PP^N}\big[\big|\int_{\R^d}\mathrm{e}^{-2i\pi \xi^\top x}\mathcal L(\mu^N-\mu)(dx)\big|^2\big] & = \E_{\PP^N}\big[\big|\int_{[0,T]}\int_{\R^d}\mathrm{e}^{-2i\pi \xi^\top x}(\mu_t^N-\mu_t)(dx)w(t)\rho(dt)\big|^2\big] \\
& \leq \int_{[0,T]}\E_{\PP^N}\big[\big|\int_{\R^d}\mathrm{e}^{-2i\pi \xi^\top x}(\mu_t^N-\mu_t)(dx)\big|^2\big]w(t)^2\rho(dt) \\
& \lesssim \sup_{t \in [0,T]}\max_{k=1,2}\int_0^\infty \PP^N\big(\big|\int_{\R^d}\varphi_k \,d(\mu^N_t-\mu_t)\big| \geq z^{1/2}\big)dz \\
& \lesssim 2\kappa_1\int_0^\infty \exp\Big(-\kappa_2\frac{Nz}{1+z^{1/2}}\Big) \lesssim N^{-1}
\end{align*}
using $|\varphi_k|_{L^2(\mu_t)} \leq 1$ that stems from $|\varphi_k|_\infty \leq 1$. Since this bound is uniform in $\xi \in \R^d$, the result follows. 
\end{proof}

\subsubsection*{Completion of proof of Theorem \ref{thm: identif Vlasov}}
We have
$$|\widehat{F}_{\varpi, \varpi'}^N- F|_2^2  \lesssim |\mathcal F(\widehat{F}_{\varpi, \varpi'}^N)- \mathcal F(F)|_2^2 = I+II,$$
with
\begin{align*}
I & = \big|\big(\frac{\mathcal F\big(\mathcal L\big((\widehat b_{h,\boldsymbol h}^N)_{\varpi'}^r\big)\big)\overline{\mathcal F(\mathcal L \mu^N)}}{|\mathcal F(\mathcal L \mu^N)|^2}- \mathcal F(F)\big){\bf 1}_{\{|\mathcal F(\mathcal L\mu^N) |^2 \geq  \varpi\}}\big|_2^2, \\
II & = \big|\mathcal F(F){\bf 1}_{\{|\mathcal F(\mathcal L\mu^N) |^2 <  \varpi\}}\big|_2^2.
\end{align*}
On $\{|\mathcal F(\mathcal L\mu^N) |^2 \geq  \varpi\}$ and using the fact that $|\mathcal F(\mathcal L \mu)(\xi)|>0$ almost everywhere, we write
\begin{align*}
&\frac{\mathcal F\big(\mathcal L(\widehat b_{h,\boldsymbol h}^N)_{\varpi'}^r\big) \overline{\mathcal F(\mathcal L \mu^N)}}{|\mathcal F(\mathcal L \mu^N)|^2} - \mathcal F(F) \\
& = \frac{\big(\mathcal F(\mathcal L(\widehat b_{h,\boldsymbol h}^N)_{\varpi'}^r\big)-\mathcal F(\mathcal L b)\big) \overline{\mathcal F(\mathcal L \mu^N)}}{|\mathcal F(\mathcal L \mu^N)|^2} + \frac{\mathcal F(\mathcal L b) \overline{\mathcal F(\mathcal L \mu^N)}}{|\mathcal F(\mathcal L \mu^N)|^2}- \mathcal F(F)   \\
& =  \frac{\big(\mathcal F(\mathcal L(\widehat b_{h,\boldsymbol h}^N)_{\varpi'}^r)-\mathcal F(\mathcal L b)\big) \overline{\mathcal F(\mathcal L \mu^N)}}{|\mathcal F(\mathcal L \mu^N)|^2} + \mathcal F(\mathcal L b) \Big(\frac{\overline{\mathcal F(\mathcal L \mu^N)}}{|\mathcal F(\mathcal L \mu^N)|^2}- \frac{\overline{\mathcal F(\mathcal L \mu)}}{|\mathcal F(\mathcal L \mu)|^2}\Big), 
\end{align*}
thanks to $\mathcal F(\mathcal L b) = \mathcal F(F)\cdot \mathcal F(\mathcal L\mu)$. It follows that for $r>0$, we have 
$$\E_{\mathbb P^N}\big[I\big]  \lesssim III+IV,$$
with
\begin{align*}
III & =  \varpi^{-2}\E_{\mathbb P^N}\big[\big|\mathcal F\big(\mathcal L((\widehat b_{h,\boldsymbol h}^N)_{\varpi'}^r-b)\big)\big|_2^2\big],\\
IV & =  \varpi^{-2}\E_{\mathbb P^N}\big[\big|\mathcal F(F)\overline{\mathcal F(\mathcal L \mu^N)}\mathcal F\big(\mathcal L(\mu^N-\mu)\big)\big|_2^2\big].
\end{align*}
By Parseval's identity, and the boundedness of $\mathcal L$,
% and Theorem \ref{thm: oracle b}, 
we have
\begin{align*}
III & \lesssim \varpi^{-2}\sup_{t \in \mathrm{Supp}(w), |x| \leq r}\E_{\mathbb P^N}\big[\big|\widehat b_{h,\boldsymbol h}^N(t,x)_{\varpi'}-b(x,\mu_t)\big|^2\big] 
&+ 
\varpi^{-2}\sup_{t \in \mathrm{Supp}(w)}
%\big(\int_{|x| \geq r}\E_{\mathbb P^N}\big[|\widehat b_{h,\boldsymbol h}^N(t,x)|^2\big]dx+
\int_{|x| \geq r}|b(x,\mu_t)|^2dx.
\end{align*}
%Following \eqref{eq first maj mu}, as soon as $0 < \varpi' \leq \kappa_4 = \kappa_4(r)$, we have
%\begin{align*} 
%&\sup_{ |x| \leq r} \E_{\PP^N}\big[\big|\widehat b_{h,\boldsymbol h}^N(t,x)_{\varpi'}- b(x,\mu_t)\big|^2\big] \nonumber \\
%& \lesssim (\varpi')^{-2}\sup_{|x| \leq r}\big(\E_{\PP^N}\big[\big(\widehat \mu_h^N(t,x)-\mu_{t}(x)\big)^2\big]+\E_{\PP^N}\big[\big|\widehat \pi_{h,\boldsymbol h}^N(t,x)-\pi(t,x)\big|^2\big]\big) \\
%&  \lesssim (\varpi')^{-2}\sup_{|x| \leq r}\big(\mathcal B_{h}^N(\mu)(t,x)^2+\mathsf V_h^N+ \mathcal B_{\boldsymbol h}^N(\pi)(t,x)^2+(1+r^2)\mathsf V_{\boldsymbol h}^N\big)
%\end{align*}
By Lemma \ref{lem: bias var mu} and Lemma \ref{lem: bias var pi}, we have
$$\sup_{t \in \mathrm{Supp}(w), |x| \leq r}\E_{\mathbb P^N}\big[\big|\widehat b_{h,\boldsymbol h}^N(t,x)_{\varpi'}-b(x,\mu_t)\big|^2\big] \lesssim \varpi'(r)^{-2}\big(\mathcal B_{h}^N(\mu)(t,x)+\mathcal B_{\boldsymbol h}^N(\pi)(t,x)+\mathsf V_h^N+\mathsf V_{\boldsymbol h}^N\big).$$
%uniformly in $t \in \mathrm{Supp}(w)$ and $|x| \leq r$. 
%tracking the dependence in the constants $\kappa_3 = \kappa_3(r)$ and $\kappa_5 = \kappa_5(r)$ and noticing that $\kappa_3 \lesssim 1$ and $\kappa_5 \lesssim (1+r^2)$ in $r$ as stems from Lemma \ref{lem: loc unif mu above}. 
Moreover, by Proposition \ref{prop : reg Holder} the smoothness of $F$, $G$ and $\mu_0$ entails some H\"older smoothness on $\mu$ and $b$ (hence on $\pi$) that implies in turn the estimate 
$$\sup_{t \in \mathrm{Supp}(w), |x| \leq r}\mathcal B_{h}^N(\mu)(t,x)+\mathcal B_{\boldsymbol h}^N(\pi)(t,x) \lesssim C(r)(h^\gamma+h_1^{\gamma'}+h_2^{\gamma''}),$$
for some $\gamma, \gamma',\gamma''>0$ and $C(r)$ a locally bounded function in $r$.
We infer
\begin{align*}
&\varpi^{-2}\sup_{t \in \mathrm{Supp}(w), |x| \leq r}\E_{\mathbb P^N}\big[\big|\widehat b_{h,\boldsymbol h}^N(t,x)_{\varpi'}-b(x,\mu_t)\big|^2\big] \lesssim \varpi^{-2} \varpi'(r)^{-2}C(r)u_N,
\end{align*}
with $u_N \rightarrow 0$, for a choice $(h,\boldsymbol h) = (h_N,\boldsymbol h_N) \rightarrow 0$ as $N \rightarrow \infty$.
%in $\rho \geq 0$
%\end{align*} 
%where $\varpi'(r)\rightarrow 0$ as $r \rightarrow \infty$. For the second term, we have
%% and $u_N \rightarrow 0$ as $N \rightarrow \infty$. By Lemma \ref{lem: tail esti}, we have
%$$\varpi^{-2}\sup_{t \in \mathrm{Supp}(w)}\big(\int_{|x| \geq r}\E_{\mathbb P^N}\big[|\widehat b_{h,\boldsymbol h}^N(t,x)|^2\big]dx \lesssim \big(\varpi \varpi'(r)^{-1}\big)^{-2}\max\big((Nh_1h_2^d)^{-1},h_2^{-d}\big)\exp(-\lambda r)$$
%by Lemma \ref{lem: tail esti}. 
Since $b(x,\mu_t) = G(x)+ F\star \mu_t(x)$, the second term in $III$ can be bounded as follows:
\begin{align*}
\int_{|x| \geq r}|b(x,\mu_t)|^2dx & \lesssim  \int_{|x| \geq r}|G(x)|^2dx+\int_{\R^d}\big(\int_{|x| \geq r}\big|F(x-y)\big|^2dx\big)\mu_t(dy)
%\int_{|x| \geq \rho}|\nabla V(x)|^2dx+\int_{|x| \geq \rho}\int_{\R^d}\big|\nabla W(x-y)\big|^2\mu_t(dy)dx \\
%& = \int_{|x| \geq \rho}|\nabla V(x)|^2dx+\int_{\R^d}\big(\int_{|x| \geq \rho}\big|\nabla W(x-y)\big|^2dx\big)\mu_t(dy)
\end{align*}
by Fubini's theorem, and both terms converge to $0$ at some rate $\widetilde C(r) \rightarrow 0$ as $r \rightarrow \infty$ by dominated convergence since $F$ and $G$ are in $L^2(\R^d)$. 
%We may then pick $(h, h_1, h_2) \rightarrow 0$ as $N \rightarrow \infty$ such that
%% $u_N \rightarrow 0$  and $u_N \lesssim v_N^{-1}$  (for instance $h=N^{-1/(2\beta+d)}, h_1=N^{-1/(2\beta'+1)}$ and $h_2=N^{-1/(2\beta''+d)}$by picking $\beta' \geq d$) and 
%\begin{align*}
%III & \lesssim \varpi^{-2}\big(\varpi'(r)^{-2}\big((1+r^2)u_N+v_N^{-1}\exp(-\lambda r)\big)+\rho(r)\big) \\
%%\varpi^{-2}\big(\varpi'(r)^{-2}\big(C(r)u_N+v_N^{-1}C'(r)^{-1}\big)+\rho(r)\big)
%%& \lesssim 
%\end{align*}
%holds, with $u_N,v_N \rightarrow 0$ and  $u_N \leq v_N$ $\rho(r) \rightarrow 0$ as $r \rightarrow \infty$.
We conclude
$$III \lesssim \varpi^{-2} \big(\varpi'(r)^{-2}C(r)u_N+\widetilde C(r)\big).$$
By Parseval's identity, the boundedness of $\overline{\mathcal F(\mathcal L \mu^N)}$ and Lemma \ref{lem: periodogram}, we also have 
\begin{align*}
%&\lesssim  \varpi^{-2}\sup_{t \in \mathrm{Supp}(w), |x| \leq r}\E_{\mathbb P^N}\big[\big|\widehat b_{\mathrm{GL}}^N(t,x)-b(x,\mu_t)\big|^2\big] + \varpi^{-2}\int_0^T \rho(dt)\int_{|x| \geq r}|b(x,\mu_t)|^2dx \\
IV \lesssim \varpi^{-2}\sup_{\xi \in \R^d} \E_{\PP^N}\big[|\mathcal F\big(\mathcal L(\mu^N-\mu)\big)(\xi)|^2\big] |F|_2^2 
%&+ \int_{|x| \geq r}
 \lesssim \varpi^{-2}N^{-1}.
%& \lesssim \varpi^{-2}\big(C(r)u_N+G(r)+N^{-1})
\end{align*}
 We finally turn to the term $II$. We have
\begin{align*}
\E_{\PP^N}\big[ II \big] & \lesssim  \E_{\PP^N}\big[\big|\mathcal F(F){\bf 1}_{\{|\mathcal F(\mathcal L(\mu^N-\mu)) |^2 \geq  \varpi\}}\big|_2^2\big] + \big|\mathcal F(F){\bf 1}_{\{|\mathcal F(\mathcal L\mu) |^2 \leq  2\varpi\}}\big|_2^2 \\
& \lesssim \varpi^{-2}\sup_{\xi \in \R^d} \E_{\PP^N}\big[|\mathcal F\big(\mathcal L(\mu^N-\mu)\big)(\xi)|^2\big] |F|_2^2 + H(\varpi)\\
& \lesssim \varpi^{-2}N^{-1}+H(\varpi),
\end{align*}
where $H(\varpi) \rightarrow 0$ as $\varpi \rightarrow 0$ by dominated convergence thanks to the property $|\mathcal F(\mathcal L\mu)(\xi)| >0$ almost everywhere of Assumption \ref{ass: fourier et derivation t}. 
We conclude
$$
\E_{\PP^N}\big[|\widehat{F}_{\varpi, \varpi'}^N- F|_2^2\big]  \lesssim \varpi^{-2}\big(\varpi'(r)^{-2}C(r)u_N+\widetilde C(r)\big)+N^{-1}\big)+H(\varpi).
$$
Let $r_N \rightarrow \infty$ slowly enough so that $\varpi'(r_N)^{-2}C(r_N)u_N \rightarrow 0$. This yields $v_N = C(r_N)u_N+\widetilde C(r_N)+N^{-1} \rightarrow 0$. Pick  now $\varpi_N \rightarrow 0$ slowly enough so that $\varpi_N^{-2}v_N \rightarrow 0$. The proof of Theorem \ref{thm: identif Vlasov} follows.

%for arbitrary $\varepsilon >0$, let  $\varpi$ be small enough so that $H(\varpi) \leq \varepsilon/4$. Let $r>0$ be large enough so that $\varpi^{-2}G(r) \leq \varepsilon/4$. It suffices now to pick $N$ large enough so that $\varpi^{-2}\big(C(r)u_N+N^{-1}\big) \leq \varepsilon/2$.
\section{Appendix} \label{sec: appendix} 

\subsubsection*{Characterisation of sub-Gaussian random variables}
We recall a classical definition of a sub-Gaussian random variable. Recommended reference is \cite{BK}.
\begin{defi} \label{def: sub gaussian} 
A real-valued random variable $Z$ such that $\E[Z]=0$ is $\lambda^2$ sub-Gaussian if 
%$$\E\big[\exp(zZ)\big] \leq \exp\big(\tfrac{1}{2}\gamma^2 z^2\big)\;\;\text{for every}\;\;z \in \R.$$
%Alternate 
one of the following conditions is satisfied, each statement implying the next:
\begin{enumerate}
\item[(i)] Laplace transform condition
$$\E\big[\exp(zZ)\big] \leq \exp\big(\tfrac{1}{2}\lambda^2 z^2\big)\;\;\text{for every}\;\;z \in \R.$$
\item[(ii)] Moment condition
$$
\E\big[Z^{2p}\big] \leq p! (4\lambda^2)^p\;\;\text{for every integer}\;\;p \geq 1. 
$$
\item[(iii)] Orlicz condition
$$
\E\big[\exp\big(\tfrac{1}{8\lambda^2}Z^2\big)\big] \leq 2.
$$
\item[(iv)] Laplace transform condition (bis)
$$\E\big[\exp(zZ)\big] \leq \exp\big(\tfrac{24}{2}\lambda^2 z^2\big)\;\;\text{for every}\;\;z \in \R.$$
%We have $(i) \Rightarrow (ii) \Rightarrow (iii)  \Rightarrow (iv)$
\end{enumerate}
\end{defi}
We will also use the following additive property of sub-Gaussian random variables: if the random variables $Z_i$ are independent and $\lambda_i^2$ sub-Gaussian, then $\rho(Z_1+Z_2)$ is $|\rho|^2(\lambda_1^2+\lambda_2^2) $ sub-Gaussian for every $\rho \in \R$. 

\subsection{Proof of Lemma \ref{lem moment class}} \label{proof lem exp}
By Assumption \ref{ass: basic lip}, the estimate 
$$|b(t,x,\mu_t)|\leq b_0+|b|_{\mathrm{Lip}}\big(|x|+\E_{\overline{\mathbb P}^N}\big[\big|X_t^i\big|\big]\big)$$ 
holds for every $(t,x)\in [0,T] \times \R^d$, where $ b_0 = \sup_{t \in [0,T]}|b(t,0,\delta_0)|$. Remember that 
$$
\overline{B}^i_t = \int_0^t c(s,X_s^i)^{-1/2}\big(dX_s^i-b(s,X_s^i,\mu_s) ds\big), \;\;1 \leq i \leq N,
$$
are  independent $d$-dimensional $\overline{\mathbb P}^{N}$-Brownian motions. By Minkowski's and Jensen's inequality, we have
\begin{align}
\nonumber \big|X_t^i \big| & \leq \big|X_0^i\big|+\int_0^t \big|b(t,X_s^i,\mu_s)\big|ds + \big|\int_0^t\sigma(s,X_s^i)d\overline{B}_{s}^i\big| \nonumber \\
\nonumber & \leq \big|X_0^i\big|+ b_0t+|b|_{\mathrm{Lip}}\int_0^t \big(\big|X_s^i\big|+ \E_{\overline{\mathbb P}^N}\big[\big|X_s^i\big|\big]\big)ds+ \big|\int_0^t\sigma(s,X_s^i)d\overline{B}_{s}^i\big| \nonumber \\
& \leq \big|X_0^i\big|+ b_0T+|b|_{\mathrm{Lip}}\int_0^t \big(\big|X_s^i\big|+ \E_{\overline{\mathbb P}^N}\big[\big|X_s^i\big|\big]\big)ds+\zeta_{T}^{i}, \label{eq: gronwall 1}
\end{align}
where $\zeta_{T}^{i} = \sup_{0 \leq t \leq T}\big|\int_0^t \sigma(s,X_s^i)d\overline{B}_{s}^i\big| $. Integrating w.r.t. $\overline{\mathbb P}^N$, we also have
%\begin{equation} \label{eq: gronwall 2}
$$
\E_{\overline{\mathbb P}^N}\big[\big|X_t^i \big|\big] \leq  \E_{\overline{\mathbb P}^N}\big[\big|X_0^i\big|\big]+ b_0T+2|b|_{\mathrm{Lip}}\int_0^t \E_{\overline{\mathbb P}^N}\big[\big|X_s^i\big|\big]ds+\E_{\overline{\mathbb P}^N}\big[\zeta_{T}^{i}\big].
$$
We infer by Gr\"onwall's lemma 
$$
\E_{\overline{\mathbb P}^N}\big[\big|X_t^i \big|\big] \leq \big(\E_{\overline{\mathbb P}^N}\big[\big|X_0^i\big|\big]+ b_0T+\E_{\overline{\mathbb P}^N}\big[\zeta_{T}^{i}\big]\big)\mathrm{e}^{2|b|_{\mathrm{Lip}}t}
$$
and plugging this estimate in \eqref{eq: gronwall 1} we infer
\begin{align*}
\big|X_t^i \big| \leq \big|X_0^i\big|+ b_0T+|b|_{\mathrm{Lip}}\int_0^t \big|X_s^i\big|ds+ 
\big(\E_{\overline{\mathbb P}^N}\big[\big|X_0^i\big|\big]+ b_0T+\E_{\overline{\mathbb P}^N}\big[\zeta_{T}^{i}\big]\big)\mathrm{e}^{2|b|_{\mathrm{Lip}}T}+\zeta_T^i.
\end{align*}
Applying Gr\"onwall's lemma again, we derive
\begin{align*}
\big|X_t^i \big| & \leq \big(\big|X_0^i\big|+ b_0T+ 
\big(\E_{\overline{\mathbb P}^N}\big[\big|X_0^i\big|\big]+ b_0T+\E_{\overline{\mathbb P}^N}\big[\zeta_{T}^{i}\big]\big)\mathrm{e}^{2|b|_{\mathrm{Lip}}T}+\zeta_T^i\big)\mathrm{e}^{|b|_{\mathrm{Lip}}t} \\
& \leq \big(\big|X_0^i\big|+\E_{\overline{\mathbb P}^N}\big[\big|X_0^i\big|\big]+2b_0T+\zeta_T^i+\E_{\overline{\mathbb P}^N}\big[\zeta_{T}^{i}\big]\big)\mathrm{e}^{3|b|_{\mathrm{Lip}}t}.
\end{align*}
Taking the exponent $2p$ and expectation w.r.t. $\overline{\mathbb P}^N$, we further obtain
\begin{align*}
\E_{\overline{\mathbb P}^N}\big[\big|X_t^i \big|^{2p}\big] & \leq 5^{2p-1}\big(2\E_{\overline{\mathbb P}^N}\big[\big|X_0^i\big|^{2p}\big]+(2b_0T)^p+2\E_{\overline{\mathbb P}^N}\big[\big(\zeta_{T}^{i})^{2p}\big]\big)\mathrm{e}^{3|b|_{\mathrm{Lip}}Tp} \\
& \leq C_6^p\big(\E_{\overline{\mathbb P}^N}\big[\big|X_0^i\big|^{2p}\big]+(b_0T)^p+\E_{\overline{\mathbb P}^N}\big[\big(\zeta_{T}^{i})^{2p}\big]\big)
\end{align*}
with $C_6=50\,\mathrm{e}^{3|b|_{\mathrm{Lip}}T}$. 
By Assumption \ref{ass: init condition}, the initial condition $|X_0^i|$ satisfies
$$\E_{\overline{\PP}^N}\big[\exp(\gamma_0|X_0^i|^2)\big] = 1 + \sum_{p \geq 1}\frac{\gamma_0^p}{p!}\E_{\overline{\PP}^N}\big[\big|X_0^i\big|^{2p}\big] \leq \gamma_1$$
hence for every $p \geq 1$, we obtain
$$
\E_{\overline{\mathbb P}^N}\big[\big|X_0^i\big|^{2p}\big] \leq p! \big(\tfrac{\gamma_1}{\gamma_0}\big)^p
$$
since $\gamma_1 \geq 1$.%\end{equation}
By Burkholder-Davis-Gundy's inequality with constant $(C^\star)^{p/2} p^{p/2}$ for some numerical constant $C^\star$, see {\it e.g.} Barlow and Yor \cite{BDG}, we also have
\begin{align*}
\E_{\overline{\mathbb P}^N}\big[\big(\zeta_{T}^{i})^{2p}\big] & \leq \Big(\frac{2p}{2p-1}\Big)^{2p}\,\E_{\overline{\mathbb P}^N}\Big[\Big|\int_0^T\sigma(t,X_t^i)d\overline{B}^i_t\Big|^{2p}\Big] \\
& \leq \Big(\frac{2p}{2p-1}\Big)^{2p}(2C^\star)^p p^{p}\E_{\overline{\mathbb P}^N}\Big[\big(\int_0^T\mathrm{Tr}\big(c(t,X_t^i)\big)dt\big)^p\Big] \\
& \leq  p^{p} (8C^\star T\big|\mathrm{Tr}(c)\big|_\infty)^p
 \leq p! \, \big(8C^\star \mathrm{e}T \big|\mathrm{Tr}(c)\big|_\infty\big)^p.
\end{align*}
Putting these estimates together, we conclude
\begin{align*}
\E_{\overline{\mathbb P}^N}\big[\big|X_t^i \big|^{2p}\big]  & \leq p! \,C_6^{p}\Big(\tfrac{\gamma_1}{\gamma_0}+T(b_0+8C^\star \mathrm{e}\big|\mathrm{Tr}(c)\big|_\infty)\Big)^p 
%& \leq  p!\,\big(9C_4^2(1+\tfrac{6}{\alpha}+8dT)\big)^p 
\end{align*}
and Lemma \ref{lem moment class} is established with 
$C_2 = C_6\big(\tfrac{\gamma_1}{\gamma_0}+T(b_0+8C^\star \mathrm{e}\big|\mathrm{Tr}(c)\big|_\infty)\big)$.
% of Lemma \ref{lem: exp moment estimate}.

\subsection{Proof of Lemma \ref{lem rosenthal}} \label{sec: proof of lemma rosen}
 Fix $\mathcal I_k  = \{i_1,\ldots, i_k\} \subset \{1,\ldots, N\}$.
 For $g:  [0,T] \times (\R^d)^{k} \times (\R^d)^\ell \rightarrow \R^d$, 
 we define 
 $$g_{\mathcal I_k}(t,y^\ell) = g(t,X_t^{i_1},\ldots, X_t^{i_k}, y^\ell).$$
For technical convenience, we establish a slightly stronger, replacing $\mathcal V_{2p}^N\big(f(t,\cdot)\big)$ in \eqref{eq: rosen} by
%\begin{equation} \label{eq: rosen}
$$
{\mathcal V}_{2p, \ell}^N\big(g_{\mathcal I_{k-\ell+1}}(t,\cdot)\big)  = \E_{\overline{\PP}^N}\Big[\big|\int_{(\R^d)^{\ell}}  g(t,X_t^{i_1},X_t^{i_2},\ldots, X_t^{i_{k-\ell+1}}, y^{\ell})(\mu_t^N-\mu_t)^{\otimes \ell}(dy^{\ell})\big|^{2p}\Big]
%\end{equation}
$$
for every $\mathcal I_{k-\ell+1} \subset \{1,\ldots, N\}$ with cardinality $k-\ell+1$ and every function $g:  [0,T] \times (\R^d)^{k-\ell+1} \times (\R^d)^\ell \rightarrow \R^d$, Lipschitz continuous in the space variables, that defines in turn a class $\mathcal G_{k-\ell+1,\ell}$. In particular $\mathcal V_{2p,\ell}^N\big(f(t,\cdot)\big)$ and ${\mathcal V}_{2p, \ell}^N\big(g_{\mathcal I_{k-\ell+1}}(t,\cdot)\big)$ agree for $\ell = k$ in which case the class $\mathcal G_{1,k}$ coincide with $\mathcal G_k$ and we obtain Lemma \ref{lem rosenthal}. We prove the result by induction.
 
\subsubsection*{Step 1:} The case $\ell = 1$. For $g\in \mathcal G_{k,1}$, $x^{k}\in (\R^d)^{k}$ and $\mathcal I \subset \{1,\ldots,N\}$, let
$$\Lambda_t^{\mathcal I}(g,x^{k}) = \int_{\R^d} g(t,x^{k},y)(\mu^{\mathcal J}_t- \mu_t)(dy),$$
where we write
$\mu^{\mathcal J}_t(dx) = |\mathcal J|^{-1}\sum_{i \in \mathcal J}\delta_{X_t ^i}(dx)$
for the empirical measure in restriction to $\mathcal I$. Observe that $\Lambda_t^{\mathcal I}(g,x^{k})$ is a sum of independent and centred random variables under $\overline{\mathbb P}^{N}$. 
We write
\begin{align*}
 \Lambda_t^{\{1,\ldots, N\}}(g, X_t^{i_1},\ldots, X_t^{i_k}) 
& = N^{-1}\sum_{i  \in \mathcal I_k}\Big(g(t,X_t^{i_1},\ldots, X_t^{i_k},X_t^{i})-\int_{\R^d}g(t, X_t^{i_1},\ldots, X_t^{i_k},y)\mu_t(dy)\Big)\\ 
&+\frac{N-k}{N} \Lambda_t^{\mathcal I_k^c}(g, X_t^{i_1},\ldots, X_t^{i_k}),
\end{align*}
since $|\mathcal I_k|=k$. We obtain the decomposition
$$\mathcal V_{2p,1}^N\big(g_{\mathcal I_k}(t,\cdot)\big)  = \E_{\overline{\PP}^N}\Big[\big| \Lambda_t^{\{1,\ldots, N\}}(g, X_t^{i_1},\ldots, X_t^{i_k}) \big|^{2p}\Big] \leq 2^{2p-1}(I+II),$$
with
\begin{align*}
I &=\frac{k^{2p-1}}{N^{2p}}\sum_{i \in \mathcal I_k}\Big(\E_{\overline{\PP}^N}\Big[\big|g(t,X_t^{i_1},\ldots, X_t^{i_k},X_t^{i})-\int_{\R^d}g(t, X_t^{i_1},\ldots, X_t^{i_k},y)\mu_t(dy)\big|^{2p}\Big]\Big),\\
II & = \Big(\frac{N-k}{N}\Big)^{2p}\E_{\overline{\PP}^N}\Big[\big|\Lambda_t^{\mathcal I_k^c}(g, X_t^{i_1},\ldots, X_t^{i_k})\big|^{2p}\Big].
\end{align*}
The term $I$ is controlled by the smoothness of $g$: 
\begin{align*}
I & \leq \frac{k^{2p-1}}{N^{2p}}|g(t,\cdot)|_{\mathrm{Lip}}^{2p} \sum_{i \in \mathcal I_k}\E_{\overline{\PP}^N}\Big[\int_{\R^d}|X_t^i-y|^{2p}\mu_t(dy)\Big] 
 \leq N^{-2p} p! \big(k^{2}4C_2\big)^p |g(t,\cdot)|_{\mathrm{Lip}}^{2p}, 
\end{align*}
where the last estimate stems from Lemma \ref{lem moment class}.
For the term $II$, writing $g = (g^1,\ldots, g^d)$ where the functions $g^j$ are real-valued, we further have 
\begin{equation} \label{eq: sec controle sub gauss}
II   \leq \Big(\frac{N-k}{N}\Big)^{2p} d^{2p-1}\sum_{j = 1}^d  \E_{\overline{\PP}^N}\big[\Lambda_t^{\mathcal I_k^c}(g^j, X_t^{i_1},\ldots, X_t^{i_k})^{2p}\big].
\end{equation}
%Conditional on $(X_t^{N}, X_t^{N-1}) = 
Moreover, for every $x\in \R^{d}$, the term
$$\Lambda_t^{\mathcal I_k^c}(g^j,x^{k}) = \frac{1}{N-k}\sum_{i \in \mathcal I_k^c}\big(g^j(t,x^{k},X_t^{i})- \E_{\overline{\PP}^{N}}\big[g^j(t,x^{k},X_t^{i})\big]\big)$$
is the sum of independent centred random variables that are independent of $(X_t^{i_1}, \ldots, X_t^{i_k})$ and 
%We plan to use a Rosenthal's type inequality: if $Z_1,\ldots, Z_{N-1}$ are independent centred random variables, we have 
%\begin{equation} \label{eq: Marcinkiewicz-Zygmund}
%\E\Big[\Big|\sum_{i = 1}^{N-1}Z_i\Big|^{2q}\Big] \leq (2q-1)^{q}\Big(\sum_{i = 1}^{N-1}\E\big[|Z_i|^{2q}\big]^{1/q}\Big)^q\;\;\text{for}\;\;q\geq 2,
%\end{equation}
%see for instance Rio \cite{Rio} for this version of Marcinkiewicz-Zygmund's inequality.
% and 
$$g^j(t,x^{k},X_t^{i})- \E_{\overline{\PP}^{N}}\big[g^j(t,x^{k},X_t^{i})\big]$$
is $\lambda^2$ sub-Gaussian with $\lambda^2 = 24C_2|g^j(t,\cdot)|_{\mathrm{Lip}}^{2}$
via the same estimate as for $I$ and the fact that (ii) implies (iv) in Definition \ref{def: sub gaussian}. Thanks to the additivity property of independent sub-Gaussian random variables, we further infer that
$\Lambda_t^{\mathcal I_k^c}(g^j,x^{k})$ is $\widetilde \lambda^2$ sub-Gaussian with
$$\widetilde \lambda^2 = \frac{1}{N-k}\lambda^2 = \frac{1}{N-k}24C_2|g^j(t,\cdot)|_{\mathrm{Lip}}^{2}.$$
Conditioning on $(X_t^{i_1}, \ldots, X_t^{i_k})$,
% we apply \eqref{eq: Marcinkiewicz-Zygmund} to $Z_i = f_{\ell,\ell'}^k(t,x,X_t^{i})- \E_{\overline{\PP}^{N}}\big[f_{\ell,\ell'}^k(t,x,X_t^{i})\big]$ and $q=p\ell$
we derive
$$  \E_{\overline{\PP}^N}\big[\Lambda_t^{\mathcal I_k^c}(g^j, X_t^{i_1},\ldots, X_t^{i_k})^{2p}\big] \leq   \frac{p! (96C_2)^p}{(N-k)^p} |g^j(t,\cdot)|_{\mathrm{Lip}}^{2p}$$
by (ii) of Definition \ref{def: sub gaussian}. Plugging this estimate in \eqref{eq: sec controle sub gauss}, we obtain
\begin{align*}
II & \leq  \frac{p! (96C_2d^2)^p}{(N-k)^p} |g(t,\cdot)|_{\mathrm{Lip}}^{2p} 
%& \leq (N-1)^{-p} p! (C^\star d^{3}8\ell)^p|f_{\ell,\ell'}(t,\cdot)|_\infty^{2p}
\end{align*}
and putting together our estimates for $I$ and $II$, we conclude
$$\mathcal V_{2p, 1}^N\big(g(t,\cdot)\big) \leq  \frac{p! K_1^p}{(N-k)^p} |g(t,\cdot)|_{\mathrm{Lip}}^{2p}$$
with $K_1 = 16(k^2+24d^2)C_2$. This establishes Lemma \ref{lem rosenthal} for $g$ in the case $\ell = 1$.
\subsubsection*{Step 2:} We assume that \eqref{eq: rosen} holds for $\mathcal V_{2p,\ell}^N\big(g_{{\mathcal I}_{k-\ell+1}}(t,\cdot)\big)$, for every $\mathcal I_{k-\ell+1} \subset \{1,\ldots, N\}$ with cardinality $k-\ell+1$ and every $g \in \mathcal G_{k-\ell+1,\ell}$ with $\ell < k$. 
%If $1 \leq \ell' \leq \ell < k$ and 
%$$g_{\ell+1}: [0,T]\times (\R^d)^{k+1} \times (\R^d)^{\ell+1} \rightarrow \R^d,$$
%let
%\begin{equation} \label{eq: tensor notation}
%g_{\ell+1, \ell'}(t, x^{k+1}, y^{\ell+1-\ell'}) 
%= g_{\ell+1}\big(t,x^{k+1},(X_{t}^{N-1},\ldots, X_t^{N-\ell'}, y^{\ell+1-\ell'})\big),
%\end{equation}
%viewed as a function 
%$$g_{\ell+1, \ell'}:[0,T] \times (\R)^{k+1} \times (\R^d)^{\ell+1-\ell'} \rightarrow \R^d$$
%that depends on the random vector $(X_{t}^{N-1},\ldots, X_t^{N-\ell'})$.
Let $g \in {\mathcal G}_{k-\ell,\ell+1}$ and ${\mathcal I}_{k-\ell} \subset \{1,\ldots, N\}$. We have:
\begin{align*}
\mathcal V_{2p,\ell+1}^N\big(g_{{\mathcal I}_{k-\ell}}(t,\cdot)\big) & = \E_{\overline{\PP}^{N}}\Big[\big|\int_{(\R^d)^{\ell+1}}  g_{\mathcal I_{k-\ell}}(t,y^{\ell+1})(\mu_t^N-\mu_t)^{\otimes (\ell+1)}(dy^{\ell+1})\big|^{2p}\Big] \\
& \leq 2^{2p-1}(III + IV),
\end{align*}
with

\begin{align*}
III & =  N^{-1}\sum_{i = 1}^N\E_{\overline{\PP}^{N}}\Big[\big|\int_{(\R^d)^{\ell}}  g\big(t,X_t^{i_1},\ldots, X_{t}^{i_{k-\ell}}, (X_t^i,y^\ell)\big)(\mu_t^N-\mu_t)^{\otimes \ell}(dy^{\ell}) \big|^{2p}\Big], \\
IV & = \int_{\R^d}\E_{\overline{\PP}^{N}}\Big[\big|\int_{(\R^d)^{\ell}}  g\big(t,X_t^{i_1},\ldots, X_{t}^{i_{k-\ell}}, (y,y^\ell)\big)(\mu_t^N-\mu_t)^{\otimes \ell}(dy^{\ell}) \big|^{2p}\Big]\mu_t(dy). 
%& =  N^{-1}\mathcal V_{2p}^N\big(g_\ell'\big) +\tfrac{N-1}{N}\E_{\overline{\PP}^{N}}\Big[\big|\int_{(\R^d)^{\ell}}  g_{\ell+1}\big(t,X_t^N,\ldots, X_{t}^{N-k}, (X_t^{N-1},y^\ell)\big)(\mu_t^N-\mu_t)^{\otimes \ell}(dy^{\ell}) \big|^{2p}\Big] \\
%&\hspace{3mm}+   \int_{\R^d}\mathcal V_{2p}^N\big(g_{\ell}''(y)\big)\mu_t(dy),
\end{align*}
Let $i_0 \in \mathcal I_{k-\ell}^c$ and put $\mathcal I_{k-\ell+1} = \mathcal I_{k-\ell} \cup \{i_0\}$. The term $IV$ can be rewritten as
$$IV = \int_{\R^d}\mathcal V^N_{2p,\ell}\big(g_{\mathcal I_{k-\ell+1}}'(t,\cdot)(y)\big)\mu_t(dy),$$
where, for fixed $y \in \R^d$, the function $g'(t,x_{i_1}, \ldots x_{i_{k-\ell}}, x_{i_0},y^\ell)(y) = g(t,x_{i_1}, \ldots x_{i_{k-\ell}},(y,y^\ell))$ with the artificial variable $x_{i_0}$ belongs to $\mathcal G_{k-\ell+1,\ell}$. By the induction hypothesis and noting that $\sup_{y \in \R^d}|g'(t,\cdot,y)|_{\mathrm{Lip}} \leq |g(t,\cdot)|_{\mathrm{Lip}}$, we infer
$$IV \leq \frac{p! K_\ell^p}{(N-k)^p}|g(t,\cdot)|_{\mathrm{Lip}}^{2p}.$$
We split the sum in $III$ over indices in $\mathcal I_{k-\ell}$ and $\mathcal I_{k-\ell}^c$. If $i \in \mathcal I_{k-\ell}$, in the same way as for $IV$, we write 
$$g\big(t,X_t^{i_1},\ldots, X_{t}^{i_{k-\ell}}, (X_t^i,y^\ell)\big) = g''_{\mathcal I_{k-\ell+1}}(t,y^\ell)$$
with $\mathcal I_{k-\ell+1} = \mathcal I_{k-\ell} \cup \{i_0\}$ for some arbitrary $i_0 \in \mathcal I_{k-\ell}^c$ and 
with $g''(t,x_{i_1},\ldots, x_{i_{k-\ell}},x_{i_0},y^\ell) =g\big(t,x_{i_1},\ldots, x_{i_{k-\ell}},(x_i,y^\ell)\big)$, where $i$ coincides with one of the $i_j \in \mathcal I_{k-\ell}$. Also, $g''$ belongs to $\mathcal G_{k-\ell+1,\ell}$. If $i \in \mathcal I_{k-\ell}^c$, we write
$$g\big(t,X_t^{i_1},\ldots, X_{t}^{i_{k-\ell}}, (X_t^i,y^\ell)\big) = g'''_{ \mathcal I_{k-\ell} \cup \{i\}}(t,y^\ell)$$
with $g'''(t,x_{i_1},\ldots, x_{i_{k-\ell}}, x_i,y^\ell) =g\big(t,x_{i_1},\ldots, x_{i_{k-\ell}},(x_i,y^\ell)\big)$ and $g'''$ belongs to $\mathcal G_{k-\ell+1,\ell}$ as well. 
We infer
\begin{align*}
III & \leq (k-\ell)N^{-1}\mathcal V^N_{2p,\ell}\big(g''_{\mathcal I_{k-\ell+1}}(t,\cdot)\big)+ N^{-1}\sum_{i \in \mathcal I_{k-\ell}^c}\mathcal V^N_{2p,\ell}\big(g'''_{ \mathcal I_{k-\ell} \cup \{i\}}(t,\cdot)\big)  \leq \frac{p! K_\ell^p}{(N-k)^p}|g(t,\cdot)|_{\mathrm{Lip}}^{2p}
\end{align*}
by the induction hypothesis and noting again that $|g''(t,\cdot)|_{\mathrm{Lip}}$ and $|g'''(t,\cdot)|_{\mathrm{Lip}}$ are controlled by $|g(t,\cdot)|_{\mathrm{Lip}}$. We conclude
$$
\mathcal V_{2p,\ell+1}^N\big(g_{{\mathcal I}_{k-\ell}}(t,\cdot)\big)  \leq 2^{2p}\frac{p! K_\ell^p}{(N-k)^p}|g(t,\cdot)|_{\mathrm{Lip}}^{2p} 
 =\frac{p! K_{\ell+1}^p}{(N-k)^p}|g(t,\cdot)|_{\mathrm{Lip}}^{2p}
$$
with $K_{\ell+1} =  4 K_\ell$. The proof of Lemma \ref{lem rosenthal} is complete.

\subsection{(Sketch of) proof of Proposition \ref{prop : reg Holder}} \label{proof : reg Holder}

\noindent {\it Step 1:} Thanks to Chapters 6 and 9 of \cite{BoKry}, since $F, G$ are bounded and $\mu_0$ satisfies Assumption \ref{ass: init condition}, it can be shown that \eqref{eq: mckv approx} admits a unique probability solution $\mu$ in the sense of \cite{BoKry}, absolutely continuous w.r.t. the Lebesgue measure, that we still denote $\mu(t,x) = \mu_t(x)$. Moreover $\mu \in \mathcal{H}_{\text{loc}}^{\delta/2,\delta}  = \cap_{(t_0,x,_0)\in (0,T)\times \R^d} \mathcal H^{\delta/2,\delta}(t_0,x_0)$ for every $0 < \delta < 1$. The main arguments of these properties rely on the existence of a suitable Lyapunov function associated to \eqref{eq: mckv approx}, following the terminology of \cite{ManitaShaposh14} and \cite{BoKry} (for instance $x \mapsto 1+ |x|^2$) together with Sobolev embeddings.\\

\noindent {\it Step 2:}  Define 
$$\widetilde{a}_k(t,x)=G^k(x)+\int_{\R^d}F^k(x-y)\mu_t(y)dy,\;\;k=1,\ldots,d,$$ 
and
$$\widetilde{a}(t,x) = \text{div} (G(x) +  \int_{\R^d}F(x-y)\mu_t(y)dy),$$ 
which are well defined since $\beta,\beta' >1$. Consider next the Cauchy problem associated to \eqref{eq: mckv approx} in its strong form:
\begin{equation}  \label{eq: mckv approx bis}
\left\{ 
\begin{array}{ll}
\partial_{t}\widetilde{\mu}_t = \frac{1}{2}\sigma^2\Delta \widetilde{\mu_t} - \sum_{k=1}^d \widetilde{a}_k(t,\cdot) \partial_{k}\widetilde{\mu_t} - \widetilde{a}(t,\cdot)\widetilde{\mu}_t \\ 
\tilde{\mu}_{t = 0} = \mu_0.
\end{array}
\right. 
\end{equation} 
Taking $\delta =\beta-\lfloor \beta \rfloor $ we obtain $\widetilde{a}_i, \widetilde{a} \in \mathcal{C}^{(\beta-\lfloor \beta \rfloor)/2,\beta-\lfloor \beta \rfloor}_{\text{loc}}$ by Step 1.\\ 

\noindent {\it Step 3:} Using $\text{inf} \; \widetilde{a} > -\infty$ and the existence of a Lyapunov function associated to the problem, by Theorem 2.3 of \cite{AngiuliLorenzi}, there exists a unique solution $\widetilde \mu$ of \eqref{eq: mckv approx bis}. Moreover, $\widetilde \mu$ is continuous on $(0,T) \times \R^d$ and
$$\widetilde \mu \in \mathcal{C}^{1+(\beta-\lfloor \beta \rfloor)/2,2+ \beta-\lfloor \beta \rfloor }_{\text{loc}}.$$
%  of \eqref{eq: mckv approx bis}.\\
It is also the unique solution defined in Theorem 12 of Chapter 1 of \cite{friedman}, therefore the unique integrable solution of the problem \eqref{eq: mckv approx}. By uniqueness, $\mu = \widetilde{\mu}$.\\ 

\noindent {\it Step 4:} If $\lfloor \beta \rfloor = 1 $, we obtain $\mu \in \mathcal{H}^{(1+\beta)/2, 1+\beta}(t_0,x_0)$ for every $(t_0,x_0)\in (0,T)\times \R^d$. Otherwise, we can iterate the process thanks to results of  Section 8.12 in \cite{KrylovHolder}: successively:
\begin{itemize}
\item Since $\partial_{x_{k'}}\widetilde{a}_k$ and $\partial_{x_k} \widetilde{a}$ are in $\mathcal{C}^{(\beta-\lfloor \beta \rfloor )/2,\beta-\lfloor \beta \rfloor}_{\text{loc}}$, we have
$$\partial_{x_k} \mu \in \mathcal{C}^{1+(\beta-\lfloor \beta \rfloor)/2,2+ \beta-\lfloor \beta \rfloor }_{\text{loc}}.$$ 
\item Since $\partial_{t}\widetilde{a}_k$ and $\partial_{t} \widetilde{a}$ are now in $\mathcal{C}^{(\beta-\lfloor \beta\rfloor)/2,\beta-\lfloor \beta \rfloor}_{\text{loc}}$, we have
$$\partial_{t} \mu \in \mathcal{C}^{1+(\beta-\lfloor \beta \rfloor)/2,2+ \beta-\lfloor \beta \rfloor }_{\text{loc}}.$$
\end{itemize}
Therefore, if $\lfloor \beta \rfloor = 2 $, we obtain $\mu \in \mathcal{H}^{(1+\beta)/2, 1+\beta}(t_0,x_0)$ for every $(t_0,x_0)\in (0,T)\times \R^d$. Otherwise, we can iterate again the process and so on.  The result follows.
%Finally, we obtain for every $\beta_1 >1$, for every $(t_0,x_0)\in (0,T)\times \R^d$, $\mu \in \mathcal{H}^{\tfrac{1+\beta_1}{2}, 1+\beta_1}(t_0,x_0)$. 
%The proof of 
%Proposition \ref{prop : reg Holder} is complete.

{\small \subsection*{Acknowledgements} {\it Informal discussions with colleagues at CEREMADE are gratefully acknowledged; we thank in particular, Pierre Cardaliaguet, Djalil Chafa\"i and St\'ephane Mischler. We also thank Denis Belomestny and Nicolas Fournier for insightful comments. This work partially answers a problem that was posed to us by Sylvie M\'el\'eard almost two decades ago (at a time we did not have the proper tools to address it!).}
}

\bibliographystyle{plain}       % (uses file "plain.bst")
%\nocite{*}
\bibliography{biblioDeH.bib}

\begin{thebibliography}{10}

\bibitem{abraham2019statistical}
Kweku Abraham and Richard Nickl.
\newblock On statistical {C}ald{\'e}ron problems.
\newblock {\em Mathematical Statistics and Machine Learning. To appear. arXiv
  preprint arXiv:1906.03486}, 2019.

\bibitem{AngiuliLorenzi}
Luciana Angiuli and Luca Lorenzi.
\newblock Compactness and invariance properties of evolution operators
  associated with {K}olmogorov operators with unbounded coefficients.
\newblock {\em Journal of mathematical analysis and applications},
  379(1):125--149, 2011.

\bibitem{Baladron}
Javier Baladron, Diego Fasoli, Olivier Faugeras, and Jonathan Touboul.
\newblock Mean-field description and propagation of chaos in networks of
  {H}odgkin-{H}uxley and {F}itz{H}ugh-{N}agumo neurons.
\newblock {\em Journal of Mathematical Neurosciences}, pages 2:Art 10,50, 2012.

\bibitem{BDG}
Martin~T. Barlow and Marc Yor.
\newblock Semimartingale inequalities via the {G}arsia-{R}odemich-{R}umsey
  lemma, and applications to local times.
\newblock {\em J. Functional Analysis}, 49(2):198--229, 1982.

\bibitem{benachour1998nonlinear}
Said Benachour, Bernard Roynette, Denis Talay, and Pierre Vallois.
\newblock Nonlinear self-stabilizing processes--i existence, invariant
  probability, propagation of chaos.
\newblock {\em Stochastic processes and their applications}, 75(2):173--201,
  1998.

\bibitem{birge1998minimum}
Lucien Birg{\'e} and Pascal Massart.
\newblock Minimum contrast estimators on sieves: exponential bounds and rates
  of convergence.
\newblock {\em Bernoulli}, 4(3):329--375, 1998.

\bibitem{BoKry}
Vladimir Bogachev, Nicolai Krylov, Michael R\"ockner, and Stanislav
  Shaposhnikov.
\newblock {\em {F}okker-{P}lanck-{K}olmogorov equations}.
\newblock Mathematical Survey and Monographs. 2015.

\bibitem{bolley2007quantitative}
Fran{\c{c}}ois Bolley, Arnaud Guillin, and C{\'e}dric Villani.
\newblock Quantitative concentration inequalities for empirical measures on
  non-compact spaces.
\newblock {\em Probability Theory and Related Fields}, 137(3-4):541--593, 2007.

\bibitem{BHJ}
Alexandre Boumezoued, Marc Hoffmann, and Paulien Jeunesse.
\newblock Nonparametric adaptive inference of birth and death processes in a
  large population limit.
\newblock {\em Mathematical Statistics and Machine Learning. To appear.
  arXiv:1903.00673.}, 2019.

\bibitem{BK}
Valery~V. Buldygin and Juriy.~V. Koza\v{c}enko.
\newblock Sub-{G}aussian random variables.
\newblock {\em Ukrain. Mat. Zh.}, 32(6):723--730, 1980.

\bibitem{BURGER}
Martin Burger, Vincezo Capasso, and Daniela Morale.
\newblock On an aggregation model with long and short range interactions.
\newblock {\em Nonlinear Analysis and Real World Applications}, 8(3):939--958,
  2007.

\bibitem{Sznitman2}
Donald~L. Burkholder, \'{E}tienne Pardoux, and Alain-Sol Sznitman.
\newblock {\em \'{E}cole d'\'{E}t\'{e} de {P}robabilit\'{e}s de {S}aint-{F}lour
  {XIX}---1989}, volume 1464 of {\em Lecture Notes in Mathematics}.
\newblock Springer-Verlag, Berlin, 1991.
\newblock Papers from the school held in Saint-Flour, August 16--September 2,
  1989, Edited by P. L. Hennequin.

\bibitem{canuto12}
Claudio Canuto, Fabio Fagnani, and Paolo Tilli.
\newblock An {E}ulerian approach to the analysis of {K}rause's consensus
  models.
\newblock {\em SIAM Journal on Control and Optimization}, 50(1):243--265, 2012.

\bibitem{CARDA}
Pierre Cardaliaguet, Fran\c{c}ois Delarue, Jean-Michel Lasry, and Pierre-Louis
  Lions.
\newblock {\em The master equation and the convergence problem in mean field
  games}, volume 201 of {\em Annals of Mathematics Studies}.
\newblock Princeton University Press, Princeton, NJ, 2019.

\bibitem{LC19}
Pierre Cardaliaguet and Charles Lehalle.
\newblock Mean field game of controls and an application to trade crowding.
\newblock {\em Mathematics and Financial Economics}, 12(3):335--363, 2019.

\bibitem{carmona2018probabilistic}
Ren{\'e} Carmona, Fran{\c{c}}ois Delarue, et~al.
\newblock {\em Probabilistic Theory of Mean Field Games with Applications
  I-II}.
\newblock Springer, 2018.

\bibitem{CATTIAUXetal}
Patrick Cattiaux, Arnaud Guillin, and Florent Malrieu.
\newblock Probabilistic approach for granular media equations in the
  non-uniformly convex case.
\newblock {\em Probab. Theory Related Fields}, 140(1-2):19--40, 2008.

\bibitem{Chassagneux}
Jean-Fran\c{c}ois Chassagneux, Lukasz Szpruch, and Alvin Tse.
\newblock Weak quantitative propagation of chaos via differential calculus on
  the space of measures.
\newblock {\em arXiv:1901.02556v1}, 2019.

\bibitem{CHAZELLE}
Bernard Chazelle, Quansen Jiu, Qianxiao Li, and Chu Wang.
\newblock Well-posedness of the limiting equation of a noisy consensus model in
  opinion dynamics.
\newblock {\em Journal of Differential Equations}, 263(1):365 -- 397, 2017.

\bibitem{coghi2018pathwise}
Michele Coghi, Jean-Dominique Deuschel, Peter Friz, and Mario Maurelli.
\newblock Pathwise {M}c{K}ean-{V}lasov theory with additive noise.
\newblock {\em arXiv preprint arXiv:1812.11773}, 2018.

\bibitem{comte2019nonparametric}
Fabienne Comte and Valentine Genon-Catalot.
\newblock Nonparametric drift estimation for iid paths of stochastic
  differential equations.
\newblock {\em Annals of Statistics, to appear}, 2019.

\bibitem{doumic2015statistical}
Marie Doumic, Marc Hoffmann, Nathalie Krell, Lydia Robert, et~al.
\newblock Statistical estimation of a growth-fragmentation model observed on a
  genealogical tree.
\newblock {\em Bernoulli}, 21(3):1760--1799, 2015.

\bibitem{doumic2012nonparametric}
Marie Doumic, Marc Hoffmann, Patricia Reynaud-Bouret, and Vincent Rivoirard.
\newblock Nonparametric estimation of the division rate of a size-structured
  population.
\newblock {\em SIAM Journal on Numerical Analysis}, 50(2):925--950, 2012.

\bibitem{BEGOMELEARD}
Bego\~{n}a Fernandez and Sylvie M\'{e}l\'{e}ard.
\newblock A {H}ilbertian approach for fluctuations on the {M}c{K}ean-{V}lasov
  model.
\newblock {\em Stochastic Process. Appl.}, 71(1):33--53, 1997.

\bibitem{fouque2013systemic}
Jean~Pierre Fouque and Li-Hsien Sun.
\newblock Systemic risk illustrated.
\newblock {\em Handbook on Systemic Risk, Eds J.P Fouque and J Langsam}, 2013.

\bibitem{FournierGuillin}
Nicolas Fournier and Arnaud Guillin.
\newblock On the rate of convergence in {W}asserstein distance of the empirical
  measure.
\newblock {\em Probab. Theory Related Fields}, 162(3-4):707--738, 2015.

\bibitem{friedman}
Avner Friedman.
\newblock {\em Partial differential equations of parabolic type}.
\newblock Courier Dover Publications, 2008.

\bibitem{Ga88}
J\"{u}rgen G\"{a}rtner.
\newblock On the {M}c{K}ean-{V}lasov limit for interacting diffusions.
\newblock {\em Math. Nachr.}, 137:197--248, 1988.

\bibitem{GCJ94}
Valentine Genon-Catalot and Jean Jacod.
\newblock Estimation of the diffusion coefficient for diffusion processes:
  random sampling.
\newblock {\em Scand. J. Statist.}, 21(3):193--221, 1994.

\bibitem{giesecke2020inference}
Kay Giesecke, Gustavo Schwenkler, and Justin~A Sirignano.
\newblock Inference for large financial systems.
\newblock {\em Mathematical Finance}, 30(1):3--46, 2020.

\bibitem{GL08}
Alexander Goldenshluger and Oleg Lepski.
\newblock Universal pointwise selection rule in multivariate function
  estimation.
\newblock {\em Bernoulli}, 14(4):1150--1190, 2008.

\bibitem{GL11}
Alexander Goldenshluger and Oleg Lepski.
\newblock Bandwidth selection in kernel density estimation: oracle inequalities
  and adaptive minimax optimality.
\newblock {\em Ann. Statist.}, 39(3):1608--1632, 2011.

\bibitem{GL14}
Alexander Goldenshluger and Oleg Lepski.
\newblock On adaptive minimax density estimation on {$R^d$}.
\newblock {\em Probab. Theory Related Fields}, 159(3-4):479--543, 2014.

\bibitem{herrmann2008large}
Samuel Herrmann, Peter Imkeller, Dierk Peithmann, et~al.
\newblock Large deviations and a {K}ramer type law for self-stabilizing
  diffusions.
\newblock {\em The Annals of Applied Probability}, 18(4):1379--1423, 2008.

\bibitem{ngoc2019nonparametric}
Van~Ha Hoang, Thanh~Mai Pham~Ngoc, Vincent Rivoirard, and Viet~Chi Tran.
\newblock {Nonparametric estimation of the fragmentation kernel based on a PDE
  stationary distribution approximation}.
\newblock arXiv preprint arXiv:1710.09172v3, March 2019.

\bibitem{HO}
Marc Hoffmann and Ad\'ela{\"\i}de Olivier.
\newblock Nonparametric estimation of the division rate of an age dependent
  branching process.
\newblock {\em Stochastic Process. Appl.}, 126(5):1433--1471, 2016.

\bibitem{JJ09}
Jan Johannes.
\newblock Deconvolution with unknown error distribution.
\newblock {\em Ann. Statist.}, 37(5A):2301--2323, 2009.

\bibitem{MJ98}
Benjamin Jourdain and Sylvie M\'{e}l\'{e}ard.
\newblock Propagation of chaos and fluctuations for a moderate model with
  smooth initial data.
\newblock {\em Ann. Inst. H. Poincar\'{e} Probab. Statist.}, 34(6):727--766,
  1998.

\bibitem{jourdain2020central}
Benjamin Jourdain and Alvin Tse.
\newblock Central limit theorem over non-linear functionals of empirical
  measures with applications to the mean-field fluctuation of interacting
  particle systems.
\newblock {\em arXiv preprint arXiv:2002.01458}, 2020.

\bibitem{KS}
Ioannis Karatzas and Steven~E Shreve.
\newblock Brownian motion.
\newblock In {\em Brownian Motion and Stochastic Calculus}, pages 47--127.
  Springer, 1998.

\bibitem{kasonga1990maximum}
Raphael~A. Kasonga.
\newblock Maximum likelihood theory for large interacting systems.
\newblock {\em SIAM Journal on Applied Mathematics}, 50(3):865--875, 1990.

\bibitem{KrylovHolder}
Nikola{\u\i}~V. Krylov.
\newblock {\em Lectures on elliptic and parabolic equations in Holder spaces}.
\newblock Number~12. American Mathematical Soc., 1996.

\bibitem{lacker2018mean}
Daniel Lacker.
\newblock Mean field games and interacting particle systems.
\newblock {\em Preprint}, 2018.

\bibitem{Lacker}
Daniel Lacker.
\newblock On a strong form of propagation of chaos for {M}c{K}ean-{V}lasov
  equations.
\newblock {\em Electron. Commun. Probab.}, 23:Paper No. 45, 11, 2018.

\bibitem{lacour2017estimator}
Claire Lacour, Pascal Massart, and Vincent Rivoirard.
\newblock Estimator selection: a new method with applications to kernel density
  estimation.
\newblock {\em Sankhya A}, 79(2):298--335, 2017.

\bibitem{LECAM}
Lucien Le~Cam.
\newblock {\em Asymptotic methods in statistical decision theory}.
\newblock Springer Series in Statistics. Springer-Verlag, New York, 1986.

\bibitem{LEPSKI}
Oleg~V. Lepski\u\i.
\newblock A problem of adaptive estimation in {G}aussian white noise.
\newblock {\em Teor. Veroyatnost. i Primenen.}, 35(3):459--470, 1990.

\bibitem{LOW}
Mark~G. Low.
\newblock Nonexistence of an adaptive estimator for the value of an unknown
  probability density.
\newblock {\em Ann. Statist.}, 20(1):598--602, 1992.

\bibitem{maida2020statistical}
Myl{\`e}ne Ma{\"\i}da, Tien~Dat Nguyen, Thanh Mai~Pham Ngoc, Vincent Rivoirard,
  and Viet~Chi Tran.
\newblock Statistical deconvolution of the free {F}okker-{P}lanck equation at
  fixed time.
\newblock {\em arXiv preprint arXiv:2006.11899}, 2020.

\bibitem{MALRIEU}
Florent Malrieu.
\newblock Logarithmic {S}obolev inequalities for some nonlinear {PDE}'s.
\newblock {\em Stochastic Process. Appl.}, 95(1):109--132, 2001.

\bibitem{ManitaShaposh14}
Oxana Manita and Stanislav Shaposhnikov.
\newblock Nonlinear parabolic equations for measures.
\newblock {\em St. Petersburg Mathematical Journal}, 25(1):43--62, 2014.

\bibitem{massart2007concentration}
Pascal Massart.
\newblock {\em Concentration inequalities and model selection}, volume~6.
\newblock Springer, 2007.

\bibitem{MCKEAN}
Henry~P. McKean~Jr.
\newblock A class of {M}arkov processes associated with nonlinear parabolic
  equations.
\newblock {\em Proceedings of the National Academy of Sciences of the United
  States of America}, 56(6):1907, 1966.

\bibitem{meleard1996asymptotic}
Sylvie M{\'e}l{\'e}ard.
\newblock Asymptotic behaviour of some interacting particle systems;
  {M}c{K}ean-{V}lasov and {B}oltzmann models.
\newblock In {\em Probabilistic models for nonlinear partial differential
  equations}, pages 42--95. Springer, 1996.

\bibitem{mogilner99}
Alexander Mogilner and Leah Edelstein-Keshet.
\newblock A non-local model for a swarm.
\newblock {\em Journal of Mathematical Biology}, 38(6):534--570, 1999.

\bibitem{monard2019consistent}
Fran{\c{c}}ois Monard, Richard Nickl, and Gabriel~P Paternain.
\newblock Consistent inversion of noisy non-{A}belian {X}-ray transforms.
\newblock {\em arXiv preprint arXiv:1905.00860}, 2019.

\bibitem{NADARAYA}
\`Elizbar~A. Nadaraja.
\newblock On a regression estimate.
\newblock {\em Teor. Verojatnost. i Primenen.}, 9:157--159, 1964.

\bibitem{nickl2017bernstein}
Richard Nickl.
\newblock {B}ernstein-von {M}ises theorems for statistical inverse problems i:
  {S}chr{\"o}dinger equation.
\newblock {\em arXiv preprint arXiv:1707.01764}, 2017.

\bibitem{nickl2017bayesian}
Richard Nickl.
\newblock On {B}ayesian inference for some statistical inverse problems with
  partial differential equations.
\newblock {\em Bernoulli News}, 24(2):5--9, 2017.

\bibitem{MR779460}
Karl Oelschl\"{a}ger.
\newblock A law of large numbers for moderately interacting diffusion
  processes.
\newblock {\em Z. Wahrsch. Verw. Gebiete}, 69(2):279--322, 1985.

\bibitem{RY99}
Daniel Revuz and Marc Yor.
\newblock {\em Continuous martingales and {B}rownian motion}, volume 293 of
  {\em Grundlehren der Mathematischen Wissenschaften [Fundamental Principles of
  Mathematical Sciences]}.
\newblock Springer-Verlag, Berlin, third edition, 1999.

\bibitem{SV79}
Daniel~W. Stroock and S.~R.~Srinivasa Varadhan.
\newblock {\em Multidimensional diffusion processes}, volume 233 of {\em
  Grundlehren der Mathematischen Wissenschaften [Fundamental Principles of
  Mathematical Sciences]}.
\newblock Springer-Verlag, Berlin-New York, 1979.

\bibitem{SZNITMAN}
Alain-Sol Sznitman.
\newblock Nonlinear reflecting diffusion process, and the propagation of chaos
  and fluctuations associated.
\newblock {\em J. Funct. Anal.}, 56(3):311--336, 1984.

\bibitem{TANAKAHITSUDA}
Hiroshi Tanaka and Masuyuki Hitsuda.
\newblock Central limit theorem for a simple diffusion model of interacting
  particles.
\newblock {\em Hiroshima Math. J.}, 11(2):415--423, 1981.

\end{thebibliography}

\end{document}